\documentclass[psamsfonts, reqno, 11pt]{amsart}

\usepackage{amsmath}
\usepackage{amssymb}
\usepackage{amsthm}
\usepackage[alphabetic]{amsrefs}
\usepackage{amscd}
\usepackage[margin=1in]{geometry}
\usepackage{fancyhdr}
\usepackage{graphicx}
\usepackage{float}
\usepackage{todonotes}
\usepackage{setspace}
\usepackage[all,arc]{xy}
\usepackage{enumerate}
\usepackage{mathrsfs}
\usepackage{color}
\usepackage{tikz-cd}
\usepackage{dsfont} %used for blackboard bold numerals
\usepackage{esint}
\usepackage{xcolor}
\usepackage[hidelinks]{hyperref}% cross-referencing
\hypersetup{
	colorlinks,
	linkcolor={red!50!black},
	citecolor={blue!50!black},
	urlcolor={blue!80!black}
}

\makeatletter
\let\@evenhead\@empty
\let\@oddhead\@empty
\makeatother
\newtheorem{thm}{Theorem}[section]
\newtheorem{cor}[thm]{Corollary}
\newtheorem{prop}[thm]{Proposition}
\newtheorem{lem}[thm]{Lemma}

\newtheorem*{thm*}{Theorem}
\newtheorem*{prop*}{Proposition}
\newtheorem*{lem*}{Lemma}
\newtheorem*{cor*}{Corollary}
\newtheorem*{conj*}{Conjecture}
\newtheorem*{mainthm*}{Main Theorem}

\newtheorem{lem'}{Lemma}
\newtheorem{thm'}{Theorem}
\newtheorem{prop'}{Proposition}

\theoremstyle{definition}

\newtheorem*{defn*}{Definition}

\newtheorem*{exmp*}{Example}

\newtheorem*{exer*}{Exercise}

\theoremstyle{remark}
\newtheorem{rem}[thm]{Remark}

\newtheorem*{claim*}{Claim}

\newtheorem*{rem*}{Remark}
\newtheorem{rem'}{Remark}

\newtheorem*{qtn'}{Question}

\newtheorem*{problem*}{Problem}
\newtheorem*{soln'}{Solution}

\DeclareMathOperator{\re}{Re}
\DeclareMathOperator{\im}{Im}

\DeclareMathOperator{\Hom}{Hom}

\DeclareMathOperator{\Sym}{Sym}
\DeclareMathOperator{\GL}{GL}
\DeclareMathOperator{\SL}{SL}
\DeclareMathOperator{\PGL}{PGL}
\DeclareMathOperator{\PSL}{PSL}
\DeclareMathOperator{\PSO}{PSO}
\DeclareMathOperator{\PSU}{PSU}

\DeclareMathOperator{\Ad}{Ad}

\DeclareMathOperator{\LHS}{LHS}
\DeclareMathOperator{\RHS}{RHS}

\DeclareMathOperator{\Arg}{Arg}

\DeclareMathOperator{\Ai}{Ai}

\newcommand{\true}{\textnormal{true}}
\newcommand{\Abs}{\textnormal{abs}}

\newcommand{\Mod}{\textnormal{mod}}

\newcommand{\cusp}{\textnormal{cusp}}

\newcommand{\R}{\mathbf{R}}
\newcommand{\C}{\mathbf{C}}
\newcommand{\Z}{\mathbf{Z}}
\newcommand{\Q}{\mathbf{Q}}

\renewcommand{\1}{\mathds{1}}

\renewcommand{\H}{\mathbf{H}}
\renewcommand{\S}{\mathcal{S}}
\newcommand{\D}{\mathcal{D}}
\newcommand{\A}{\mathbf{A}}

\newcommand{\F}{\mathcal{F}}
\newcommand{\M}{\mathcal{M}}

\renewcommand{\L}{\mathcal{L}}
\renewcommand{\P}{\mathcal{P}}

\newcommand{\pvec}[1]{\vec{#1}\mkern2mu\vphantom{#1}}

\makeatletter \DeclareRobustCommand\widecheck[1]{{\mathpalette\@widecheck{#1}}} \def\@widecheck#1#2{%
	\setbox\z@\hbox{\m@th$#1#2$}%
	\setbox\tw@\hbox{\m@th$#1%
		\widehat{%
			\vrule\@width\z@\@height\ht\z@ \vrule\@height\z@\@width\wd\z@}$}%
	\dp\tw@-\ht\z@ \@tempdima\ht\z@ \advance\@tempdima2\ht\tw@ \divide\@tempdima\thr@@ \setbox\tw@\hbox{%
		\raise\@tempdima\hbox{\scalebox{1}[-1]{\lower\@tempdima\box \tw@}}}%
	{\ooalign{\box\tw@ \cr \box\z@}}}
\makeatother

\numberwithin{equation}{section}

\title{Weyl bound for trilinear periods via conformal bootstrap}

\author{Anshul Adve}
\address{Department of Mathematics, Princeton University, Princeton, NJ 08540, USA}
\email{aadve@princeton.edu}

\author{James Bonifacio}
\address{Department of Physics and Astronomy, University of Mississippi, University, MS 38677, USA}
\email{bonifacio@phy.olemiss.edu}

\author{Petr Kravchuk}
\address{King’s College London, Strand, London, WC2R 2LS, UK}
\email{petr.kravchuk@kcl.ac.uk}

\author{Dalimil Maz\'a\v{c}}
\address{Institut de Physique Th\'{e}orique, Universit\'{e} Paris-Saclay, CEA, CNRS, 91191 Gif-sur-Yvette, France}
\email{dalimil.mazac@gmail.com}

\author{Sridip Pal}
\address{Walter Burke Institute for Theoretical Physics, Caltech, Pasadena,  CA 91125, USA\vspace{-0.2cm}}\address{Curr.\ Add: Institut des Hautes \'Etudes Scientifiques,  91440 Bures sur Yvette,  France}
\email{sridip@ihes.fr}

\author{Alex Radcliffe}
\address{King’s College London, Strand, London, WC2R 2LS, UK}
\email{alexander.radcliffe@kcl.ac.uk}

\author{Gordon Rogelberg}
\address{Department of Physics, Yale University, New Haven, CT 06511, USA}
\email{gordon.rogelberg@yale.edu}

\date{\today}

\begin{document}
	
	\begin{abstract}
		Let $f_1,f_2$ be holomorphic modular forms of the same weight for a cocompact lattice $\Gamma < \PSL_2(\R)$.
		We estimate the rate of decay of the coefficients in the expansion of $f_1\overline{f_2}$ in a Laplace eigenbasis.
		By specializing our main theorem to the case where $\Gamma$ is arithmetic, we obtain new instances of the Weyl bound for triple product $L$-functions in the spectral aspect.
		Our method builds on the conformal bootstrap in physics.
	\end{abstract}
	
	\maketitle
	
	\parskip 0em
	\tableofcontents
	\parskip 0.5em
	
	\section{Introduction} \label{sec:intro}
	
	Let $\H$ denote the upper half-plane model of the hyperbolic plane.
	The orientation-preserving isometry group of $\H$ is $G = \PSL_2(\R)$, acting by M\"obius transformations.
	Let $\Gamma$ be a cocompact lattice in $G$ (not necessarily arithmetic), so $\Gamma \backslash \H$ is a compact hyperbolic surface, possibly with elliptic points.
	Let $f_1,f_2 \colon \H \to \C$ be holomorphic modular forms for $\Gamma$ of the same weight $2k$, with $k \in \Z_{>0}$.
	Equip $\Gamma \backslash \H$ with the area measure normalized to have total mass $1$. Then normalize the $f_i$ to have Petersson norm
	\begin{align*}
		\|y^k|f_i|\|_{L^2(\Gamma \backslash \H)} = 1.
	\end{align*}
	The left hand side is well-defined because $y^k|f_i|$ is a $\Gamma$-invariant function on $\H$.
	Since $f_1,f_2$ have the same weight (namely $2k$), $y^{2k} f_1 \overline{f_2}$ is also $\Gamma$-invariant, so it can be expanded in a basis of Laplace eigenfunctions for $\Gamma \backslash \H$ (since $\Gamma \backslash \H$ is compact, there is an honest orthonormal eigenbasis).
	In this paper, we consider the rate of decay of the coefficients in this expansion.
	These coefficients are sometimes called ``trilinear periods" or ``triple product periods," because they are given as integrals of products of three automorphic forms.
	
	Let $\{\varphi_r\}_{r \in \Z_{\geq 0}}$ be an orthonormal basis of $L^2(\Gamma \backslash \H)$ consisting of real-valued Laplace eigenfunctions, with $\varphi_0 = 1$.
	Let $\lambda_r \geq 0$ be the eigenvalue of $\varphi_r$.
	Assume the $\varphi_r$ are ordered so that $\lambda_r$ increases with $r$.
	Write $\lambda_r = s_r(1-s_r)$, with $s_r = \sigma_r + it_r$ the unique solution to this quadratic equation with real part $\sigma_r \in [\frac{1}{2},1]$ and imaginary part $t_r \geq 0$. Then $s_r \in (\frac{1}{2},1] \cup (\frac{1}{2}+i\R_{\geq 0})$.
	Define
	\begin{align*}
		C_r
		= \langle y^{2k} f_1\overline{f_2}, \varphi_r \rangle_{L^2(\Gamma \backslash \H)}.
	\end{align*}
	We are interested in the size of $C_r$ as $r \to \infty$ with $\Gamma, f_1, f_2$ fixed.
	Therefore, \textit{we allow all implicit constants to depend on $\Gamma, f_1, f_2$}.
	
	Since $y^{2k} f_1 \overline{f_2}$ is real analytic and $\varphi_r$ oscillates at frequency $\sim t_r$, one expects $C_r$ to decay exponentially with $t_r$ (see Subsection~\ref{subsec:notation} for the definition of $\sim$ and other asymptotic notation).
	We normalize out this exponential decay as follows.
	Denote
	\begin{align} \label{eqn:C_tilde_r_def}
		\widetilde{C}_r
		= \begin{cases}
			C_r &\text{if } t_r \leq 1, \\
			t_r^{-2k+1} e^{\pi t_r/2} C_r &\text{if } t_r > 1
		\end{cases}
	\end{align}
	(the cutoff $t_r \leq 1$ is arbitrary; all that matters is that $|\widetilde{C}_r| \sim |C_r|$ for $t_r \lesssim 1$).
	With this normalization, the $\widetilde{C}_r$ have size $\sim 1$ on average, as explained in Remark~\ref{rem:sharpness}.
	Our main result, Theorem~\ref{thm:Weyl}, says that on average over a short interval, one has $|\widetilde{C}_r| \lesssim_{\varepsilon} (t_r+1)^{\varepsilon}$.
	As a stepping stone to this theorem, we will show in Proposition~\ref{prop:convexity} that on average over a long interval, one has $|\widetilde{C}_r| \lesssim 1$.
	
	\begin{thm}[Weyl bound for $\widetilde{C}_r$] \label{thm:Weyl}
		Let $T \geq 1$ and $\varepsilon > 0$. Then
		\begin{align} \label{eqn:Weyl}
			\sum_{|t_r-T| \leq T^{\frac{1}{3}}} |\widetilde{C}_r|^2
			\lesssim_{\varepsilon} T^{\frac{4}{3}+\varepsilon}.
		\end{align}
	\end{thm}
	
	\begin{prop}[Convexity bound for $\widetilde{C}_r$] \label{prop:convexity}
		Let $T \geq 1$. Then
		\begin{align} \label{eqn:convexity}
			\sum_{t_r \leq T} |\widetilde{C}_r|^2
			\lesssim T^2.
		\end{align}
	\end{prop}
	
	The names ``Weyl bound" and ``convexity bound" are justified by the fact that Theorem~\ref{thm:Weyl} and Proposition~\ref{prop:convexity} imply the Weyl and convexity bounds, respectively, for certain triple product $L$-functions (see Section~\ref{sec:L}).
	For a similar reason, we call
	\begin{align} \label{eqn:Lindelof_C_r_tilde}
		|\widetilde{C}_r|^2 \lesssim_{\varepsilon} (t_r+1)^\varepsilon
	\end{align}
	the ``Lindel\"of bound" for $\widetilde{C}_r$, or simply ``Lindel\"of" for short.
	By Weyl's law, $\LHS\eqref{eqn:Weyl}$ is a sum of $\sim T^{\frac{4}{3}}$ many terms, and $\LHS\eqref{eqn:convexity}$ is a sum of $\sim T^2$ many terms, so \eqref{eqn:Weyl} and \eqref{eqn:convexity} are Lindel\"of-on-average bounds.
	
	\begin{rem}[Sharpness] \label{rem:sharpness}
		Using a real Tauberian theorem as in \cite{Qiao--Rychkov}, the proof of Proposition~\ref{prop:convexity} can be refined to give the asymptotic $\LHS\eqref{eqn:convexity} = (c+o(1)) T^2$ for some constant $c>0$ depending only on $k$. Thus the $\widetilde{C}_r$ indeed have size $\sim 1$ on average.
	\end{rem}
	
	\begin{rem}[Uniform dependence on $f_1,f_2$] \label{rem:uniform_dep_on_f}
		Once we know \eqref{eqn:Weyl} for $f_1,f_2$ ranging over a fixed orthonormal basis of weight $2k$ forms, we know it for all $f_1,f_2$ by linearity. Consequently the implicit constant in \eqref{eqn:Weyl} only depends on $f_1,f_2$ through $k$.
		The same goes for \eqref{eqn:convexity}.
	\end{rem}
	
	\begin{rem}[Non-compact case] \label{rem:non-cpt_case}
		If $\Gamma \backslash \H$ is non-compact (but still finite volume), then the left hand sides of \eqref{eqn:Weyl} and \eqref{eqn:convexity} can each be interpreted as a sum over cusp forms plus an integral over Eisenstein series.
		Our method extends easily to this situation, simply by including Eisenstein series in all spectral decompositions. The only reason we restrict to the compact case is to simplify notation.
	\end{rem}
	
	When $\Gamma$ is a congruence group coming from a quaternion algebra and $f_1,f_2,\varphi_r$ are Hecke eigenforms, each $\widetilde{C}_r$ is either zero or satisfies
	\begin{align} \label{eqn:C_tilde_approx_L}
		|\widetilde{C}_r|^2 = (t_r+1)^{o(1)} L(\tfrac{1}{2}, f_1 \otimes f_2 \otimes \varphi_r),
	\end{align}
	where $L(s, f_1 \otimes f_2 \otimes \varphi_r)$ is the triple product $L$-function of Garrett \cite{Garrett} (without $\Gamma$-factors) associated to the three cusp forms on $\PGL_2$ corresponding to $f_1,f_2,\varphi$ by Jacquet--Langlands.
	The formula \eqref{eqn:C_tilde_approx_L} is essentially due to Watson \cite[Theorem~3]{Watson} (see Theorem~\ref{thm:triple} and Corollary~\ref{cor:triple}).
	The proof of \eqref{eqn:C_tilde_approx_L} shows that the central value $L(\frac{1}{2}, f_1 \otimes f_2 \otimes \varphi_r)$ is nonnegative.
	The central values of such triple product $L$-functions have been well-studied from an analytic point of view because of their connection to quantum unique ergodicity on $\Gamma \backslash \H$ \cite{Watson,Nelson_25,Sarnak_01,Liu--Ye,Sound_weak_subconvexity,Holowinsky--Sound,Lindenstrauss,Sound_QUE}, and because when $\Gamma \backslash \H$ is non-compact, they generalize Rankin--Selberg $L$-functions \cite{Michel--Venkatesh} (see also Remark~\ref{rem:Weyl_Rankin-Selberg} and the table in Section~\ref{sec:L}).
	
	By specializing Theorem~\ref{thm:Weyl} to the above arithmetic setting, in Corollary~\ref{cor:Weyl_triple_product} we obtain a Weyl-type subconvex bound in the spectral aspect for $L(\frac{1}{2}, f_1 \otimes f_2 \otimes \varphi_r)$: we show that this $L$-value is bounded by its analytic conductor (defined in \cite{Iwaniec--Sarnak}) to the power $\frac{1}{6}+\varepsilon$, for any fixed $\varepsilon > 0$.
	Here the analytic conductor is $\sim (t_r+1)^8$.
	Note $\frac{8}{6} = \frac{4}{3}$ is the exponent in the right hand side of \eqref{eqn:Weyl}.
	No subconvex bound was previously known for $\Gamma \backslash \H$ compact (see Remarks~\ref{rem:table_summary} and \ref{rem:Gamma_may_be_cocompact}).
	By Remark~\ref{rem:non-cpt_case}, Theorem~\ref{thm:Weyl} also applies when $\Gamma \backslash \mathbf{H}$ is non-compact, but in this case the arithmetic applications are less novel (see Remark~\ref{rem:Weyl_Rankin-Selberg}).
	We refer to \cite{Michel} for a broad overview of subconvexity.
	
	The Weyl exponent $\frac{1}{6}$ is a significant threshold.
	Sub-Weyl bounds have only been proved for $\GL_1$ in the $t$-aspect (see, e.g., \cite{Bombieri--Iwaniec_1,Bombieri--Iwaniec_2}) and depth aspect \cite{Milicevic}, and very recently for $\GL_2$ in the same aspects \cite{Holowinsky_et_al}. Even for the Riemann zeta function, the exponent $\frac{1}{6}$ has only been marginally improved, the world record being $\frac{1}{6}-\frac{1}{84}$ \cite{Bourgain}.
	
	Let us now return to our original setting where $\Gamma$ is not necessarily arithmetic.
	Historically, it was a challenge even to prove the polynomial bound $|\widetilde{C}_r|^2 \lesssim t_r^{O(1)}$.
	This bound was desired for applications, and the problem of proving it was posed by Selberg \cite{Selberg} in the generality where $\Gamma \backslash \H$ may be non-compact and $f_1,f_2$ may be Maass rather than holomorphic.
	It was eventually proved in the non-compact holomorphic case by Good \cite{Good_81} using Poincar\'e series, and in full generality by Sarnak \cite{Sarnak_94}.
	
	All known estimates for an individual $\widetilde{C}_r$ are obtained as corollaries of estimates for $\widetilde{C}_r$ on average.
	These averaged estimates are of the following general shape (c.f. Theorem~\ref{thm:Weyl} and Proposition~\ref{prop:convexity}).
	Let $T \geq 1$ and $H = H(T) \in [1,T]$. Then for every $\varepsilon > 0$,
	\begin{align} \label{eqn:general_averaged_bound}
		\sum_{|t_r-T| \leq H} |\widetilde{C}_r|^2
		\lesssim_{\varepsilon} T^{1+\varepsilon} H M,
	\end{align}
	where $M$ is some explicit function of $T$.
	If $M = 1$, then this is Lindel\"of on average (because by Weyl's law, there are $\sim TH$ many terms on the left hand side).
	Dropping all but one term, \eqref{eqn:general_averaged_bound} implies that
	\begin{align} \label{eqn:general_C_r_bound}
		|\widetilde{C}_r|^2
		\lesssim_{\varepsilon} T^{1+\varepsilon} H M
		\qquad \text{for} \qquad
		|t_r-T| \leq H.
	\end{align}
	In the arithmetic setting, when \eqref{eqn:C_tilde_approx_L} holds, \eqref{eqn:general_C_r_bound} implies
	\begin{align} \label{eqn:general_L_bound}
		L(\tfrac{1}{2}, f_1 \otimes f_2 \otimes \varphi_r)
		\lesssim_{\varepsilon} T^{1+\varepsilon} H M
		\qquad \text{for} \qquad
		|t_r-T| \leq H.
	\end{align}
	The following table shows, in chronological order, which cases of \eqref{eqn:general_averaged_bound} have been proved, and in each case what quality $L$-function bound is given by \eqref{eqn:general_L_bound}.
	In the table, we allow $\Gamma \backslash \mathbf{H}$ to be non-compact, in which case the left hand side of \eqref{eqn:general_averaged_bound} should be interpreted as a sum over cusp forms plus an integral over Eisenstein series, as in Remark~\ref{rem:non-cpt_case}.
	In addition, we allow $f_1,f_2$ to both be Maass instead of both being holomorphic, in which case we set $k=0$; then \eqref{eqn:C_tilde_r_def} is still the correct normalization.
	The last line of the table corresponds to Theorem~\ref{thm:Weyl} together with Remark~\ref{rem:non-cpt_case}.
	
	\begin{table}[H]
		\centering
		\begin{tabular}{|c|c|c|c|c|c|}
			\hline
			\textbf{Citation} & \textbf{Assumptions on $\Gamma$} & \textbf{Assumptions on $f_1,f_2$} & $H$ & $M$ & \textbf{$L$-function bound}\\
			\hline
			\hline
			\cite{Good_81} & $\Gamma \backslash \H$ has cusps & Holomorphic & $T$ & $1$ & Convexity \\
			\hline
			\cite{Sarnak_94} & None & None & $1$ & $T$ & Convexity \\
			\hline
			\cite{Bernstein--Reznikov_99,Bernstein--Reznikov_04} & None & None & $T$ & $1$ & Convexity \\
			\hline
			\cite{Bernstein--Reznikov_10} & None & Maass & $T^{\frac{1}{3}}$ & $T^{\frac{1}{3}}$ & Subconvex (halfway \\
			& & & & & from convex to Weyl) \\
			\hline
			\cite{Suvitie} & Congruence group & Holomorphic & $T^{\frac{1}{3}}$ & $1$ & Weyl \\
			& in $\PSL_2(\Z)$ &&&& \\
			\hline
			\cite{Blomer--Jana--Nelson} & Congruence group & None & $T^{\frac{1}{3}}$ & $1$ & Weyl \\
			& in $\PSL_2(\Z)$ &&&& \\
			\hline
			\textbf{This paper} & \textbf{None} & \textbf{Holomorphic} & $T^{\frac{1}{3}}$ & $1$ & \textbf{Weyl} \\
			\hline
		\end{tabular}
		%\caption{Sample Table}
		%\label{table:sample_table}
	\end{table}
	
	Technically \cite{Sarnak_94,Bernstein--Reznikov_99,Bernstein--Reznikov_04} all assume $f_1,f_2$ are Maass, but they can be adapted to the holomorphic case.
	Similarly, both \cite{Suvitie} and \cite{Blomer--Jana--Nelson} assume $\Gamma = \PSL_2(\Z)$, but they can likely be adapted to any congruence subgroup.
	
	\begin{rem}[Summary of table] \label{rem:table_summary}
		The situation was as follows prior to the present paper.
		\begin{itemize} \itemsep = 0.5em
			\item When $\Gamma \backslash \H$ is compact, the only known instance of \eqref{eqn:general_averaged_bound} which implies a subconvex bound is due to Bernstein and Reznikov in their most recent paper \cite{Bernstein--Reznikov_10}. It applies only when $f_1,f_2$ are Maass, and it does not reach the Weyl bound.
			
			\item The only known instances of \eqref{eqn:general_averaged_bound} which imply the Weyl bound are for $\Gamma$ a congruence subgroup of $\PSL_2(\Z)$, due to Suvitie \cite{Suvitie} and Blomer, Jana, and Nelson \cite{Blomer--Jana--Nelson}.
		\end{itemize}
	\end{rem}
	
	We obtain the Weyl bound for all lattices $\Gamma$, including cocompact $\Gamma$, but only when $f_1,f_2$ are holomorphic (Remark~\ref{rem:use_of_discrete_series} points out where we use holomorphicity).
	We do not yet see how to treat the case where $f_1,f_2$ are Maass, although there is no serious obstruction to doing so.
	The main challenge appears to be that the relevant special functions are more complicated.
	More generally, it would be desirable to extend our results to the ``all archimedean types" setting as in \cite{Blomer--Jana--Nelson}.
	We believe we can make partial progress in this direction --- see Remark~\ref{rem:weight_aspect}.
	
	Our method is very similar to that of Bernstein and Reznikov in their papers \cite{Bernstein--Reznikov_04,Bernstein--Reznikov_10,Reznikov_unfolding}, especially \cite{Bernstein--Reznikov_10}.
	We highlight their work in Section~\ref{sec:Bernstein--Reznikov}.
	The main conceptual difference is in how we choose our test vectors --- specifically, a new feature of our work is that we employ \textit{non-factorizable} test vectors (see Remark~\ref{rem:non-factor}).
	This is explained in Subsection~\ref{subsec:abstract_heuristics} (and ``test vector" is defined at the end of Subsection~\ref{subsec:abstract_framework}).
	
	Our work is also similar in spirit to \cite{Michel--Venkatesh} and \cite[Section~4]{Venkatesh}, both of which prove subconvex bounds for trilinear periods by harmonic analytic or geometric techniques rather than techniques from analytic number theory.
	The details are however quite different.
	
	Our approach is inspired by the conformal bootstrap in physics, via an analogy between conformal field theories (CFTs) and hyperbolic manifolds.
	This analogy was introduced by the second author \cite{Bonifacio_22_1,Bonifacio_22_2}, following similar work on Einstein manifolds by the second author and Hinterbichler \cite{Bonifacio--Hinterbichler}, to put bounds on the first few eigenvalues of the Laplacian.
	The analogy was sharpened by the third, fourth, and fifth authors \cite{Kravchuk--Mazac--Pal}, who placed it in a representation-theoretic setting.
	The key tool we will use to prove Theorem~\ref{thm:Weyl} is the \textit{crossing equation} from \cite{Kravchuk--Mazac--Pal} (the term ``crossing" comes from physics; see Section~\ref{sec:crossing}).
	In language which has become standard in the analytic theory of automorphic forms, the crossing equation is a \textit{spectral reciprocity formula}.
	It is useful because among such reciprocity formulas, it is particularly explicit.

	By the polarization identity, it suffices to prove Theorem~\ref{thm:Weyl} and Proposition~\ref{prop:convexity} when $f_1 = f_2$; in this paragraph we assume this and write $f = f_1 = f_2$. In general, crossing equations express the associativity of multiplication $C^{\infty}(\Gamma\backslash G)\times C^{\infty}(\Gamma\backslash G)\rightarrow C^{\infty}(\Gamma\backslash G)$ after decomposing into irreducible representations. They are thus naturally indexed by quadruples of irreducibles in $L^2(\Gamma \backslash G)$, or equivalently quadruples of automorphic forms on $\Gamma \backslash G$. What we call ``the" crossing equation, namely Corollary~\ref{cor:crossing}, corresponds to the quadruple $(f,f,\overline{f},\overline{f})$, as we will see in Section~\ref{sec:crossing}. This case works out especially cleanly --- the special functions which appear are only hypergeometric ${}_2F_1$'s, whereas for other quadruples they may be more complicated.
	From the point of view of the conformal bootstrap, Theorem~\ref{thm:Weyl} is a single-correlator bound, because we only consider the single quadruple $(f,f,\overline{f},\overline{f})$.
	Experience with the conformal bootstrap in CFT suggests that using more correlators, i.e., using crossing equations associated to more quadruples, can give significantly better results.
	A striking example of this is the recent improvement in numerical precision of our knowledge of critical exponents for the 3d Ising model \cite{Chang_et_al_25}.
	It would be interesting to see if a mixed-correlator bootstrap could be set up to place bounds on the $C_r$.
	In \cite{Adve_25}, the first author shows rigorously that any statement about Laplace eigenvalues and trilinear periods which is provable for an arbitrary compact hyperbolic surface can in principle be proved directly from the crossing equations (making use of all quadruples of automorphic forms).
	This ``completeness result" genuinely requires using more quadruples than just those considered in \cite{Kravchuk--Mazac--Pal}.
	Indeed, the sixth author gave convincing numerical evidence in \cite{Radcliffe} that the nearly sharp bounds in \cite{Kravchuk--Mazac--Pal} do not converge to a sharp bound unless new correlators are introduced.
	
	As a consequence of the fact that Corollary~\ref{cor:crossing} involves ${}_2F_1$'s, we will need uniform asymptotics for hypergeometrics with multiple parameters varying.
	These are quite subtle, so we prefer not to quote them from the literature for fear of misinterpreting a reference (and in any case we were often unable to find references).
	We therefore work out all the needed asymptotics from scratch.
	The only facts about hypergeometric functions which we cite are exact identities.
	We do however cite Bessel function asymptotics with one parameter varying, because these are well-known and straightforward.
	
	\subsection{Organization}
	
	Section~\ref{sec:L} discusses the implications of Theorem~\ref{thm:Weyl} for $L$-functions.
	Section~\ref{sec:crossing} recalls the crossing equation from \cite{Kravchuk--Mazac--Pal}, and sets up the notation we will need to use it in the rest of the paper.
	Section~\ref{sec:strategy_heuristics} outlines the proof of Theorem~\ref{thm:Weyl} and explains heuristically why our method achieves the Weyl bound.
	The proof of Theorem~\ref{thm:Weyl} will use a certain averaged form of the crossing equation (see Corollaries~\ref{cor:crossing_averaged} and \ref{cor:crossing_averaged_specific}), and Section~\ref{sec:mu_def} motivates our choice of averaging.
	Choosing how to average is equivalent to choosing a test vector in a certain representation of $G$, as we explain in Section~\ref{sec:Bernstein--Reznikov}. Section~\ref{sec:Bernstein--Reznikov} also discusses the similarities and differences between this paper and \cite{Bernstein--Reznikov_10}.
	Section~\ref{sec:main_proofs} proves our main results, Theorem~\ref{thm:Weyl} and Proposition~\ref{prop:convexity}, assuming several estimates for hypergeometrics and oscillatory integrals involving hypergeometrics.
	These estimates are proved in Sections~\ref{sec:complementary}--\ref{sec:t_averaging}.
	Section~\ref{sec:complementary} proves them when the spectral parameter is bounded, in which case everything can be bounded trivially.
	Section~\ref{sec:asymptotics} proves all the estimates we will need for hypergeometrics.
	Finally, Sections~\ref{sec:u_averaging} and \ref{sec:t_averaging} use the results of Section~\ref{sec:asymptotics} to estimate the oscillatory integrals which arise in the averaged crossing equation in Corollary~\ref{cor:crossing_averaged_specific}.
	
	The proof of Theorem~\ref{thm:Weyl} is contained in Sections~\ref{sec:crossing}, \ref{sec:mu_def}, and \ref{sec:main_proofs}--\ref{sec:t_averaging}.
	The main conceptual ideas are in Sections~\ref{sec:crossing}--\ref{sec:main_proofs}.
	The analysis in Sections~\ref{sec:complementary}--\ref{sec:t_averaging} is just a matter of technique, and could be carried out in multiple ways.
	The algebraic expressions in these last four sections become quite lengthy, so when reading them it would be helpful to have a computer algebra system at hand.
	Indeed, for many of the results in these four sections, checking the statement numerically may be more convincing and more efficient than reading the proofs.
	
	\subsection{Notation} \label{subsec:notation}
	
	Throughout, $G = \PSL_2(\R)$, and $K$ is the maximal compact subgroup $\PSO_2(\R)$.
	
	We always use the principal branch of the logarithm. Consequently, $z^s$ is formally defined as $\exp(s\log z)$, with $\log z$ the principal branch.
	We also always use the branch of the hypergeometric ${}_2F_1$ which is defined on $\C \setminus [1,\infty)$, and the branch of the Bessel function $K_{\nu}$ defined on $\C \setminus (-\infty,0]$.
	
	We use $O$-notation and $o$-notation in the usual way.
	Given $X,Y \in \R$ with $Y$ nonnegative, the inequality $X \lesssim Y$, or equivalently $Y \gtrsim X$, means that $X \leq O(Y)$. If $X$ is also nonnegative, then $X \sim Y$ means $X \lesssim Y \lesssim X$.
	The stronger inequality $X \ll Y$, or equivalently $Y \gg X$, means that $X \leq cY$ for some $c \gtrsim 1$, where $c$ is sufficiently small for the relevant context.
	In other words, $X \ll Y$ is the negation of $X \gtrsim Y$.
	The weakest inequality $X \lessapprox Y$, or equivalently $Y \gtrapprox X$, means that $X \lesssim_{\delta} T^{\delta} Y$ for every $\delta > 0$, where $T$ is as in Proposition~\ref{prop:convexity}, Theorem~\ref{thm:Weyl}, or Theorem~\ref{thm:Weyl_L}, depending on the context.
	Similarly, $\widetilde{O}(Y)$ signifies a quantity which is $O_{\delta}(T^{\delta}Y)$ for every $\delta > 0$.
	Given $Z \gtrsim 1$, we write $Z^{-\infty}$ for a nonnegative number which is $\lesssim_n Z^{-n}$ for every $n \geq 0$.
	Here and throughout, we allow implicit constants to depend on all parameters which appear in their subscripts. As stated above, we further allow implicit constants to depend on $\Gamma,f_1,f_2$.
	
	In addition to the above inequality signs which have rigorous meaning, we also informally use $\approx,\lll,\ggg$.
	Roughly speaking, $X \approx Y$ means that $X$ is equal to $Y$ up to a negligible error, and $X \lll Y$, or equivalently $Y \ggg X$, should be thought of as the negation of $X \gtrapprox Y$.
	Since we do not define $\approx,\lll,\ggg$ precisely, we do not use them in proofs.
	
	The indicator function of a region defined only up to constants, such as $\1_{|x| \lesssim 1}$, is understood to be a smooth cutoff.
	Similarly, an integral over such a region is understood to be defined with respect to a smooth cutoff.
	Unless specified otherwise, these cutoffs are compactly supported in the given region rather than just having rapidly decaying tails.
	For example, we may write
	\begin{align*}
		\Big|\int_{|x| \lesssim 1} e^{-ix\cdot \xi} \, dx\Big|
		\lesssim_d (|\xi|+1)^{-\infty}
	\end{align*}
	for $\xi \in \R^d$; this would be false with a sharp cutoff.
	
	Further conventions are adopted at various places in the body of the paper.
	They are always written in italics, and in the most important cases they appear near the beginning of a section or subsection.
	
	\subsection*{Acknowledgements} We would like to thank Peter Sarnak for his encouragement and for feedback on earlier drafts. The first author is supported by the National Science Foundation Graduate Research Fellowship
	Program under Grant No. DGE-2039656. The third author is supported by UK Research and Innovation (UKRI) under the UK
	government's Horizon Europe funding Guarantee [grant number EP/X042618/1] and the
	Science and Technology Facilities Council [grant number ST/X000753/1]. During the work, the fifth author has been supported by the U.S. Department of Energy, Office of Science, Office of High Energy Physics, under Award Number DE-SC0011632, and by the Walter Burke Institute for Theoretical Physics. The sixth author has been supported by a studentship from the Faculty of Natural, Mathematical \& Engineering Sciences at King's College London. The seventh author has been supported by Simons Foundation grant 488651 (Simons Collaboration on the Nonperturbative
Bootstrap) and DOE grant DE-SC0017660.
	
	\section{Application to $L$-functions} \label{sec:L}
	
	In this section, \textit{we always work over $\Q$}.
	
	For $N \in \Z_{>0}$ and $T \geq H \geq 1$, let $\F_{N,T,H}$ be the family of cuspidal automorphic representations $\pi$ of $\PGL_2$ with level dividing $N$, such that $\pi_{\infty}$ is principal series with spectral parameter in $[T-H,T+H]$.
	By Weyl's law, $\#\F_{N,T,H} \sim_N TH$.
	Our main result about $L$-functions is the following Lindel\"of-on-average bound, which we will deduce from Theorem~\ref{thm:Weyl}.
	
	\begin{thm}[Weyl bound for triple product $L$-functions] \label{thm:Weyl_L}
		Let $\pi_1,\pi_2$ be cuspidal automorphic representations of $\PGL_2$, with $\pi_{1,\infty}, \pi_{2,\infty}$ discrete series of the same weight.
		Let $N \in \Z_{>0}$ and $T \geq 1$.
		Denote $\F = \F_{N,T,T^{1/3}}$.
		Then for every $\varepsilon>0$,
		\begin{align} \label{eqn:Weyl_L}
			\sum_{\pi_3 \in \F} L(\tfrac{1}{2}, \pi_1 \otimes \pi_2 \otimes \pi_3)
			\lesssim_{\pi_1,\pi_2,N,\varepsilon} T^{\frac{4}{3}+\varepsilon}.
		\end{align}
	\end{thm}
	
	As mentioned earlier, $L(\tfrac{1}{2}, \pi_1 \otimes \pi_2 \otimes \pi_3)$ is always nonnegative, so there is no need for absolute values on the left hand side.
	
	\begin{rem}[Weyl bound for Rankin--Selberg $L$-functions] \label{rem:Weyl_Rankin-Selberg}
		By Remark~\ref{rem:non-cpt_case}, our methods also give the analog of Theorem~\ref{thm:Weyl_L} where $\pi_3$ varies over Eisenstein series with spectral parameter in $[T-H,T+H]$.
		Then $L(\tfrac{1}{2}, \pi_1 \otimes \pi_2 \otimes \pi_3)$ is the square of a Rankin--Selberg $L$-function.
		To be precise, we can show that for any fixed $\pi_1,\pi_2$ as in Theorem~\ref{thm:Weyl_L}, any fixed Dirichlet character $\chi$, any $T \geq 2$, and any $\varepsilon > 0$,
		\begin{align} \label{eqn:Weyl_Rankin-Selberg}
			\int_{T-T^{\frac{1}{3}}}^{T+T^{\frac{1}{3}}} |L(\tfrac{1}{2}+it, \pi_1 \otimes \pi_2 \otimes \chi)|^2 \, dt
			\lesssim_{\pi_1,\pi_2,\chi,\varepsilon} T^{\frac{4}{3}+\varepsilon}.
		\end{align}
		This could be proved using the methods of \cite{Blomer--Jana--Nelson} as well.
		As written, \eqref{eqn:Weyl_Rankin-Selberg} is weaker than Lindel\"of on average. This is a consequence of the fact that on congruence hyperbolic surfaces, the Eisenstein part of the spectrum contributes less to the spectral measure than the cuspidal part.
	\end{rem}
	
	By nonnegativity of the central value, Theorem~\ref{thm:Weyl_L} implies that each individual $L$-value on the left hand side of \eqref{eqn:Weyl_L} is bounded by $T^{\frac{4}{3}+\varepsilon}$. Since the analytic conductor satisfies
	\begin{align*}
		C(\pi_1 \otimes \pi_2 \otimes \pi_3)
		\sim_{\pi_1,\pi_2,N} T^8
	\end{align*}
	for $\pi_3 \in \F_{N,T,T^{1/3}}$, we obtain
	
	\begin{cor}[Weyl bound for individual triple product $L$-functions] \label{cor:Weyl_triple_product}
		Let $\pi_1,\pi_2,\pi_3$ be cuspidal automorphic representations of $\PGL_2$, with $\pi_{1,\infty}, \pi_{2,\infty}$ discrete series of the same weight, and $\pi_{3,\infty}$ principal series.
		Let $N \in \Z_{>0}$, and assume $\pi_3$ has level dividing $N$.
		Then for every $\varepsilon > 0$,
		\begin{align} \label{eqn:Weyl_triple_general}
			L(\tfrac{1}{2}, \pi_1 \otimes \pi_2 \otimes \pi_3)
			\lesssim_{\pi_1,\pi_2,N,\varepsilon} C(\pi_1 \otimes \pi_2 \otimes \pi_3)^{\frac{1}{6}+\varepsilon}.
		\end{align}
	\end{cor}
	
	\begin{rem} \label{rem:weight_aspect}
		By similar methods, we expect we can remove the assumption in Corollary~\ref{cor:Weyl_triple_product} that $\pi_3$ is principal series. In other words, we should be able to treat the case where $\pi_3$ is discrete series at infinity, varying in the weight aspect.
		We hope to work this out in detail in the near future.
	\end{rem}
	
	To get a sense of the strength of Corollary~\ref{cor:Weyl_triple_product}, suppose we knew \eqref{eqn:Weyl_triple_general} for all triples $(\pi_1,\pi_2,\pi_3)$ of automorphic representations of $\PGL_2$ (of full level for simplicity), with $\pi_1,\pi_2$ fixed and $\pi_3$ varying.
	Then by taking some of the $\pi_i$ to be Eisenstein, we would obtain Weyl bounds for simpler $L$-functions. This point of view was used to great effect in \cite{Michel--Venkatesh} (though the goal there was to prove general subconvex bounds in a clean fashion, rather than to prove a Weyl bound).
	The following table illustrates how \eqref{eqn:Weyl_triple_general} specializes, for different choices of $\pi_1,\pi_2,\pi_3$, to many of the known cases of the Weyl bound for families of $L$-functions varying in an archimedean aspect.
	In this table, $\pi,\sigma$ denote cuspidal automorphic representations of $\PGL_2$, and again we suppose for simplicity that $\pi,\sigma$ are of full level.
	
	\begin{table}[H]
		\centering
		\begin{tabular}{|c|c|c|}
			\hline
			\textbf{Citation} & \textbf{Weyl bound} & \textbf{Choice of} $\pi_1,\pi_2,\pi_3$ \\
			\hline
			\hline
			\cite{Landau} following & $|\zeta(\tfrac{1}{2}+it)| \lesssim_{\varepsilon} (|t|+1)^{\frac{1}{6}+\varepsilon}$ & $\pi_1 = \pi_2 = 1 \boxplus 1$, \\
			Hardy--Littlewood & & $\pi_3 = |\cdot|_{\A}^{it} \boxplus |\cdot|_{\A}^{-it}$ \\
			\hline
			\cite{Ivic}, \cite{Jutila_01} & $|L(\tfrac{1}{2}, \pi)| \lesssim_{\varepsilon} C(\pi)^{\frac{1}{6}+\varepsilon}$ & $\pi_1 = \pi_2 = 1 \boxplus 1$, \\
			\cite{Jutila--Motohashi_05}
			& & $\pi_3 = \pi$ \\
			\hline
			\cite{Good_82}, \cite{Meurman} & $|L(\tfrac{1}{2}+it, \pi)| \lesssim_{\pi,\varepsilon} (|t|+1)^{\frac{1}{3}+\varepsilon}$ & $\pi_1 = 1 \boxplus 1$, $\pi_2 = \pi$, \\
			\cite{Jutila--Motohashi_05}
			& & $\pi_3 = |\cdot|_{\A}^{it} \boxplus |\cdot|_{\A}^{-it}$ \\
			\hline
			\cite{Lau--Liu--Ye},\cite{Jutila--Motohashi_06} & $|L(\tfrac{1}{2}, \pi \otimes \sigma)| \lesssim_{\sigma,\varepsilon} C(\pi)^{\frac{1}{3}+\varepsilon}$ & $\pi_1 = 1 \boxplus 1$, \\
			& & $\pi_2 = \sigma$, $\pi_3 = \pi$ \\
			\hline
			\cite{Suvitie} & $|L(\tfrac{1}{2}+it, \pi \otimes \sigma)| \lesssim_{\pi,\sigma,\varepsilon} (|t|+1)^{\frac{2}{3}+\varepsilon}$ & $\pi_1 = \pi$, $\pi_2 = \sigma$ \\
			\cite{Blomer--Jana--Nelson} & & $\pi_3 = |\cdot|_{\A}^{it} \boxplus |\cdot|_{\A}^{-it}$ \\
			\hline
		\end{tabular}
		%\vspace{0.5em}
		%\caption{{\color{red} [CAPTION]}}
		%\label{table:sample_table}
	\end{table}
	
	To our knowledge, the only known Weyl bound for a family of $L$-functions genuinely of degree $>4$ is that of Blomer, Jana, and Nelson \cite{Blomer--Jana--Nelson} (generalizing \cite{Suvitie}), who show that \eqref{eqn:Weyl_triple_general} holds when $\pi_1,\pi_2$ are cuspidal and $\pi_1,\pi_2,\pi_3$ have level $1$.
	There is a fundamental difficulty in extending their method to higher (but still bounded) level when $\pi_3$ is also cuspidal --- see Remark~\ref{rem:Gamma_may_be_cocompact}.
	In contrast, our method makes no distinction between level $1$ and level $O(1)$.
	
	The reduction from Theorem~\ref{thm:Weyl_L} to Theorem~\ref{thm:Weyl} goes via Theorems~\ref{thm:Prasad--Loke} and \ref{thm:triple} below.
	
	\begin{thm}[Prasad \cite{Prasad}, Loke \cite{Loke}] \label{thm:Prasad--Loke}
		Let $\pi_1, \pi_2, \pi_3$ be cuspidal automorphic representations of $\PGL_2$.
		Assume $L(\frac{1}{2}, \pi_1 \otimes \pi_2 \otimes \pi_3) \neq 0$.
		Then there exists a unique quaternion algebra $B$, and automorphic representations $\pi_i^B$ of $B^{\times}$ corresponding to $\pi_i$ by Jacquet--Langlands, such that
		\begin{align*}
			\Hom_{B^{\times}(\A)}(\pi_1^B \otimes \pi_2^B \otimes \pi_3^B, \C) \neq 0.
		\end{align*}
		Furthermore, $B$ is unramified at all places $v$ where at least one of $\pi_{1,v}, \pi_{2,v}, \pi_{3,v}$ is unramified.
	\end{thm}
	
	\begin{thm}[Triple product formula] \label{thm:triple}
		Let $\pi_1,\pi_2,\pi_3$ be cuspidal automorphic representations of $\PGL_2$, with $\pi_{1,\infty}, \pi_{2,\infty}$ discrete series of the same weight $2k$, and $\pi_{3,\infty}$ principal series.
		Let $N \in \Z_{>0}$, and assume $\pi_3$ has level dividing $N$.
		Assume also that $L(\frac{1}{2}, \pi_1 \otimes \pi_2 \otimes \pi_3) \neq 0$.
		Then let $B$ be as in Theorem~\ref{thm:Prasad--Loke}. Since $\pi_{3,\infty}$ is unramified, $B$ is unramified at infinity; identify $B^{\times}(\R)$ with $\GL_2(\R)$. Then there exists a subgroup $\Gamma \subseteq B^{\times}(\Z) \cap \SL_2(\R)$ of index $O_{\pi_1,\pi_2,N}(1)$, holomorphic cusp forms $f_1,f_2 \colon \mathbf{H} \to \C$ for $\Gamma$ of weight $2k$, and a Maass cusp form $\varphi \colon \Gamma \backslash \H \to \R$, such that $f_1,f_2,\varphi$ are Petersson-normalized Hecke eigenforms whose Jacquet--Langlands transfers to $\PGL_2$ generate $\pi_1,\pi_2,\pi_3$ respectively, and
		\begin{align} \label{eqn:triple}
			|\langle y^{2k} f_1 \overline{f_2}, \varphi \rangle_{L^2(\Gamma \backslash \H)}|^2
			\sim_{\pi_1, \pi_2, N} \frac{\Lambda(\frac{1}{2}, \pi_1 \otimes \pi_2 \otimes \pi_3)}{\Lambda(1, \pi_3, \Ad)},
		\end{align}
		where $\Lambda$ denotes the completed $L$-function.
		Moreover, for each $i=1,2$, one can take $f_i$ to depend only on $\pi_i$ and $\Gamma$, and one can take $\varphi$ to depend only on $\pi_3$ and $\Gamma$.
	\end{thm}
	
	\begin{proof}
		When $N$ is squarefree, Theorem~\ref{thm:triple} is a direct consequence of Watson's formula \cite[Theorem 3]{Watson}. In general, one can argue as in Chapter~6 of Woodbury's thesis \cite{Woodbury_Thesis} to deduce Theorem~\ref{thm:triple} from Ichino's formula \cite{Ichino}.
		Indeed, Chapter~6 of \cite{Woodbury_Thesis} proves an analogue of Theorem~\ref{thm:triple} in which the level of $\pi_3$ is allowed to vary as opposed to the spectral parameter.
	\end{proof}
	
	\begin{cor} \label{cor:triple}
		Let the notation be as in Theorem~\ref{thm:triple}, and let $t$ be the spectral parameter of $\pi_{3,\infty}$.
		Then one has
		\begin{align} \label{eqn:triple_bd}
			L(\tfrac{1}{2}, \pi_1 \otimes \pi_2 \otimes \pi_3)
			\sim_{\pi_1, \pi_2, N} e^{\pi t} (t+1)^{-4k+2+o(1)} |\langle y^{2k} f_1 \overline{f_2}, \varphi \rangle_{L^2(\Gamma \backslash \H)}|^2,
		\end{align}
		where the $o(1)$ in the exponent goes to zero as $t \to \infty$.
	\end{cor}
	
	We will only need the upper bound.
	
	\begin{proof}
		The Gamma factor in the right hand side of \eqref{eqn:triple} is
		\begin{align} \label{eqn:Gamma_factor}
			2^{4-4k} \pi^{1-4k} \Gamma\Big(2k-\frac{1}{2} + it\Big) \Gamma\Big(2k-\frac{1}{2} - it\Big)
		\end{align}
		(see, e.g., Appendix A of \cite{Woodbury_real}).
		By Stirling's approximation,
		\begin{align} \label{eqn:Stirling}
			\eqref{eqn:Gamma_factor} \sim_k e^{-\pi t} (t+1)^{4k-2}.
		\end{align}
		Therefore, we can rewrite \eqref{eqn:triple} as
		\begin{align*}
			|\langle y^{2k} f_1 \overline{f_2}, \varphi \rangle_{L^2(\Gamma \backslash \H)}|^2
			\sim_{\pi_1, \pi_2, N} e^{-\pi t} (t+1)^{4k-2} \frac{L(\frac{1}{2}, \pi_1 \otimes \pi_2 \otimes \pi_3)}{L(1, \pi_3, \Ad)}.
		\end{align*}
		The adjoint $L$-value in the denominator is positive and of size $\sim_N (t+1)^{o(1)}$ by \cite[Theorem~2]{Iwaniec} and \cite{Hoffstein--Lockhart}. The claimed estimate \eqref{eqn:triple_bd} follows.
	\end{proof}
	
	\begin{rem} \label{rem:Gamma_may_be_cocompact}
		The methods of Suvitie \cite{Suvitie} and Blomer--Jana--Nelson \cite{Blomer--Jana--Nelson} fail when the quaternion algebra $B$ in Theorem~\ref{thm:Prasad--Loke} is division, because then the lattice $\Gamma$ in Theorem~\ref{thm:triple} is cocompact.
		Both \cite{Suvitie} and \cite{Blomer--Jana--Nelson} use the Kuznetsov formula, and thus crucially use a cusp.
	\end{rem}
	
	Assuming Theorem~\ref{thm:Weyl}, we can now prove Theorem~\ref{thm:Weyl_L} easily.
	
	\begin{proof}[Proof of Theorem~\ref{thm:Weyl_L}]
		Write $\F = \F_{N,T,T^{1/3}}$.
		By the ramification condition in Theorem~\ref{thm:Prasad--Loke} and the index bound in Theorem~\ref{thm:triple}, there is a finite set $\S$ of lattices in $\SL_2(\R)$, depending only on $\pi_1,\pi_2,N$, such that for any $\pi_3 \in \F$, the lattice $\Gamma$ in Theorem~\ref{thm:triple} can be taken to be in $\S$.
		Choose some such $\Gamma_{\pi_3} \in \S$ for each $\pi_3 \in \F$.
		Then define $f_{1,\Gamma_{\pi_3}}, f_{2,\Gamma_{\pi_3}}, \varphi_{\pi_3,\Gamma_{\pi_3}}$ to be the forms which appear in Theorem~\ref{thm:triple}.
		Note that for each $i=1,2$, we may assume by the last sentence of Theorem~\ref{thm:triple} that $f_{i,\Gamma_{\pi_3}}$ only depends on $\pi_3$ through $\Gamma_{\pi_3}$.
		Thus it makes sense to define, for $\Gamma \in \S$,
		\begin{align*}
			f_{i,\Gamma}
			= \begin{cases}
				f_{i,\Gamma_{\pi_3}} &\text{if } \Gamma = \Gamma_{\pi_3} \text{ for some } \pi_3 \in \F, \\
				0 &\text{otherwise}.
			\end{cases}
		\end{align*}
		Similarly, for $\pi_3 \in \F$ and $\Gamma \in \S$, let
		\begin{align*}
			\varphi_{\pi_3,\Gamma}
			= \begin{cases}
				\varphi_{\pi_3,\Gamma_{\pi_3}} &\text{if } \Gamma = \Gamma_{\pi_3}, \\
				0 &\text{otherwise}.
			\end{cases}
		\end{align*}
		Then for each $\pi_3 \in \F$, there exists $\Gamma \in \S$ such that \eqref{eqn:triple_bd} holds with $f_i = f_{i,\Gamma}$ and $\varphi = \varphi_{\pi_3,\Gamma}$.
		Since $\varphi_{\pi_3,\Gamma}$ is either zero or Jacquet--Langlands transfers to $\pi_3$, the nonzero $\varphi_{\pi_3,\Gamma}$ have spectral parameter in $[T-T^{1/3}, T+T^{1/3}]$, and the $\varphi_{\pi_3,\Gamma}$ are orthogonal as $\pi_3$ ranges over $\F$.
		The nonzero $\varphi_{\pi_3,\Gamma}$ are also $L^2$-normalized.
		Therefore, there exists an orthonormal Laplace eigenbasis $\{\varphi_{r,\Gamma}\}$ for $L^2_{\cusp}(\Gamma \backslash \H)$, such that each nonzero $\varphi_{\pi_3,\Gamma}$ is equal to $\varphi_{r,\Gamma}$ for some $r$.
		Let $t_{r,\Gamma}$ be the spectral parameter of $\varphi_{r,\Gamma}$.
		Let $2k$ be the weight of $\pi_{1,\infty}$ and $\pi_{2,\infty}$, and define $C_{r,\Gamma}$ and $\widetilde{C}_{r,\Gamma}$ in terms of $f_{1,\Gamma}, f_{2,\Gamma}, \varphi_{r,\Gamma}, k$ as in Section~\ref{sec:intro}. Then by the upper bound in \eqref{eqn:triple_bd} followed by the definition \eqref{eqn:C_tilde_r_def} of $\widetilde{C}_{r,\Gamma}$,
		\begin{align*}
			\sum_{\pi_3 \in \F} L(\tfrac{1}{2}, \pi_1 \otimes \pi_2 \otimes \pi_3)
			&\lessapprox_{\pi_1,\pi_2,N}
			\sum_{\Gamma \in \S} \, \sum_{|t_{r,\Gamma} - T| \leq T^{1/3}} e^{\pi t_r} (t_r+1)^{-4k+2} |C_{r,\Gamma}|^2
			\\&\sim \sum_{\Gamma \in \S} \, \sum_{|t_{r,\Gamma} - T| \leq T^{1/3}} |\widetilde{C}_{r,\Gamma}|^2.
		\end{align*}
		By Theorem~\ref{thm:Weyl}, Remark~\ref{rem:non-cpt_case}, and the fact that $\S$ depends only on $\pi_1,\pi_2,N$, the right hand side is $\lessapprox_{\pi_1,\pi_2,N} T^{\frac{4}{3}}$.
	\end{proof}
	
	\section{The crossing equation} \label{sec:crossing}
	
	As mentioned in Section~\ref{sec:intro}, by the polarization identity, it is enough to prove Theorem~\ref{thm:Weyl} and Proposition~\ref{prop:convexity} when $f_1 = f_2$.
	\textit{So assume from now on that $f_1 = f_2 =: f$.}
	Recall that we set $G = \PSL_2(\R)$.
	View $f$ and $\varphi_r$ as functions on $\Gamma \backslash G$ in the usual way, so (abusing notation)
	\begin{align} \label{eqn:f_lift}
		f\bigg(
		\begin{pmatrix}
			1 & x \\
			0 & 1
		\end{pmatrix}
		\begin{pmatrix}
			y^{1/2} & 0 \\
			0 & y^{-1/2}
		\end{pmatrix}
		\begin{pmatrix}
			\cos\theta & \sin\theta \\
			-\sin\theta & \cos\theta
		\end{pmatrix}
		\bigg)
		= y^k f(x+iy) e^{2ki\theta}
	\end{align}
	and
	\begin{align*}
		\varphi_r\bigg(
		\begin{pmatrix}
			1 & x \\
			0 & 1
		\end{pmatrix}
		\begin{pmatrix}
			y^{1/2} & 0 \\
			0 & y^{-1/2}
		\end{pmatrix}
		\begin{pmatrix}
			\cos\theta & \sin\theta \\
			-\sin\theta & \cos\theta
		\end{pmatrix}
		\bigg)
		= \varphi_r(x+iy).
	\end{align*}
	Given $g \in G$ and $F$ a function on $\Gamma \backslash G$, write $gF$ for the function on $\Gamma \backslash G$ obtained by right-translation by $g$, namely $gF(x) = F(xg)$.
	Equip $\Gamma \backslash G$ with the Haar probability measure.
	Then
	\begin{align} \label{eqn:C_r_G_def}
		C_r = \langle |f|^2, \varphi_r \rangle_{L^2(\Gamma \backslash G)}
		\qquad \text{and} \qquad
		C_0
		= \langle |f|^2, 1 \rangle_{L^2(\Gamma \backslash G)}
		= \|f\|_{L^2(\Gamma \backslash G)}^2
		= 1.
	\end{align}
	In particular, $C_r \in \R$,
	so $|C_r|^2 = C_r^2$, and similarly $|\widetilde{C}_r|^2 = \widetilde{C}_r^2$.
	
	In the remainder of this section we summarize the derivation of the crossing equation from \cite{Kravchuk--Mazac--Pal}.
	We give a fairly detailed summary because this equation (in the form of Corollary~\ref{cor:crossing}) is the key algebraic input which powers our method.
	An alternative derivation of Corollary~\ref{cor:crossing} is given in the ``Applications" section of \cite{Adve_25}.
	
	The crossing equation comes from the following two observations in \cite{Kravchuk--Mazac--Pal}.
	Aside from these two observations, we will use no other nontrivial information about the $C_r$.
	
	\begin{itemize} \itemsep = 0.5em
		\item For $g_1,g_2,g_3,g_4 \in G$, the ``four-point correlation"
		\begin{align} \label{eqn:abstract_4_pt_cor}
			\int_{\Gamma \backslash G} (g_1f) (g_2f) (g_3\overline{f}) (g_4\overline{f})
		\end{align}
		can be expressed, via spectral expansion, as a linear combination of the $C_r^2$, where the $r$th coefficient is an explicit function of the $g_i$ and the eigenvalue $\lambda_r$. When written in the correct notation, this function takes a particularly simple form.
		
		\item It is not visible from the spectral expansion that switching $g_1$ and $g_2$ preserves \eqref{eqn:abstract_4_pt_cor}. This symmetry places constraints on the $C_r$ and $\lambda_r$.
	\end{itemize}
	
	The fact that the spectral expansion of \eqref{eqn:abstract_4_pt_cor} takes the shape described above can be seen through abstract representation-theoretic considerations as in \cite{Bernstein--Reznikov_04,Bernstein--Reznikov_10,Reznikov_unfolding}.
	This is discussed in Section~\ref{sec:Bernstein--Reznikov}.
	However, the observation that the $r$th coefficient can be written in a simple and explicit form comes from CFT, where there is a very similar formula for spectral expansions of four-point correlation functions (see, e.g., \cite[Section~2]{Mazac--Paulos_I}).
	This formula involves scaling dimensions (analogous to $\lambda_r$) and OPE coefficients (analogous to $C_r$), which are observable quantities associated to a CFT.
	The use of symmetries analogous to $g_1 \leftrightarrow g_2$ to constrain scaling dimensions and OPE coefficients is called the conformal bootstrap in physics.
	
	The paper \cite{Kravchuk--Mazac--Pal} used the constraints derived from the two observations above to study the low energy eigenvalues $\lambda_r$, in particular the first eigenvalue $\lambda_1$.
	Similar results have been obtained for low energy eigenvalues of the Laplacian on hyperbolic 3-manifolds \cite{Bonifacio--Mazac--Pal} and the Dirac operator on spin hyperbolic surfaces \cite{Gesteau--Pal--Simmons-Duffin--Xu}.
	In this paper, we use a subset of the constraints from \cite{Kravchuk--Mazac--Pal} to study the coefficients $C_r$ in the high energy limit $r \to \infty$.
	
	Since \eqref{eqn:abstract_4_pt_cor} is a function of $g_1,g_2,g_3,g_4$ and $\dim G = 3$, we have that \eqref{eqn:abstract_4_pt_cor} is a function of $4 \times 3 = 12$ real variables. Following \cite{Kravchuk--Mazac--Pal}, our first goal is to use symmetry to encode this as a holomorphic function of one complex variable. This is accomplished in Lemma~\ref{lem:fn_of_cross-ratio}.
	
	\textit{A priori}, the translates of $f$ are parameterized by $G$: for each $g \in G$, there is a translate $gf$. However, since $f$ is a weight vector, $gf$ only depends (up to scalars) on the coset $gK \in G/K$ (recall $K = \PSO_2(\R)$).
	The quotient $G/K$ identifies naturally with the upper half-plane. Actually, by applying a M\"obius transformation, we prefer to identify $G/K$ with the unit disc $\mathbf{D}$.
	Then we see that up to scalars, the translates of $f$ are parameterized by the unit disc.
	The following proposition makes this parameterization explicit, and shows that it is in fact holomorphic. The proposition also gives a similar parameterization of the translates of $\overline{f}$.
	
	The disc $\mathbf{D}$ is naturally $\PSU(1,1)/U(1)$, so identifying $G/K$ with $\mathbf{D}$ amounts to choosing an isomorphism $\phi \colon G \to \PSU(1,1)$ which takes $K$ to the diagonal $U(1)$.
	We make the choice
	\begin{align*}
		\phi(g)
		= \begin{pmatrix}
			1 & i \\
			i & 1
		\end{pmatrix}^{-1}
		g
		\begin{pmatrix}
			1 & i \\
			i & 1
		\end{pmatrix},
		\qquad \text{so} \qquad
		\phi
		\begin{pmatrix}
			\cos\theta & \sin\theta \\
			-\sin\theta & \cos\theta
		\end{pmatrix}
		=
		\begin{pmatrix}
			e^{i\theta} & 0 \\
			0 & e^{-i\theta}
		\end{pmatrix}.
	\end{align*}
	Let $\mathbf{D}'
	= \{z \in \widehat{\C} : |z| > 1\}$ be the standard disc around $\infty$ in the Riemann sphere $\widehat{\C} = \C \cup \{\infty\}$.
	We parameterize the translates of $\overline{f}$ by $\mathbf{D}'$.
	
	\begin{prop} \label{prop:coherent_states}
		There exist unique holomorphic functions $\mathscr{O} \colon \mathbf{D} \to C^{\infty}(\Gamma \backslash G)$ and $\widetilde{\mathscr{O}} \colon \mathbf{D}' \to C^{\infty}(\Gamma \backslash G)$, such that
		\begin{align} \label{eqn:O(0)}
			\mathscr{O}(0)
			= f
			\qquad \text{and} \qquad
			\lim_{z \to \infty} z^{2k} \widetilde{\mathscr{O}}(z)
			= \overline{f},
		\end{align}
		and such that for all $g \in G$ and $z \in \mathbf{D}$ (respectively $z \in \mathbf{D}'$),
		\begin{align} \label{eqn:O_equivariance}
			g\mathscr{O}(z)
			= (cz+d)^{-2k} \mathscr{O}(\phi(g) \cdot z)
			\qquad \text{and} \qquad
			g\widetilde{\mathscr{O}}(z)
			= (cz+d)^{-2k} \widetilde{\mathscr{O}}(\phi(g) \cdot z),
		\end{align}
		where
		\begin{align*}
			\phi(g)
			= \begin{pmatrix}
				a & b \\
				c & d
			\end{pmatrix}
			\qquad \text{and} \qquad
			\phi(g) \cdot z
			= \frac{az+b}{cz+d}.
		\end{align*}
	\end{prop}
	
	For $z \in \mathbf{D}$, the function $\mathscr{O}(z) \in C^{\infty}(\Gamma \backslash G)$ is analogous to a local operator in a CFT.
	
	We first outline the proof of Proposition~\ref{prop:coherent_states}, and then explain how it follows from \cite{Kravchuk--Mazac--Pal}.
	
	\begin{proof}[Sketch of proof]
		Since $\PSU(1,1)$ acts transitively on $\mathbf{D}$ and $\mathbf{D}'$, uniqueness of $\mathscr{O}$ and $\widetilde{\mathscr{O}}$ is clear from \eqref{eqn:O(0)} and \eqref{eqn:O_equivariance}.
		Existence is more interesting.
		Once we have existence of $\mathscr{O}$, we can take
		\begin{align*}
			\widetilde{\mathscr{O}}(z)
			= z^{-2k} \overline{\mathscr{O}(\overline{z}^{-1})},
		\end{align*}
		so it suffices to construct $\mathscr{O}$.
		By the discussion above, there is a unique equivariant map $\mathbf{P}(\mathscr{O}) \colon \mathbf{D} \to \mathbf{P}(C^{\infty}(\Gamma \backslash G))$, taking values in the projectivization of $C^{\infty}(\Gamma \backslash G)$, which maps the origin in $\mathbf{D}$ to the line generated by $f$.
		Since $f$ isn't itself $K$-invariant (though the line it generates is), the map $\mathbf{P}(\mathscr{O})$ doesn't quite lift to an equivariant map $\mathbf{D} \to C^{\infty}(\Gamma \backslash G)$.
		Instead, one can check that there is a unique lift $\mathscr{O}$ satisfying \eqref{eqn:O(0)} and the twisted equivariance property \eqref{eqn:O_equivariance}.
		Since $f$ is smooth, it follows from \eqref{eqn:O_equivariance} that $\mathscr{O}$ is smooth.
		
		It remains to show that $\mathscr{O}$ is holomorphic, i.e., $\overline{\partial}\mathscr{O}$ is identically zero on $\mathbf{D}$. By equivariance, $\overline{\partial}\mathscr{O}$ will vanish identically as soon as it vanishes at one point. So it suffices to check that $\overline{\partial}\mathscr{O}(0) = 0$.
		This is a computation which eventually reduces to the fact that $f$ is a lowest weight vector in the $G$-representation $C^{\infty}(\Gamma \backslash G)$, so $f$ is killed by the lowering operator.
	\end{proof}
	
	\begin{rem} \label{rem:use_of_discrete_series}
		We emphasize that the holomorphy of $\mathscr{O}$ is due to $f$ being lowest weight in $C^{\infty}(\Gamma \backslash G)$, which in turn is due to $f$ being holomorphic as a function on $\H$. This is the only place where we use that $f$ is a holomorphic modular form as opposed to a Maass form.
	\end{rem}
	
	The proof of Proposition~\ref{prop:coherent_states} in \cite{Kravchuk--Mazac--Pal} is a little different: they first define $\mathscr{O}$ by an explicit formula and then check that it obeys \eqref{eqn:O(0)} and \eqref{eqn:O_equivariance}.
	
	\begin{proof}[Proof of Proposition~\ref{prop:coherent_states} from \cite{Kravchuk--Mazac--Pal}]
		As explained above, uniqueness is clear. Existence follows from Definition~2.1, (3.31), and Proposition~3.3 in \cite{Kravchuk--Mazac--Pal} (actually \cite[(3.31)]{Kravchuk--Mazac--Pal} has a typo --- it should say
		\begin{align*}
			u \cdot \mathcal{O}(z) = (\overline{\beta}z+\overline{\alpha})^{-2n} \mathcal{O}(u \cdot z),
		\end{align*}
		as is used for example in \cite[(3.51)]{Kravchuk--Mazac--Pal}).
	\end{proof}
	
	With the above parameterization of the translates of $f$ and $\overline{f}$ in hand, considering \eqref{eqn:abstract_4_pt_cor} is equivalent to considering
	\begin{align} \label{eqn:4_pt_cor}
		\int_{\Gamma \backslash G} \mathscr{O}(z_1) \mathscr{O}(z_2) \widetilde{\mathscr{O}}(z_3) \widetilde{\mathscr{O}}(z_4)
	\end{align}
	for $z_1,z_2 \in \mathbf{D}$ and $z_3,z_4 \in \mathbf{D}'$. The spectral expansion of \eqref{eqn:4_pt_cor}, rather than \eqref{eqn:abstract_4_pt_cor}, is computed in \cite{Kravchuk--Mazac--Pal}, and it is this computation which has a close parallel in CFT.
	
	Let $\pi_r$ be the unitary subrepresentation of $L^2(\Gamma \backslash G)$ generated by $\varphi_r$.
	Then $\pi_r$ is irreducible.
	Let $\D$ be the direct sum of all discrete series subrepresentations of $L^2(\Gamma \backslash G)$.
	Then we have the orthogonal decomposition
	\begin{align} \label{eqn:L^2(Gamma_mod_G)_decomp}
		L^2(\Gamma \backslash G)
		= \D \oplus \bigoplus_{r=0}^{\infty} \pi_r.
	\end{align}
	Let $P_{\D}$ and $P_r$ denote the corresponding orthogonal projections.
	These projections commute with complex conjugation, so for any $u,v \in L^2(\Gamma \backslash G)$,
	\begin{align}
		\int_{\Gamma \backslash G} u v
		= \langle u, \overline{v} \rangle
		&= \langle P_{\D}(u), P_{\D}(\overline{v}) \rangle + \sum_{r=0}^{\infty} \langle P_r(u), P_r(\overline{v}) \rangle
		\notag
		\\&= \int_{\Gamma \backslash G} P_{\D}(u) P_{\D}(v) + \sum_{r=0}^{\infty} \int_{\Gamma \backslash G} P_r(u) P_r(v),
		\label{eqn:projections_commute}
	\end{align}
	where $\langle \cdot,\cdot \rangle = \langle \cdot,\cdot \rangle_{L^2(\Gamma \backslash G)}$.
	Applying this with $u = \mathscr{O}(z_1) \widetilde{\mathscr{O}}(z_3)$ and $v = \mathscr{O}(z_2) \widetilde{\mathscr{O}}(z_4)$,
	\begin{align} \label{eqn:4_pt_spectral_decomp}
		\eqref{eqn:4_pt_cor}
		= I_{\D}(\pvec{z}) + \sum_{r=0}^{\infty} I_r(\pvec{z}),
	\end{align}
	where
	\begin{align*}
		I_{\D}(\pvec{z})
		= \int_{\Gamma \backslash G} P_{\D}(\mathscr{O}(z_1) \widetilde{\mathscr{O}}(z_3)) P_{\D}(\mathscr{O}(z_2) \widetilde{\mathscr{O}}(z_4))
	\end{align*}
	and
	\begin{align*}
		I_r(\pvec{z})
		= \int_{\Gamma \backslash G} P_r(\mathscr{O}(z_1) \widetilde{\mathscr{O}}(z_3)) P_r(\mathscr{O}(z_2) \widetilde{\mathscr{O}}(z_4)).
	\end{align*}
	It follows from abstract functional-analytic considerations that $\RHS\eqref{eqn:4_pt_spectral_decomp}$ converges absolutely and locally uniformly in $\vec{z}$ --- this is explained using Dini's theorem in the paragraph after \cite[Lemma~3.13]{Kravchuk--Mazac--Pal}.
	
	In our notation, \cite[Lemma 3.11]{Kravchuk--Mazac--Pal} says that $I_{\D}$ is identically zero.
	If one is familiar with the conformal bootstrap, then their argument is very natural.
	In Lemma~\ref{lem:vanishing_hom} and Proposition~\ref{prop:I_d_vanishes} below, we give an alternative argument which is perhaps more natural for someone with a background in automorphic forms. Both arguments are purely representation-theoretic.
	
	\begin{lem} \label{lem:vanishing_hom}
		Let $\pi,\pi'$ be discrete series representations of $G$.
		Then
		\begin{align*}
			\Hom_G(\pi \otimes \overline{\pi}, \pi')
			= 0.
		\end{align*}
	\end{lem}
	
	We will apply Lemma~\ref{lem:vanishing_hom} with $\pi$ smooth and $\pi'$ unitary, but we suppress analytic technicalities in the statement of the lemma because the proof shows that any reasonable interpretation of the statement is true.
	
	\begin{proof}
		By replacing $\pi, \pi'$ with their complex conjugates if necessary, we may assume $\pi, \pi'$ are both lowest weight representations.
		Let $M \in \Hom_G(\pi \otimes \overline{\pi}, \pi')$; we wish to show that $M = 0$.
		Consider
		\begin{align*}
			E = \{v \in \pi : M(v \otimes w) = 0 \text{ for all } w \in \overline{\pi}\}.
		\end{align*}
		Let $v \in \pi$ be a nonzero vector of lowest weight.
		Then for all weight vectors $w \in \overline{\pi}$, the tensor product $v \otimes w$ is a weight vector in $\pi \otimes \overline{\pi}$ of weight $\leq 0$. On the other hand, since $\pi'$ is a lowest weight representation, all of its weights are strictly positive. Thus $M(v \otimes w) = 0$ for all weight vectors $w \in \overline{\pi}$. Since weight vectors span $\overline{\pi}$, we deduce that $v \in E$ and in particular $E \neq 0$.
		It is clear from the definition of $E$ that $E$ is a closed subrepresentation of $\pi$, so $E = \pi$ because $\pi$ is irreducible. This means that $M = 0$.
	\end{proof}
	
	\begin{prop} \label{prop:I_d_vanishes}
		$I_{\D}$ is identically zero.
	\end{prop}
	
	\begin{proof}
		It suffices to show that for any discrete series subrepresentation $\pi' \subseteq \D$,
		\begin{align} \label{eqn:vanishing_projection}
			P_{\pi'}(\mathscr{O}(z_1) \widetilde{\mathscr{O}}(z_3))
			= 0.
		\end{align}
		Let $\pi \subseteq C^{\infty}(\Gamma \backslash G)$ be the smooth discrete series subrepresentation generated by $f$. Then $\LHS\eqref{eqn:vanishing_projection}$ is the image of $\mathscr{O}(z_1) \otimes \widetilde{\mathscr{O}}(z_3)$ under the $G$-equivariant map $\pi \otimes \overline{\pi} \to \pi'$ given by multiplication followed by orthogonal projection. By Lemma~\ref{lem:vanishing_hom}, this map is identically zero.
	\end{proof}
	
	To compute \eqref{eqn:4_pt_cor}, it now only remains to compute $I_r(\pvec{z})$.
	
	Given $z_i,z_j \in \C$, denote $z_{ij} = z_i - z_j$.
	
	\begin{lem} \label{lem:fn_of_cross-ratio}
		Let $z_1,z_2 \in \mathbf{D}$ and $z_3,z_4 \in \mathbf{D}' \setminus \{\infty\}$, and let $r \in \Z_{\geq 0}$.
		Then
		\begin{align} \label{eqn:conf_inv_I_r}
			z_{12}^{2k} z_{34}^{2k} I_r(\pvec{z})
		\end{align}
		depends only on the cross-ratio
		\begin{align} \label{eqn:cross-ratio}
			\frac{z_{12} z_{34}}{z_{13} z_{24}}.
		\end{align}
	\end{lem}
	
	\begin{proof}
		Using the twisted equivariance property \eqref{eqn:O_equivariance}, direct calculation shows that \eqref{eqn:conf_inv_I_r} is preserved under the diagonal action of $\PSU(1,1)$ on $\mathbf{D} \times \mathbf{D} \times \mathbf{D}' \times \mathbf{D}'$ by M\"obius transformations.
		Since \eqref{eqn:conf_inv_I_r} is holomorphic, it follows by analytic continuation that \eqref{eqn:conf_inv_I_r} is preserved not just by M\"obius transformations in $\PSU(1,1)$ but by those in the complexification $\PGL_2(\C)$.
		Thus indeed \eqref{eqn:conf_inv_I_r} depends only on the cross-ratio.
	\end{proof}
	
	Combining \eqref{eqn:4_pt_spectral_decomp}, Proposition~\ref{prop:I_d_vanishes}, and Lemma~\ref{lem:fn_of_cross-ratio}, we have now accomplished our goal of expressing \eqref{eqn:abstract_4_pt_cor} (up to simple prefactors) as a holomorphic function of one complex variable, namely the cross ratio of $\vec{z}$.
	The next step is to compute \eqref{eqn:conf_inv_I_r} explicitly in terms of the cross-ratio.
	The answer will be expressed in terms of the hypergeometric function $H_s$ defined by
	\begin{align} \label{eqn:H_s_def}
		H_s(z)
		= {}_2F_1(s,1-s;1;z).
	\end{align}
	Recall that the Taylor expansion of $z \mapsto {}_2F_1(a,b;c;z)$ at the origin is
	\begin{align} \label{eqn:2F1_series}
		{}_2F_1(a,b;c;z)
		= \sum_{n=0}^{\infty} \frac{(a)_n (b)_n}{(c)_n} \frac{z^n}{n!},
	\end{align}
	where
	\begin{align*}
		(q)_n
		= q(q+1) \cdots (q+n-1)
	\end{align*}
	is the (rising) Pochhammer symbol (and $(q)_0 = 1$).
	Assuming $c \not\in \Z_{\leq 0}$ to avoid division by zero, the series \eqref{eqn:2F1_series} is well-defined and absolutely convergent for $|z| < 1$. When $s=0$ in \eqref{eqn:H_s_def}, inspecting the Taylor coefficients of $H_0$ at the origin shows that $H_0$ is identically $1$.
	
	\begin{prop} \label{prop:conf_block_computation}
		Let $z_1,z_2 \in \mathbf{D}$ and $z_3,z_4 \in \mathbf{D}' \setminus \{\infty\}$, and let $r \in \Z_{\geq 0}$.
		Let $z$ be the cross-ratio \eqref{eqn:cross-ratio}. Then
		\begin{align} \label{eqn:conf_block_computation}
			z_{12}^{2k} z_{34}^{2k} I_r(\pvec{z})
			= C_r^2 z^{2k} H_{s_r}(z).
		\end{align}
	\end{prop}
	
	Recall that we defined $s_r$ in Section~\ref{sec:intro} to satisfy $\lambda_r = s_r(1-s_r)$.
	As with Proposition~\ref{prop:coherent_states}, we first sketch the proof of Proposition~\ref{prop:conf_block_computation}, and then say how it follows from \cite{Kravchuk--Mazac--Pal}.
	
	\begin{proof}[Sketch of proof]
		Let $\Delta$ be the Casimir operator for $G$, normalized so that it acts on right $K$-invariant functions on $\Gamma \backslash G$ in the same way that the Laplacian acts on functions on $\Gamma \backslash G/K = \Gamma \backslash \mathbf{H}$. Then $P_r(\Delta-\lambda_r) = 0$ as operators on $L^2(\Gamma \backslash G)$, and consequently
		\begin{align} \label{eqn:P_r(Delta-lamda_r)=0}
			\int_{\Gamma \backslash G} P_r((\Delta-\lambda_r)[\mathscr{O}(z_1) \widetilde{\mathscr{O}}(z_3)]) P_r(\mathscr{O}(z_2) \widetilde{\mathscr{O}}(z_4))
			= 0.
		\end{align}
		One can rewrite $\LHS\eqref{eqn:P_r(Delta-lamda_r)=0}$ by expanding $(\Delta-\lambda_r)[\mathscr{O}(z_1) \widetilde{\mathscr{O}}(z_3)]$ via the product rule for differentiation. After doing this and simplifying, \eqref{eqn:P_r(Delta-lamda_r)=0} becomes a second order ODE for $\LHS\eqref{eqn:conf_block_computation}$ as a function of the cross-ratio $z$, and hence confines $\LHS\eqref{eqn:conf_block_computation}$ to a two-dimensional space of functions.
		One can check that $\RHS\eqref{eqn:conf_block_computation}$ solves the ODE.
		The ODE has a (regular) singular point at the origin, but $\LHS\eqref{eqn:conf_block_computation}$ is holomorphic for $z$ near the origin. This further confines $\LHS\eqref{eqn:conf_block_computation}$ to the one-dimensional subspace of solutions to the ODE which are holomorphic near the origin.
		Thus
		\begin{align*}
			z_{12}^{2k} z_{34}^{2k} I_r(\pvec{z})
			= c z^{2k} H_{s_r}(z)
		\end{align*}
		for some constant $c \in \C$; we wish to show that $c = C_r^2$.
		Dividing both sides by $z^{2k}$,
		\begin{align} \label{eqn:divided_prop}
			z_{13}^{2k} z_{24}^{2k} I_r(\pvec{z})
			= c H_{s_r}(z).
		\end{align}
		Set $z_1 = z_2 = 0$.
		Then by \eqref{eqn:O(0)} and the definition of $I_r(\pvec{z})$,
		\begin{align} \label{eqn:z_4_to_infty_limit}
			\lim_{z_3,z_4 \to \infty} \LHS\eqref{eqn:divided_prop}
			= \int_{\Gamma \backslash G} P_r(f\overline{f}) P_r(f\overline{f})
			= \|P_r(|f|^2)\|_{L^2(\Gamma \backslash G)}^2.
		\end{align}
		Since $|f|^2$ is $K$-invariant, $P_r(|f|^2) \in \pi_r^K = \C\varphi_r$, so
		\begin{align*}
			P_r(|f|^2)
			= \langle P_r(|f|^2), \varphi_r \rangle \varphi_r
			= \langle |f|^2, \varphi_r \rangle \varphi_r
			= C_r\varphi_r.
		\end{align*}
		Inserting this into \eqref{eqn:z_4_to_infty_limit} and using that $\varphi_r$ is $L^2$-normalized, we get
		\begin{align*}
			\lim_{z_3,z_4 \to \infty} \LHS\eqref{eqn:divided_prop}
			= C_r^2.
		\end{align*}
		On the other hand,
		\begin{align*}
			\lim_{z_3,z_4 \to \infty} \RHS\eqref{eqn:divided_prop}
			= c H_{s_r}(0)
			= c.
		\end{align*}
		Equating these two limits, we conclude that indeed $c=C_r^2$.
	\end{proof}
	
	\begin{proof}[Proof of Proposition~\ref{prop:conf_block_computation} from \cite{Kravchuk--Mazac--Pal}]
		In our notation, \cite[(3.88) and Lemma 3.14]{Kravchuk--Mazac--Pal} say
		\begin{align} \label{eqn:KMP_Lem_3.14}
			C_r^2 z_{12}^{-2k} z_{34}^{-2k} \Big(\frac{z}{z-1}\Big)^{2k} H_{s_r}\Big(\frac{z}{z-1}\Big)
			= I_r(z_1,z_2,z_4,z_3).
		\end{align}
		After multiplying both sides by $z_{12}^{2k} z_{34}^{2k}$, this is almost \eqref{eqn:conf_block_computation}. The two differences are that on the left hand side of \eqref{eqn:KMP_Lem_3.14}, one has $\frac{z}{z-1}$ in place of $z$, and on the right hand side of \eqref{eqn:KMP_Lem_3.14}, $z_3,z_4$ are switched. The explanation for this is that $\frac{z}{z-1}$ is the cross-ratio of $z_1,z_2,z_4,z_3$. Thus \eqref{eqn:KMP_Lem_3.14} implies \eqref{eqn:conf_block_computation} by switching $z_3$ and $z_4$.
	\end{proof}
	
	Combining \eqref{eqn:4_pt_spectral_decomp} with Propositions~\ref{prop:I_d_vanishes} and \ref{prop:conf_block_computation}, we finally obtain a formula for \eqref{eqn:4_pt_cor}.
	
	\begin{thm}[Spectral expansion of 4-point function] \label{thm:4_pt_formula}
		Let $z_1,z_2 \in \mathbf{D}$ and $z_3,z_4 \in \mathbf{D}' \setminus \{\infty\}$.
		Let $z$ be the cross-ratio \eqref{eqn:cross-ratio}.
		Then
		\begin{align} \label{eqn:4_pt_formula}
			z_{12}^{2k} z_{34}^{2k} \int_{\Gamma \backslash G} \mathscr{O}(z_1) \mathscr{O}(z_2) \widetilde{\mathscr{O}}(z_3) \widetilde{\mathscr{O}}(z_4)
			= z^{2k} \sum_{r=0}^{\infty} C_r^2 H_{s_r}(z),
		\end{align}
		where the right hand side converges absolutely and locally uniformly in $z$.
	\end{thm}
	
	\begin{proof}
		As stated above, \eqref{eqn:4_pt_formula} is immediate from \eqref{eqn:4_pt_spectral_decomp} and Propositions~\ref{prop:I_d_vanishes} and \ref{prop:conf_block_computation}.
		The convergence assertion holds because $\RHS\eqref{eqn:4_pt_spectral_decomp}$ converges absolutely and locally uniformly in $\vec{z}$.
	\end{proof}
	
	After switching $z_1,z_2$, Theorem~\ref{thm:4_pt_formula} essentially becomes \cite[Theorem~2.3, part~(2)]{Kravchuk--Mazac--Pal}.
	
	\begin{cor}[Crossing equation] \label{cor:crossing}
		Let $z \in \C \setminus [1,\infty)$. Then
		\begin{align} \label{eqn:crossing}
			z^{2k} \sum_{r=0}^{\infty} C_r^2 H_{s_r}(z)
			= \Big(\frac{z}{z-1}\Big)^{2k} \sum_{r=0}^{\infty} C_r^2 H_{s_r}\Big(\frac{z}{z-1}\Big),
		\end{align}
		with both sides converging absolutely and locally uniformly in $z$.
	\end{cor}
	
	\begin{proof}
		The set of complex numbers which arise as the cross-ratio of four points $z_1,z_2,z_3,z_4$ as in Theorem~\ref{thm:4_pt_formula} is exactly $\C \setminus [1,\infty)$.
		So choose such $z_1,z_2,z_3,z_4$ with cross-ratio $z$. The left hand side of \eqref{eqn:4_pt_formula} is the same when $z_1,z_2$ are switched, i.e.,
		\begin{align} \label{eqn:4_pt_switch}
			z_{12}^{2k} z_{34}^{2k} \int_{\Gamma \backslash G}
			\mathscr{O}(z_1) \mathscr{O}(z_2) \widetilde{\mathscr{O}}(z_3) \widetilde{\mathscr{O}}(z_4)
			= z_{21}^{2k} z_{34}^{2k} \int_{\Gamma \backslash G}
			\mathscr{O}(z_2) \mathscr{O}(z_1) \widetilde{\mathscr{O}}(z_3) \widetilde{\mathscr{O}}(z_4).
		\end{align}
		Expressing each side of this equation using Theorem~\ref{thm:4_pt_formula} yields \eqref{eqn:crossing}, because the cross-ratio of $z_2,z_1,z_3,z_4$ is $\frac{z}{z-1}$.
		Theorem~\ref{thm:4_pt_formula} gives convergence as well.
	\end{proof}
	
	Following the terminology in \cite{Kravchuk--Mazac--Pal}, we call $\LHS\eqref{eqn:crossing}$ the \textit{u-channel expansion} of \eqref{eqn:4_pt_switch}, and $\RHS\eqref{eqn:crossing}$ the \textit{t-channel expansion} of \eqref{eqn:4_pt_switch}, or just ``u-channel" and ``t-channel" for short.
	
	Next, we rewrite the crossing equation in the form which will be most useful for us later.
	For $s = \sigma+it$, denote
	\begin{align} \label{eqn:tilde_H_s_def}
		\widetilde{H}_s
		= \begin{cases}
			H_s &\text{if } t \leq 1, \\
			t^{4k-2} e^{-\pi t} H_s &\text{if } t > 1,
		\end{cases}
		\qquad \text{so that} \qquad
		C_r^2 H_{s_r} = \widetilde{C}_r^2 \widetilde{H}_{s_r}.
	\end{align}
	
	\begin{cor}[Crossing equation (renormalized and reparameterized)] \label{cor:crossing_2}
		Let $z \in \C$ with $\re z > 0$. Then
		\begin{align} \label{eqn:crossing_2}
			z^{2k} \sum_{r=0}^{\infty} \widetilde{C}_r^2 \widetilde{H}_{s_r}(1-z^2)
			= z^{-2k} \sum_{r=0}^{\infty} \widetilde{C}_r^2 \widetilde{H}_{s_r}(1-z^{-2}),
		\end{align}
		with both sides converging absolutely and locally uniformly in $z$.
	\end{cor}
	
	We also refer to the left/right hand side of \eqref{eqn:crossing_2} as the u-/t-channel.
	We call the hypergeometric $\widetilde{H}_s(1-z^2)$ on the left the \textit{u-channel conformal block}, and the hypergeometric $\widetilde{H}_s(1-z^{-2})$ on the right the \textit{t-channel conformal block}.
	
	\begin{proof}
		Let $w = 1-z^2$. Since $\re z > 0$, we have $w \in \C \setminus [1,\infty)$.
		Applying Corollary~\ref{cor:crossing} with $w$ in place of $z$, and then dividing both sides by $w^{2k}$, we get
		\begin{align} \label{eqn:crossing_divided_by_w^2k}
			\sum_{r=0}^{\infty} C_r^2 H_{s_r}(w)
			= \Big(\frac{1}{w-1}\Big)^{2k} \sum_{r=0}^{\infty} C_r^2 H_{s_r}\Big(\frac{w}{w-1}\Big)
		\end{align}
		for $w \neq 0$, with absolute and locally uniform convergence in $w$.
		By Cauchy's estimates for holomorphic functions, \eqref{eqn:crossing_divided_by_w^2k} extends over $w=0$ with the same quality convergence.
		We have
		\begin{align*}
			\frac{1}{w-1}
			= -z^{-2}
			\qquad \text{and} \qquad
			\frac{w}{w-1}
			= 1-z^{-2}.
		\end{align*}
		Inserting this into \eqref{eqn:crossing_divided_by_w^2k}, and replacing $C_r$ and $H_{s_r}$ by their normalized forms using \eqref{eqn:tilde_H_s_def},
		\begin{align*}
			\sum_{r=0}^{\infty} \widetilde{C}_r^2 \widetilde{H}_{s_r}(1-z^2)
			= z^{-4k} \sum_{r=0}^{\infty} \widetilde{C}_r^2 \widetilde{H}_{s_r}(1-z^{-2}),
		\end{align*}
		with absolute and locally uniform convergence in $z$.
		Multiplying both sides by $z^{2k}$ yields \eqref{eqn:crossing_2}.
	\end{proof}
	
	\begin{cor}[Averaged crossing equation, general form] \label{cor:crossing_averaged}
		Let $\mu$ be a finite compactly supported measure on the right half-plane $\{\re z > 0\}$.
		Denote
		\begin{align*}
			W(s)
			= \int z^{2k} \widetilde{H}_s(1-z^2) \, d\mu(z)
			\qquad \text{and} \qquad
			\widecheck{W}(s)
			= \int z^{-2k} \widetilde{H}_s(1-z^{-2}) \, d\mu(z).
		\end{align*}
		Then
		\begin{align*}
			\sum_{r=0}^{\infty} W(s_r) \widetilde{C}_r^2
			= \sum_{r=0}^{\infty} \widecheck{W}(s_r) \widetilde{C}_r^2,
		\end{align*}
		with absolute convergence on both sides.
	\end{cor}
	
	\begin{proof}
		Integrate \eqref{eqn:crossing_2} against $\mu$, and pass the integral through the sums. This is justified because the sums converge locally uniformly in $z$. The absolute convergence in Corollary~\ref{cor:crossing_averaged} follows from the absolute convergence in Corollary~\ref{cor:crossing_2}.
	\end{proof}
	
	\section{Strategy and heuristics} \label{sec:strategy_heuristics}
	
	From now on, fix some small constant $\varepsilon > 0$, and \textit{allow all implicit constants to depend on $\varepsilon$}. Theorem~\ref{thm:Weyl} and Proposition~\ref{prop:convexity} are trivial for $T \lesssim 1$, so \textit{assume also that $T \gg 1$}.
	
	Let $H = T^{\frac{1}{3}+2\varepsilon}$.
	To prove Theorem~\ref{thm:Weyl}, we aim to prove the Lindel\"of-on-average bound
	\begin{align} \label{eqn:Lindelof_on_average}
		\sum_{|t_r-T| \leq H} \widetilde{C}_r^2
		\lessapprox TH.
	\end{align}
	In Section~\ref{sec:mu_def}, we will engineer a finite measure $\mu$ with compact support in $\{\re z > 0\}$, such that if $W$ is defined as in Corollary~\ref{cor:crossing_averaged}, then
	\begin{align} \label{eqn:W_approx_cutoff}
		W(s_r)
		\approx \1_{|t_r-T| \lesssim H}.
	\end{align}
	Proposition~\ref{prop:I_u_principal} makes this precise.
	
	Define $\widecheck{W}$ as in Corollary~\ref{cor:crossing_averaged} with the same choice of $\mu$. Then we have the averaged crossing equation
	\begin{align} \label{eqn:insert_W_approx_cutoff}
		\sum_{|t_r-T| \lesssim H}
		\widetilde{C}_r^2
		= \sum_{r=0}^{\infty} \widecheck{W}(s_r) \widetilde{C}_r^2,
	\end{align}
	where the cutoff to $|t_r-T| \lesssim H$ on the left hand side is interpreted as $W(s_r)$; this notation is justified by \eqref{eqn:W_approx_cutoff}.
	Our goal now is to upper bound the right hand side by $TH$.
	We will trivially bound $\widecheck{W}(s_r)$ when $t_r \leq 1$, and in Proposition~\ref{prop:I_t_principal} we will give an asymptotic for $\widecheck{W}(s_r)$ when $t_r > 1$.
	Since the statements and proofs of these estimates are rather involved (though elementary), in this section we only make use of the five qualitative properties of $\widecheck{W}$ below.
	Taking these properties for granted, the plan for the remainder of this section is to explain heuristically why we are able to bound $\RHS\eqref{eqn:insert_W_approx_cutoff}$ by $TH$ (yielding \eqref{eqn:Lindelof_on_average}), and why we would not be able to do this if $H$ were smaller than $T^{\frac{1}{3}}$.
	Since $\lambda_0 = 0$, we have $s_0 = 1$.
	Nevertheless, for clarity, we do not simplify $s_0$ to $1$ below.
	
	\begin{enumerate} \itemsep = 0.5em
		\item $\widecheck{W}(s_0)$ is the expected size of $\LHS\eqref{eqn:insert_W_approx_cutoff}$, so $\widecheck{W}(s_0) \sim TH$ (in particular $\widecheck{W}(s_0)$ is positive).
		
		\item $\widecheck{W}(s_0) \ggg |\widecheck{W}(s_r)|$ for $r>0$.
		
		\item $\widecheck{W}(s_r)$ is essentially supported on $\{t_r \lesssim T/H\}$.
		
		\item $\widecheck{W}(\tfrac{1}{2}+it)$ is oscillatory (quantitatively, it oscillates at frequencies $\sim \log T$, though we won't need this for the heuristics in this section or for the proof of Theorem~\ref{thm:Weyl}).
		
		\item $|\widecheck{W}(\tfrac{1}{2}+it)|$ is flat, meaning that if we define
		\begin{align*}
			M(N) = \max_{t \in [N,2N]} |\widecheck{W}(\tfrac{1}{2}+it)|,
		\end{align*}
		then for any fixed $1 \leq N \lesssim T/H$, there is a subset $S \subseteq [N,2N]$ of measure $\gtrsim N$ such that
		\begin{align*}
			|\widecheck{W}(\tfrac{1}{2}+it)| \gtrsim M(N)
			\qquad \text{for} \qquad
			t \in S.
		\end{align*}
		Informally, a nonnegative function is flat if on any dyadic interval, it achieves its maximum (up to a constant) a positive proportion of the time.
		For many purposes, upper bounding a flat function by its maximum on each dyadic interval is sharp.
	\end{enumerate}
	
	\begin{rem}
		The property (2) is because among all spherical irreducible unitary representations of $G$, the trivial representation is furthest from being tempered.
	\end{rem}
	
	\begin{rem}[Comparison with the Mellin transform]
		Qualitatively, $\widecheck{W}$ behaves like the Mellin transform of $x \mapsto W(\tfrac{1}{2}+ix)$. By \eqref{eqn:W_approx_cutoff}, the Mellin transform $\mathcal{M}(W)$ at $s=\sigma+it$ is approximately
		\begin{align*}
			\M(W)(s)
			\approx \M(\1_{|\cdot-T| \lesssim H})(s)
			= \int_{|x-T| \lesssim H} x^s \, \frac{dx}{x}
			= T^{s-1} H \1_{|t| \lesssim T/H},
		\end{align*}
		where the cutoff on the right hand side is not compactly supported but has rapidly decaying tails.
		Thus $\mathcal{M}(W)$ obeys (2)--(5), including the quantitative part of (4). It also obeys the ``$\GL_1$ version of (1)," namely $\mathcal{M}(W)(s_0) \sim H$. The reason $\widecheck{W}(s_0)$ is bigger by a factor of $T$ than $\M(W)(s_0)$ is that the Weyl law for $\PGL_2$ is bigger by a factor of $T$ than the Weyl law for $\GL_1$.
		More concretely, if $\Lambda_d$ is the Laplace spectrum of a fixed compact $d$-manifold for $d=1,2$, then $\Lambda_2 \cap [T,2T]$ is denser by a factor of $T$ than $\Lambda_1 \cap [T,2T]$.
		The dimensions 1 and 2 are relevant because the Riemmanian symmetric spaces associated to $\GL_1$ and $\PGL_2$ (i.e., the line and the hyperbolic plane) have dimensions 1 and 2, respectively.
		In summary, $\widecheck{W}$ and $\M(W)$ share all five properties (1)--(5).
		It would be interesting to have a conceptual explanation for this.
		One could imagine an explanation as follows. It was pointed out by A. Venkatesh that there is an integral operator mapping $W$ to $\widecheck{W}$, and that after suitable normalization, the kernel of this integral operator can be interpreted as a 6j symbol (the analogous appearance of the 6j symbol in the conformal bootstrap is well-known \cite{Karateev--Kravchuk--Simmons-Duffin,Kravchuk18,Liu--Perlmutter--Rosenhaus--Simmons-Duffin}).
		Perhaps this operator is a Fourier integral operator with canonical relation and principal symbol (defined in \cite{Hormander,Duistermaat}) similar to those of the Mellin transform.
	\end{rem}
	
	By \eqref{eqn:insert_W_approx_cutoff} and (3),
	\begin{align*}
		\sum_{|t_r-T| \lesssim H} \widetilde{C}_r^2
		\approx \sum_{t_r \lesssim T/H} \widecheck{W}(s_r) \widetilde{C}_r^2.
	\end{align*}
	On the right side, in view of (2) and (5), we expect that either the $r=0$ term will dominate because it is big, or the terms where $t_r \sim T/H$ will dominate because there are many of them.
	Indeed, one generally expects that for sums with flat amplitude, the extremal ranges should dominate (up to constants), unless there is a good reason why not.
	Assuming this,
	\begin{align} \label{eqn:C_r_average=main+error}
		\sum_{|t_r-T| \lesssim H} \widetilde{C}_r^2
		\approx \widecheck{W}(s_0) + \sum_{t_r \sim T/H} \widecheck{W}(s_r) \widetilde{C}_r^2
	\end{align}
	(here we have used that $\widetilde{C}_0 = C_0 = 1$ by \eqref{eqn:C_r_G_def}).
	We have fixed $H = T^{\frac{1}{3}+2\varepsilon}$ with an eye toward the Weyl bound, but we could have run the same argument with $H = T$, in which case \eqref{eqn:C_r_average=main+error} would be
	\begin{align*}
		\sum_{t_r \lesssim T} \widetilde{C}_r^2
		\approx \widecheck{W}(s_0) + \sum_{t_r \sim 1} \widecheck{W}(s_r) \widetilde{C}_r^2.
	\end{align*}
	In this case the sum on the right hand side contains only $O(1)$ many terms, so $\widecheck{W}(s_0)$ dominates by (2). Using this followed by (1), we get
	\begin{align} \label{eqn:heuristic_convexity}
		\sum_{t_r \lesssim T} \widetilde{C}_r^2
		\approx \widecheck{W}(s_0)
		\sim T^2.
	\end{align}
	This (heuristically) establishes Proposition~\ref{prop:convexity}, the convexity bound.
	
	Coming back to the Weyl bound and \eqref{eqn:C_r_average=main+error}, in view of (1) and (2) we denote
	\begin{align} \label{eqn:X_E_def}
		X = \widecheck{W}(s_0) \sim TH,
		\qquad
		E_{\true}
		= \Big|\sum_{t_r \sim T/H} \widecheck{W}(s_r) \widetilde{C}_r^2\Big|,
		\qquad
		E_{\Abs}
		= \sum_{t_r \sim T/H} |\widecheck{W}(s_r)| \widetilde{C}_r^2,
	\end{align}
	having in mind that $X$ is the main term in \eqref{eqn:C_r_average=main+error}, $E_{\true}$ is the ``true error" in \eqref{eqn:C_r_average=main+error}, and $E_{\Abs}$ is the bigger error obtained from estimating $E_{\true}$ by taking the absolute values inside the sum.
	By (4), the sum defining $E_{\true}$ is oscillatory, and the oscillation is fast enough that it appears very difficult to get any cancellation (see Remark~\ref{rem:Weyl_barrier}).
	Therefore, the best we can do is to trivially bound $E_{\true}$ by $E_{\Abs}$.
	This is the only place where our argument is not sharp.
	We estimate $E_{\Abs}$ using the convexity bound:
	\begin{align} \label{eqn:E_abs_convexity_bd}
		E_{\Abs}
		\lesssim M \sum_{t_r \sim T/H} \widetilde{C}_r^2
		\lesssim M(T/H)^2
		\qquad \text{where} \qquad
		M = \max_{t \sim T/H} |\widecheck{W}(\tfrac{1}{2}+it)|.
	\end{align}
	Since $|\widecheck{W}(\tfrac{1}{2}+it)|$ is flat by (5), we believe that this is sharp.
	Inserting this into \eqref{eqn:C_r_average=main+error},
	\begin{align} \label{eqn:expected_rigorous_bd}
		\sum_{|t_r-T| \lesssim H} \widetilde{C}_r^2
		= X + O(M(T/H)^2)
	\end{align}
	(we have changed $\approx$ in \eqref{eqn:C_r_average=main+error} to $=$ in \eqref{eqn:expected_rigorous_bd} because the error in $\approx$ should be much less than the error term in \eqref{eqn:expected_rigorous_bd}).
	We will be able to establish \eqref{eqn:expected_rigorous_bd} rigorously. Since the main term $X$ is of size $TH$, this implies \eqref{eqn:Lindelof_on_average} (i.e., we win) if and only if
	\begin{align} \label{eqn:win_condition}
		M(T/H)^2
		\lessapprox TH.
	\end{align}
	After a long computation involving hypergeometric asymptotics and stationary phase, we will find
	\begin{align} \label{eqn:M_size}
		M \sim \frac{H^{\frac{3}{2}}}{T^{\frac{1}{2}}}
	\end{align}
	(this follows from Proposition~\ref{prop:I_t_principal}).
	As a consequence, \eqref{eqn:win_condition} holds if and only if $H \gtrapprox T^{\frac{1}{3}}$. We have $H = T^{\frac{1}{3}+2\varepsilon}$, so we win, though just barely.
	On the other hand, to prove \eqref{eqn:Lindelof_on_average} for $H$ below $T^{\frac{1}{3}}$, we would need a bound for $E_{\true}$ better than $E_{\Abs}$.
	
	For the rigorous proof, there is no way to get around computing $M$.
	However, assuming some more heuristics, we can see that we should win without appealing to \eqref{eqn:M_size}, as follows.
	By \eqref{eqn:C_r_average=main+error},
	\begin{align} \label{eqn:expected_rigorous_bd_2}
		\sum_{|t_r-T| \lesssim H} \widetilde{C}_r^2
		= X + O(E_{\Abs}).
	\end{align}
	Since $X \sim TH$ and we can upper bound $E_{\Abs}$ sharply, \eqref{eqn:expected_rigorous_bd_2} implies \eqref{eqn:Lindelof_on_average} (i.e., we win) as long as
	\begin{align} \label{eqn:win_condition_2}
		E_{\Abs}
		\lessapprox X.
	\end{align}
	Let us suppose that the sum defining $E_{\true}$ has square root cancellation.
	Then
	\begin{align} \label{eqn:E_sqrt_cancellation}
		E_{\Abs}
		\sim \sqrt{(T/H)^2} E_{\true}
		= (T/H) E_{\true},
	\end{align}
	because by Weyl's law, this sum contains $\sim (T/H)^2$ many terms.
	We can also characterize $E_{\true}$ as the deviation of $\LHS\eqref{eqn:C_r_average=main+error}$ from the main term $X$.
	The generalization of square root cancellation to sums of pseudorandom \textit{nonnegative} numbers is square root deviation from the expected value.
	This would suggest
	\begin{align} \label{eqn:E_sqrt_deviation}
		E_{\true}
		\sim \frac{X}{\sqrt{TH}},
	\end{align}
	because $\LHS\eqref{eqn:C_r_average=main+error}$ is a sum over $\sim TH$ many terms.
	Putting \eqref{eqn:E_sqrt_cancellation} and \eqref{eqn:E_sqrt_deviation} together,
	\begin{align} \label{eqn:E_abs_size}
		E_{\Abs}
		\sim \frac{T^{\frac{1}{2}}}{H^{\frac{3}{2}}} X.
	\end{align}
	Thus
	\begin{align} \label{eqn:win_iff_H>T^1/3}
		\eqref{eqn:win_condition_2}
		\iff
		H \gtrapprox T^{\frac{1}{3}},
	\end{align}
	and as above, we win because $H = T^{\frac{1}{3}+2\varepsilon}$.
	We find it appealing that \eqref{eqn:win_iff_H>T^1/3} can be predicted without having to compute either the left or right hand sides of \eqref{eqn:win_condition_2} (we know how big $\RHS\eqref{eqn:win_condition_2}$ is, but we didn't use this in our derivation of \eqref{eqn:win_iff_H>T^1/3}).
	
	These heuristics also imply \eqref{eqn:M_size}: by \eqref{eqn:E_abs_size} and the expectation that \eqref{eqn:E_abs_convexity_bd} is sharp,
	\begin{align*}
		M(T/H)^2
		\sim \frac{T^{\frac{1}{2}}}{H^{\frac{3}{2}}} X.
	\end{align*}
	Inserting $X \sim TH$ and rearranging yields \eqref{eqn:M_size}.
	
	\section{Choice of averaged crossing equation} \label{sec:mu_def}
	
	To execute the strategy in Section~\ref{sec:strategy_heuristics}, we need to choose a measure $\mu$ on the right half-plane $\{\re z > 0\}$ which makes \eqref{eqn:W_approx_cutoff} hold, where as above, we define $W,\widecheck{W}$ as in Corollary~\ref{cor:crossing_averaged}.
	In particular, restricting \eqref{eqn:W_approx_cutoff} to the tempered part of the spectrum, we want $\mu$ to satisfy
	\begin{align} \label{eqn:mu_property}
		\int z^{2k} \widetilde{H}_{\frac{1}{2}+it}(1-z^2) \, d\mu(z)
		\approx \1_{|t-T| \lesssim H}
	\end{align}
	for $t \geq 0$; the left hand side is $W(\frac{1}{2}+it)$.
	
	Recall that Corollary~\ref{cor:crossing_averaged} is proved by averaging the crossing equation \eqref{eqn:crossing_2} against $\mu$.
	In general, one obtains asymptotic information from equations like \eqref{eqn:crossing_2} by sending the free parameter to the boundary of its domain, so in this case sending $z$ to $i\R$.
	This suggests that we take $\mu$ to be supported near $i\R$, with the distance to $i\R$ depending on $T$.
	The conformal block $\widetilde{H}_s(1-z^2)$ has singularities at $z=0$ and $z=\infty$, coming from the singularities of the hypergeometric.
	Thus in fact it makes sense to take $\mu$ supported either near $0$ or near $\infty$.
	Because of the $z \leftrightarrow z^{-1}$ symmetry of the crossing equation \eqref{eqn:crossing_2}, we might as well take $\mu$ supported near $0$ (it turns out that $\mu$ near $0$ is good for localizing in the u-channel, while by symmetry $\mu$ near $\infty$ is good for localizing in the t-channel).
	
	The critical range of $t$ is $t \sim T$, because this is where the right hand side of \eqref{eqn:mu_property} is nonzero.
	So suppose for now that $t \sim T$.
	Then for $z$ near $0$ in a certain range depending on $T$, we can approximate
	\begin{align} \label{eqn:hyp_approx_Bessel}
		\widetilde{H}_{\frac{1}{2}+it}(1-z^2)
		\approx \frac{1}{\pi} t^{4k-2} K_0(2tz)
	\end{align}
	in terms of the Bessel function $K_0$.
	This approximation is accurate up to a power saving error in $T$ (see \eqref{eqn:u_asymptotic_in_Bessels}).
	Right now, we will not say exactly what the range of $z$ is in which this holds, because it will hold when we need it.
	The appearance of the Bessel function can be understood as follows.
	Making the change of variable $w = tz$, \eqref{eqn:hyp_approx_Bessel} becomes
	\begin{align*}
		\widetilde{H}_{\frac{1}{2}+it}\Big(1 - \frac{w^2}{t^2}\Big)
		\approx \frac{1}{\pi} t^{4k-2} K_0(2w).
	\end{align*}
	The left hand side satisfies an ODE in $w$ which has regular singular points at $0,t,\infty$ (coming from the regular singular points $0,1,\infty$ of the hypergeometric ODE).
	As $t \to \infty$, two of these singular points coalesce at $\infty$, leading to an irregular singularity at $\infty$.
	This phenomenon is known as confluence.
	The simplest ODEs with a regular singularity at $0$ and an irregular singularity at $\infty$ are the Bessel ODEs.
	
	Suppose that $z$ is not only in the range where \eqref{eqn:hyp_approx_Bessel} holds, but also satisfies $|z| \gg T^{-1}$.
	Then since $t \sim T$, we have $|tz| \gg 1$, and we can asymptotically expand $K_0$ at infinity to get
	\begin{align} \label{eqn:hyp_approx_exp}
		\widetilde{H}_{\frac{1}{2}+it}(1-z^2)
		\approx \frac{t^{4k-\frac{5}{2}}}{2\sqrt{\pi z}} e^{-2tz}.
	\end{align}
	This is made precise by Proposition~\ref{prop:u_block_asymptotic}.
	
	With \eqref{eqn:hyp_approx_exp} in mind, we take $\mu$ to be the measure supported on the line $\re z = T^{-1}$, given by
	\begin{align} \label{eqn:mu_def}
		d\mu(z)
		= T^{-4k+\frac{5}{2}} H \1_{|y| \lesssim H^{-1}} e^{2iTy} z^{-2k+\frac{1}{2}} \, dy
		\qquad \text{where} \qquad
		z = T^{-1} + iy
	\end{align}
	(below we will take $\1_{|y| \lesssim H^{-1}}$ to be an explicit cutoff which is not quite smooth, but which is essentially so because it is exponentially small at its singular points).
	This is motivated by the following rough calculation: by \eqref{eqn:hyp_approx_exp}, continuing to denote $z = T^{-1}+iy$,
	\begin{align*}
		W(\tfrac{1}{2}+it)
		= \LHS\eqref{eqn:mu_property}
		&\approx \frac{t^{4k-\frac{5}{2}}}{2\sqrt{\pi}} \int_{|y| \lesssim H^{-1}} z^{2k-\frac{1}{2}} e^{-2tz} \, d\mu(z)
		\\&= \frac{1}{2\sqrt{\pi}} \Big(\frac{t}{T}\Big)^{4k-\frac{5}{2}} H \int_{|y| \lesssim H^{-1}} e^{-2t(T^{-1}+iy)} e^{2iTy} \, dy
		\\&= \frac{1}{2\sqrt{\pi}} \Big(\frac{t}{T}\Big)^{4k-\frac{5}{2}} H e^{-2t/T} \int_{|y| \lesssim H^{-1}} e^{2i(T-t)y} \, dy
		\\&\approx \1_{|t-T| \lesssim H},
	\end{align*}
	as desired.
	For technical reasons, we choose the cutoff $\1_{|y| \lesssim H^{-1}}$ in the definition of $\mu$ to be the truncated Gaussian $\1_{|y| \leq T^{\varepsilon}/H} \, e^{-(Hy)^2}$. Then
	\begin{align} \label{eqn:final_W_def}
		W(s)
		= T^{-4k+\frac{5}{2}} H \int_{-T^{\varepsilon}/H}^{T^{\varepsilon}/H} e^{-(Hy)^2} e^{2iTy} z^{\frac{1}{2}} \widetilde{H}_s(1-z^2) \, dy
	\end{align}
	and
	\begin{align} \label{eqn:final_W_check_def}
		\widecheck{W}(s)
		= T^{-4k+\frac{5}{2}} H \int_{-T^{\varepsilon}/H}^{T^{\varepsilon}/H} e^{-(Hy)^2} e^{2iTy} z^{-4k+\frac{1}{2}} \widetilde{H}_s(1-z^{-2}) \, dy,
	\end{align}
	where
	\begin{align} \label{eqn:y_range}
		z = T^{-1} + iy
		\qquad \text{with} \qquad
		|y| \leq T^{\varepsilon}/H = T^{-\frac{1}{3}-\varepsilon}.
	\end{align}
	Note that both $W$ and $\widecheck{W}$ are real-valued for $s \in (\frac{1}{2},1] \cup (\frac{1}{2}+i\R_{\geq 0})$, because under the involution $y \leftrightarrow -y$ of the domain of integration, each integrand goes to its complex conjugate.
	
	Applying Corollary~\ref{cor:crossing_averaged} with the above choice of $\mu$, we arrive at the specific averaged crossing equation which we will use to prove the Weyl bound.
	
	\begin{cor}[Our choice of averaged crossing equation] \label{cor:crossing_averaged_specific}
		Let $W,\widecheck{W}$ as in \eqref{eqn:final_W_def} and \eqref{eqn:final_W_check_def}. Then
		\begin{align*}
			\sum_{r=0}^{\infty} W(s_r) \widetilde{C}_r^2
			= \sum_{r=0}^{\infty} \widecheck{W}(s_r) \widetilde{C}_r^2,
		\end{align*}
		with absolute convergence on both sides.
	\end{cor}
	
	\section{Comparison with Bernstein--Reznikov} \label{sec:Bernstein--Reznikov}
	
	In this section, we allow $f$ to be either holomorphic or Maass (we continue to take $f_1 = f_2 = f$).
	If $f$ is Maass, assume $f$ is real-valued, and set $k=0$ in \eqref{eqn:C_tilde_r_def} and \eqref{eqn:f_lift}.
	Recall from the table in Section~\ref{sec:intro} that Bernstein and Reznikov proved
	
	\begin{thm}[\cite{Bernstein--Reznikov_10}]
		Assume $f$ is Maass. Then
		\begin{align*}
			\sum_{|t_r-T| \leq T^{\frac{1}{3}}} \widetilde{C}_r^2
			\lessapprox T^{\frac{5}{3}}.
		\end{align*}
	\end{thm}
	
	This is weaker than Lindel\"of on average, because the exponent on the right hand side is $\frac{5}{3}$ rather than $\frac{4}{3}$.
	This is enough for a subconvex bound, but not for the Weyl bound.
	
	Our method and \cite{Bernstein--Reznikov_10} can be placed in a common representation-theoretic framework.
	We describe this framework in Subsection~\ref{subsec:abstract_framework}, using notation similar to Bernstein and Reznikov's.
	Then in Subsection~\ref{subsec:abstract_heuristics}, we explain that the difference between our method and theirs is in the choice of test vectors. We also explain heuristically why their approach to choosing test vectors, which uses positivity in a crucial way, cannot give the Weyl bound.
	
	\subsection{Common representation-theoretic framework} \label{subsec:abstract_framework}
	
	Let $\pi$ be the smooth subrepresentation of $C^{\infty}(\Gamma \backslash G)$ generated by $f$.
	Let $I \in \Hom_G(\pi \otimes \pi \otimes \overline{\pi} \otimes \overline{\pi},\C)$ be the integration functional
	\begin{align*}
		I(\psi_1 \otimes \psi_2 \otimes \psi_3 \otimes \psi_4)
		= \int_{\Gamma \backslash G} \psi_1 \psi_2 \psi_3 \psi_4.
	\end{align*}
	Applying \eqref{eqn:projections_commute} with $u = \psi_1\psi_3$ and $v = \psi_2\psi_4$, we get
	\begin{align} \label{eqn:I_decomp_on_pure_tensors}
		I(\psi_1 \otimes \psi_2 \otimes \psi_3 \otimes \psi_4)
		= I_{\D}(\psi_1 \otimes \psi_2 \otimes \psi_3 \otimes \psi_4) + \sum_{r=0}^{\infty} I_r(\psi_1 \otimes \psi_2 \otimes \psi_3 \otimes \psi_4),
	\end{align}
	where $I_{\D}, I_r \in \Hom_G(\pi \otimes \pi \otimes \overline{\pi} \otimes \overline{\pi},\C)$ are defined by
	\begin{align} \label{eqn:I_D_def}
		I_{\D}(\psi_1 \otimes \psi_2 \otimes \psi_3 \otimes \psi_4)
		= \int_{\Gamma \backslash G} P_{\D}(\psi_1 \psi_3) P_{\D}(\psi_2 \psi_4)
	\end{align}
	and
	\begin{align} \label{eqn:I_r_def}
		I_r(\psi_1 \otimes \psi_2 \otimes \psi_3 \otimes \psi_4)
		= \int_{\Gamma \backslash G} P_r(\psi_1 \psi_3) P_r(\psi_2 \psi_4)
	\end{align}
	(here the notation is as in \eqref{eqn:L^2(Gamma_mod_G)_decomp}).
	Pure tensors span $\pi \otimes \pi \otimes \overline{\pi} \otimes \overline{\pi}$, so \eqref{eqn:I_decomp_on_pure_tensors} implies
	\begin{align} \label{eqn:I_decomp_general}
		I = I_{\D} + \sum_{r=0}^{\infty} I_r.
	\end{align}
	
	\begin{prop} \label{prop:I_decomp_without_D}
		Let $\Psi \in \pi \otimes \pi \otimes \overline{\pi} \otimes \overline{\pi}$.
		If $f$ is Maass, assume $\Psi$ is a linear combination of terms of the form $\psi_1 \otimes \psi_2 \otimes \psi_3 \otimes \psi_4$ with at least one of $\psi_1\psi_3$ and $\psi_2\psi_4$ being $K$-invariant. Then
		\begin{align*}
			I(\Psi)
			= \sum_{r=0}^{\infty} I_r(\Psi).
		\end{align*}
	\end{prop}
	
	\begin{proof}
		In view of \eqref{eqn:I_decomp_general}, it suffices to show that $I_{\D}(\Psi) = 0$.
		If $f$ is holomorphic, so $\pi$ is discrete series, then this follows from the proof of Proposition~\ref{prop:I_d_vanishes}.
		If $f$ is Maass, then since we're assuming $\Psi$ is of the above form, it suffices to check that
		\begin{align*}
			I_{\D}(\psi_1 \otimes \psi_2 \otimes \psi_3 \otimes \psi_4) = 0
			\qquad \text{whenever} \qquad
			\psi_1\psi_3 \text{ or } \psi_2\psi_4 \text{ is } K\text{-invariant}.
		\end{align*}
		Indeed, the $K$-invariance condition forces at least one of the two factors in the integrand in \eqref{eqn:I_D_def} to be zero.
	\end{proof}
	
	Define $M_r \in \Hom_G(\pi \otimes \overline{\pi}, \pi_r)$ to be the ``multiplication followed by projection" map
	\begin{align*}
		M_r(\psi \otimes \psi')
		= P_r(\psi \psi').
	\end{align*}
	If $f$ is Maass, then $f$ is real-valued, so $\overline{\pi} = \pi$ inside $C^{\infty}(\Gamma \backslash G)$, and commutativity of multiplication implies that $M_r$ factors through $\Sym^2\pi$.
	Thus in general, if we denote
	\begin{align*}
		\Pi
		= \begin{cases}
			\pi \otimes \overline{\pi} &\text{if } f \text{ is holomorphic}, \\
			\Sym^2\pi &\text{if } f \text{ is Maass},
		\end{cases}
	\end{align*}
	then $M_r \in \Hom_G(\Pi, \pi_r)$.
	This notation is useful because
	\begin{align} \label{eqn:Hom_one-dim}
		\dim \Hom_G(\Pi, \pi_r) = 1;
	\end{align}
	this is discussed in \cite[Section~2.4.2]{Bernstein--Reznikov_04}.
	Furthermore, there exists a unique map $M_r^{\Mod} \in \Hom_G(\Pi,\pi_r)$ such that $M_r^{\Mod}(f \otimes \overline{f}) = \varphi_r$.
	Since $f \otimes \overline{f}$ is $K$-invariant,
	\begin{align*}
		M_r(f \otimes \overline{f})
		= P_r(|f|^2)
		= P_{\pi_r^K}(|f|^2)
		= \langle |f|^2, \varphi_r \rangle \varphi_r
		= C_r \varphi_r.
	\end{align*}
	Thus by the one-dimensionality \eqref{eqn:Hom_one-dim},
	\begin{align} \label{eqn:M_r=C_rM_r^mod}
		M_r = C_r M_r^{\Mod}.
	\end{align}
	Following \cite{Bernstein--Reznikov_10}, we use the superscript ``mod" because we think of $M_r^{\Mod}$ as a purely representation-theoretically defined ``model" of $M_r$. Indeed, we see from the unique characterization of $M_r^{\Mod}$ that its definition is independent of the inclusions $\pi \hookrightarrow C^{\infty}(\Gamma \backslash G)$ and $\pi_r \hookrightarrow L^2(\Gamma \backslash G)$, whereas the definition of $M_r$ clearly depends on the realizations of $\pi$ and $\pi_r$ as spaces of functions on $\Gamma \backslash G$.
	
	Define $I_r^{\Mod} \in \Hom_G(\pi \otimes \pi \otimes \overline{\pi} \otimes \overline{\pi}, \C)$ by
	\begin{align*}
		I_r^{\Mod}(\psi_1 \otimes \psi_2 \otimes \psi_3 \otimes \psi_4)
		= \int_{\Gamma \backslash G} M_r^{\Mod}(\psi_1 \psi_3) M_r^{\Mod}(\psi_2 \psi_4).
	\end{align*}
	Then by \eqref{eqn:M_r=C_rM_r^mod} and the definition \eqref{eqn:I_r_def} of $I_r$,
	\begin{align} \label{eqn:I_r=C_r^2I_r^mod}
		I_r = C_r^2 I_r^{\Mod}.
	\end{align}
	Similarly to \eqref{eqn:tilde_H_s_def}, let
	\begin{align} \label{eqn:I_tilde_def}
		\widetilde{I}_r^{\Mod}
		= \begin{cases}
			I_r^{\Mod} &\text{if } t_r \leq 1, \\
			t_r^{4k-2} e^{-\pi t_r} I_r^{\Mod} &\text{if } t_r > 1,
		\end{cases}
		\qquad \text{so that} \qquad
		I_r
		= C_r^2 I_r^{\Mod}
		= \widetilde{C}_r^2 \widetilde{I}_r^{\Mod}.
	\end{align}
	
	Since $\pi \subseteq C^{\infty}(\Gamma \backslash G)$, an element $\Psi \in \pi \otimes \pi \otimes \overline{\pi} \otimes \overline{\pi}$ is naturally a function on $(\Gamma \backslash G)^4$. Let $\Psi|_{\Gamma \backslash G}$ denote the restriction of $\Psi$ to the diagonal, and view $\Psi|_{\Gamma \backslash G}$ as a function on $\Gamma \backslash G$. For example, if $\Psi$ is a pure tensor $\psi_1 \otimes \psi_2 \otimes \psi_3 \otimes \psi_4$, then $\Psi|_{\Gamma \backslash G}$ is the pointwise product of the $\psi_i$.
	Consequently, for general $\Psi$,
	\begin{align} \label{eqn:I_formula_general_input}
		I(\Psi)
		= \int_{\Gamma \backslash G} \Psi|_{\Gamma \backslash G}.
	\end{align}
	
	\begin{prop}[Representation-theoretic crossing inequality] \label{prop:crossing_ineq}
		Let $\Psi,\widecheck{\Psi} \in \pi \otimes \pi \otimes \overline{\pi} \otimes \overline{\pi}$ satisfy the pointwise inequality $\Psi|_{\Gamma \backslash G} \leq \widecheck{\Psi}|_{\Gamma \backslash G}$.
		If $f$ is Maass, assume $\Psi,\widecheck{\Psi}$ both obey the hypothesis of Proposition~\ref{prop:I_decomp_without_D}.
		Denote
		\begin{align*}
			W(s_r)
			= \widetilde{I}_r^{\Mod}(\Psi)
			\qquad \text{and} \qquad
			\widecheck{W}(s_r)
			= \widetilde{I}_r^{\Mod}(\widecheck{\Psi}).
		\end{align*}
		Then
		\begin{align} \label{eqn:crossing_ineq}
			\sum_{r=0}^{\infty} W(s_r) \widetilde{C}_r^2
			\leq \sum_{r=0}^{\infty} \widecheck{W}(s_r) \widetilde{C}_r^2.
		\end{align}
	\end{prop}
	
	\begin{proof}
		By \eqref{eqn:I_formula_general_input} and the pointwise inequality $\Psi|_{\Gamma \backslash G} \leq \widecheck{\Psi}|_{\Gamma \backslash G}$, we have $I(\Psi) \leq I(\widecheck{\Psi})$.
		Spectrally expanding both sides of this inequality via Proposition~\ref{prop:I_decomp_without_D}, and then inserting \eqref{eqn:I_tilde_def}, we obtain \eqref{eqn:crossing_ineq}.
	\end{proof}
	
	Our strategy and Bernstein--Reznikov's can now be stated in the same language: in order to upper bound
	\begin{align} \label{eqn:to_estimate_for_us_and_BR}
		\sum_{|t_r-T| \leq T^{\frac{1}{3}}} \widetilde{C}_r^2,
	\end{align}
	choose $\Psi,\widecheck{\Psi}$ as in Proposition~\ref{prop:crossing_ineq} such that
	\begin{align} \label{eqn:W_lower_bd}
		W(s_r)
		\gtrsim \1_{|t_r-T| \lesssim T^{1/3}}
	\end{align}
	(maybe up to a small additive error), and then trivially bound $\RHS\eqref{eqn:crossing_ineq}$ by dyadic decomposition and the convexity bound (Proposition~\ref{prop:convexity}).
	We call $\Psi,\widecheck{\Psi}$ \textit{test vectors}.
	
	\subsection{Difference in choice of test vectors} \label{subsec:abstract_heuristics}
	
	We depart from \cite{Bernstein--Reznikov_10} in our choice of test vectors $\Psi,\widecheck{\Psi}$.
	In this subsection, we first explain how Corollary~\ref{cor:crossing_averaged_specific} can be seen as a special case of Proposition~\ref{prop:crossing_ineq}, and then we describe Bernstein--Reznikov's choice of $\Psi,\widecheck{\Psi}$ and explain why it does not give a sharp bound for \eqref{eqn:to_estimate_for_us_and_BR}.
	
	Given $\vec{z} = (z_1,z_2,z_3,z_4)$ as in Proposition~\ref{prop:conf_block_computation}, let
	\begin{align} \label{eqn:Psi_z_def}
		\Psi_{\vec{z}}
		= z_{12}^{2k} z_{34}^{2k} \,
		\mathscr{O}(z_1) \otimes \mathscr{O}(z_2) \otimes \widetilde{\mathscr{O}}(z_3) \otimes \widetilde{\mathscr{O}}(z_4).
	\end{align}
	Then Proposition~\ref{prop:conf_block_computation} says that
	\begin{align} \label{eqn:conf_block_computation_restated}
		I_r(\Psi_{\vec{z}})
		= C_r^2 z^{2k} H_{s_r}(z),
	\end{align}
	where $z$ is the cross-ratio \eqref{eqn:cross-ratio} of $\vec{z}$. Thus by \eqref{eqn:I_r=C_r^2I_r^mod},
	\begin{align} \label{eqn:I_r^mod_computation}
		I_r^{\Mod}(\Psi_{\vec{z}})
		= z^{2k} H_{s_r}(z)
	\end{align}
	(technically \eqref{eqn:conf_block_computation_restated} doesn't imply \eqref{eqn:I_r^mod_computation} when $C_r = 0$, but the proof of Proposition~\ref{prop:conf_block_computation} gives \eqref{eqn:I_r^mod_computation}).
	Let $\widecheck{\Psi}_{\vec{z}}$ be defined by the same formula \eqref{eqn:Psi_z_def} but with $z_1,z_2$ switched.
	The cross-ratio of $z_2,z_1,z_3,z_4$ is $\frac{z}{z-1}$, so \eqref{eqn:I_r^mod_computation} is equivalent to
	\begin{align} \label{eqn:I_r^mod_check_computation}
		I_r^{\Mod}(\widecheck{\Psi}_{\vec{z}})
		= \Big(\frac{z}{z-1}\Big)^{2k} H_{s_r}\Big(\frac{z}{z-1}\Big).
	\end{align}
	Since multiplication is commutative and the exponent $2k$ is even,
	\begin{align} \label{eqn:pointwise_equality}
		\Psi_{\vec{z}}|_{\Gamma \backslash G}
		= \widecheck{\Psi}_{\vec{z}}|_{\Gamma \backslash G}.
	\end{align}
	Thus, applying Proposition~\ref{prop:crossing_ineq} with $\Psi = \Psi_{\vec{z}}$ and $\widecheck{\Psi} = \widecheck{\Psi}_{\vec{z}}$, we get \eqref{eqn:crossing_ineq} with equality because \eqref{eqn:pointwise_equality} is an equality.
	Using \eqref{eqn:I_tilde_def} to rewrite the summands in \eqref{eqn:crossing_ineq} as
	\begin{align*}
		W(s_r) \widetilde{C}_r^2
		= I_r^{\Mod}(\Psi) C_r^2
		\qquad \text{and} \qquad
		\widecheck{W}(s_r) \widetilde{C}_r^2
		= I_r^{\Mod}(\widecheck{\Psi}) C_r^2,
	\end{align*}
	and then inserting \eqref{eqn:I_r^mod_computation} and \eqref{eqn:I_r^mod_check_computation}, we obtain the crossing equation \eqref{eqn:crossing} from Corollary~\ref{cor:crossing}.
	Now, the averaged crossing equation which we will use to prove the Weyl bound, namely Corollary~\ref{cor:crossing_averaged_specific}, was derived by averaging Corollary~\ref{cor:crossing} over $z$.
	Thus Corollary~\ref{cor:crossing_averaged_specific} can be obtained from Proposition~\ref{prop:crossing_ineq} by taking $\Psi,\widecheck{\Psi}$ to be suitable weighted averages of $\Psi_{\vec{z}}, \widecheck{\Psi}_{\vec{z}}$ over $\vec{z}$, with the same weight. More precisely, there exists a measure $\nu$, such that Proposition~\ref{prop:crossing_ineq} specializes to Corollary~\ref{cor:crossing_averaged_specific} when
	\begin{align*}
		\Psi = \int \Psi_{\vec{z}} \, d\nu(\pvec{z})
		\qquad \text{and} \qquad
		\widecheck{\Psi}
		= \int \widecheck{\Psi}_{\vec{z}} \, d\nu(\pvec{z}).
	\end{align*}
	Here the pushforward of $\nu$ via the cross-ratio map $\vec{z} \mapsto z$ is related to the measure \eqref{eqn:mu_def} by the same change of variable used to get from Corollary~\ref{cor:crossing} to Corollary~\ref{cor:crossing_2}.
	
	In summary, our choice of test vector $\Psi$ is a certain weighted average of the pure tensors \eqref{eqn:Psi_z_def}, and our choice of $\widecheck{\Psi}$ is given by switching the first and second coordinates in $\Psi$.
	The weight with respect to which we average was chosen carefully in Section~\ref{sec:mu_def} (in slightly different language), to make \eqref{eqn:W_lower_bd} sharp.
	
	In contrast, with Bernstein and Reznikov's choice of $\Psi$, \eqref{eqn:W_lower_bd} is not sharp.
	This is the source of their loss in estimating \eqref{eqn:to_estimate_for_us_and_BR}; the rest of their argument is sharp.
	They build $\Psi, \widecheck{\Psi}$ as follows.
	For each $n \in \Z_{\geq 0}$, let $f_n \in \pi$ be an element of weight $2n$ with $\|f_n\|_{L^2(\Gamma \backslash G)} = 1$. Fix some $n$ such that $4n = T+O(1)$, and let
	\begin{align*}
		\Psi = f_n \otimes f_n \otimes \overline{f_n} \otimes \overline{f_n}
	\end{align*}
	and
	\begin{align*}
		\widecheck{\Psi}
		= f_n \otimes f_n \otimes \overline{f_n} \otimes \overline{f_n}
		+ 2 f_n \otimes f_{n+1} \otimes \overline{f_n} \otimes \overline{f_{n+1}}
		+ f_{n+1} \otimes f_{n+1} \otimes \overline{f_{n+1}} \otimes \overline{f_{n+1}}.
	\end{align*}
	Then
	\begin{align*}
		\Psi|_{\Gamma \backslash G}
		= |f_n|^4
		\qquad \text{and} \qquad
		\widecheck{\Psi}|_{\Gamma \backslash G}
		= (|f_n|^2 + |f_{n+1}|^2)^2.
	\end{align*}
	In particular, the pointwise inequality $\Psi|_{\Gamma \backslash G} \leq \widecheck{\Psi}|_{\Gamma \backslash G}$ clearly holds, so $\Psi,\widecheck{\Psi}$ can be used as test vectors in Proposition~\ref{prop:crossing_ineq}.
	One has
	\begin{align} \label{eqn:I_r^mod_BR}
		I_r^{\Mod}(\Psi)
		= \|P_r(|f_n|^2)\|_{L^2(\Gamma \backslash G)}^2
		\qquad \text{and} \qquad
		I_r^{\Mod}(\widecheck{\Psi})
		= \|P_r(|f_n|^2) + P_r(|f_{n+1}|^2)\|_{L^2(\Gamma \backslash G)}^2.
	\end{align}
	Define $W,\widecheck{W}$ as in Proposition~\ref{prop:crossing_ineq}.
	Then by \eqref{eqn:I_r^mod_BR}, both $W$ and $\widecheck{W}$ are nonnegative. The nonnegativity of $W$ is a good sign, because \eqref{eqn:W_lower_bd} is supposed to hold.
	The nonnegativity of $\widecheck{W}$ is a bad sign, because it indicates that \eqref{eqn:W_lower_bd} is not sharp. Indeed, if \eqref{eqn:W_lower_bd} were sharp, then by the property (4) in Section~\ref{sec:strategy_heuristics}, we would expect $\widecheck{W}$ to be oscillatory.
	Bernstein and Reznikov compute an asymptotic for $W$: in our notation, \cite[Proposition~9.1]{Bernstein--Reznikov_10} implies that
	\begin{align} \label{eqn:W_BR_asymptotic}
		W(s_r)
		\approx \1_{t_r \lesssim T} \, t_r^{-5/3} \Ai\Big(\frac{4n-t_r}{t_r^{1/3}}\Big)^2,
	\end{align}
	where $\Ai(\cdot)^2$ is the square of the Airy function. Technically this doesn't satisfy \eqref{eqn:W_lower_bd}, but it does after renormalizing by the constant $T^{\frac{5}{3}}$, which is all that matters; if one wishes, one can multiply $\Psi, \widecheck{\Psi}$ by $T^{\frac{5}{3}}$ to make \eqref{eqn:W_lower_bd} hold on the nose. We shall ignore this trivial issue.
	If $\Ai(\cdot)^2$ were essentially a bump function at the origin, then \eqref{eqn:W_lower_bd} would be sharp. However, we don't expect \eqref{eqn:W_lower_bd} to be sharp, and indeed it is not --- $\Ai(x)^2$ has a tail as $x \to -\infty$ which decays like $(1-x)^{-\frac{1}{2}}$. This tail is visible in the following plot of $\Ai(\cdot)^2$.
	\begin{figure}[H]
		\centering
		\includegraphics[width=0.7\textwidth]{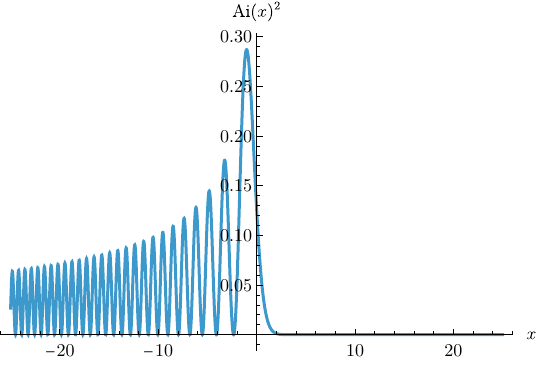}
	\end{figure}
	
	Consequently, Bernstein and Reznikov's bound of $T^{\frac{5}{3}}$ is not Lindel\"of on average for \eqref{eqn:to_estimate_for_us_and_BR}, but rather for
	\begin{align*}
		\sum_{T \leq t_r \lesssim T} \Big(1 + \frac{t_r-T}{T^{\frac{1}{3}}}\Big)^{-\frac{1}{2}} \widetilde{C}_r^2
	\end{align*}
	(recall that in \eqref{eqn:W_BR_asymptotic}, $n$ satisfies $4n = T+O(1)$).
	Indeed, assuming Lindel\"of, the dominant contribution to this sum is from terms with $t_r \in [2T,3T]$, say, in which case the coefficient in front of $\widetilde{C}_r^2$ is $\sim T^{-\frac{1}{3}}$. There are $\sim T^2$ many such terms, so the Lindel\"of-on-average bound is $T^2 T^{-\frac{1}{3}} = T^{\frac{5}{3}}$.
	
	To conclude, the reason we improve on \cite{Bernstein--Reznikov_10} is that our choice of $\Psi$ makes \eqref{eqn:W_lower_bd} sharp.
	
	\begin{rem}[Non-factorizable test vectors] \label{rem:non-factor}
		In order to make \eqref{eqn:W_lower_bd} sharp, we have to take $\Psi$ to be quite complicated. Specifically, our $\Psi \in \pi \otimes \pi \otimes \overline{\pi} \otimes \overline{\pi}$ is not a finite linear combination of pure tensors, but rather a weighted average over a continuous family of the pure tensors \eqref{eqn:Psi_z_def}.
		The possibility of using non-factorizable test vectors to go beyond \cite{Bernstein--Reznikov_10} is discussed in \cite[Section~1.3]{Blomer--Jana--Nelson}, and at the top of p.1179 they propose a class of potential test vectors.
		For $\Psi$ in their class, $W$ is a square and hence nonnegative, which makes \eqref{eqn:W_lower_bd} easier to verify.
		Our $\Psi$ lies even outside of their class, and indeed our $W$ is not nonnegative.
		However, we are able to show that $W$ is bounded below by a small negative constant (see Proposition~\ref{prop:I_u_principal} and Figure~\ref{fig:I_u}).
		This is why we say after \eqref{eqn:W_lower_bd} that we allow a small additive error.
		In \cite{Reznikov_unfolding}, Reznikov used a method similar to \cite{Bernstein--Reznikov_10} to give a purely analytic proof of the Weyl bound for a certain family of $L$-functions of degree 4 rather than 8.
		In this case, the test vector $\Psi$ is an element of $\pi \otimes \pi$ instead of $\pi \otimes \pi \otimes \pi \otimes \pi$ (here $\pi = \overline{\pi}$ because $f$ is Maass in \cite{Reznikov_unfolding} and \cite{Bernstein--Reznikov_10}).
		The problem of choosing $\Psi$ to make \eqref{eqn:W_lower_bd} sharp is much easier in \cite{Reznikov_unfolding} --- it amounts to inverting the Fourier transform --- but nevertheless as Reznikov points out, $\Psi$ still ends up being rather complicated. In particular, $\Psi$ is not a finite linear combination of pure tensors of weight vectors. However, $\Psi$ can be presented as a pure tensor of two non-$K$-finite vectors.
		This is in contrast with our situation, where $\Psi$ genuinely does not factor.
	\end{rem}
	
	\section{Proofs of convexity and Weyl bounds assuming results below} \label{sec:main_proofs}
	
	In this section we prove Proposition~\ref{prop:convexity} (the convexity bound) and Theorem~\ref{thm:Weyl} (the Weyl bound) assuming certain estimates for the conformal blocks and for $W, \widecheck{W}$ (as defined by \eqref{eqn:final_W_def} and \eqref{eqn:final_W_check_def}). These estimates will be proved in Sections~\ref{sec:complementary}--\ref{sec:t_averaging}.
	
	The convexity bound implies that $\widetilde{C}_r$ has polynomial growth. Before we can prove this, we first need to show that $\widetilde{C}_r$ has subexponential growth.
	
	\begin{prop} \label{prop:tildeC_r_subexp}
		Let $\delta > 0$ be arbitrarily small. Then
		\begin{align*}
			\widetilde{C}_r^2
			\lesssim_{\delta} e^{\delta t_r}.
		\end{align*}
	\end{prop}
	
	\begin{proof}
		Corollary~\ref{cor:t_block_lower_bd} below says that there exists $z \in \C$ depending on $\delta$, with $\re z > 0$, such that for $r$ sufficiently large,
		\begin{align*}
			|\widetilde{H}_{s_r}(1-z^{-2})|
			\gtrsim_{\delta} e^{-\delta t_r}.
		\end{align*}
		On the other hand, we know that the right hand side of the crossing equation \eqref{eqn:crossing_2} converges. This forces $\widetilde{C}_r^2 \lesssim_{\delta} e^{\delta t_r}$.
	\end{proof}
	
	\subsection{Convexity bound}
	\label{subsec:convexity}
	
	To prove the convexity bound, we apply the crossing equation \eqref{eqn:crossing_2} with $z = T^{-1}$. We then have the following estimates for the conformal blocks.
	
	\begin{prop}[Lower bound for u-channel conformal block at $z=T^{-1}$] \label{prop:u_block_at_T^-1}
		Let $s \in (\frac{1}{2},1] \cup (\frac{1}{2}+i\R_{\geq 0})$, and write $s = \sigma+it$. Then
		\begin{align} \label{eqn:u_block_at_T^-1}
			\widetilde{H}_s(1-T^{-2})
			\gtrsim T^{4k-2} \1_{t \sim T}.
		\end{align}
	\end{prop}
	
	\begin{proof}
		It follows from \eqref{eqn:2F1_series} that the Taylor series for $\widetilde{H}_s$ at the origin has nonnegative coefficients. It also has radius of convergence $1$ (so it converges at $1-T^{-2}$).
		Therefore $\widetilde{H}_s(1-T^{-2}) \geq 0$.
		When $t \sim T$, Corollary~\ref{cor:u_block_size_z>0} says that
		\begin{align*}
			\widetilde{H}_s(1-T^{-2})
			\sim T^{4k-2}.
		\end{align*}
		Thus \eqref{eqn:u_block_at_T^-1} holds.
	\end{proof}
	
	\begin{prop}[Upper bound for t-channel conformal block at $z=T^{-1}$] \label{prop:t_block_at_T^-1}
		There is a constant $c \gtrsim 1$, such that the following holds.
		Let $s \in (\frac{1}{2},1] \cup (\frac{1}{2}+i\R_{\geq 0})$, and write $s = \sigma+it$. Then
		\begin{align*}
			|\widetilde{H}_s(1-T^2)|
			\lessapprox T^{-2(1-\sigma)} e^{-ct}.
		\end{align*}
	\end{prop}
	
	\begin{proof}
		This follows from Proposition~\ref{prop:t_complementary} when $t \leq 1$, and Corollary~\ref{cor:t-block_triv_bd} when $t > 1$.
	\end{proof}
	
	From Proposition~\ref{prop:t_block_at_T^-1}, we get an asymptotic for the t-channel in \eqref{eqn:crossing_2} at $z=T^{-1}$:
	
	\begin{prop} \label{prop:convexity_asymptotic}
		We have
		\begin{align} \label{eqn:convexity_asymptotic}
			T^{-2k} \sum_{r=0}^{\infty} \widetilde{C}_r^2 \widetilde{H}_{s_r}(1-T^{-2})
			= T^{2k}(1 + \widetilde{O}(T^{-2(1-\sigma_1)})).
		\end{align}
	\end{prop}
	
	\begin{proof}
		Taking $z=T^{-1}$ in \eqref{eqn:crossing_2}, and separating the $r=0$ term in the t-channel from the rest,
		\begin{align} \label{eqn:crossing_z=1/T}
			T^{-2k} \sum_{r=0}^{\infty} \widetilde{C}_r^2 \widetilde{H}_{s_r}(1-T^{-2})
			= T^{2k} \Big(1 + \sum_{r=1}^{\infty} \widetilde{C}_r^2 \widetilde{H}_{s_r}(1-T^2)\Big).
		\end{align}
		By Proposition~\ref{prop:t_block_at_T^-1}, there exists $c \gtrsim 1$, such that
		\begin{align*}
			|\widetilde{H}_{s_r}(1-T^2)|
			\lessapprox T^{-2(1-\sigma_r)} e^{-ct_r}.
		\end{align*}
		By Weyl's law and the subexponential growth of the $\widetilde{C}_r$ (Proposition~\ref{prop:tildeC_r_subexp}), we deduce that the sum on the right hand side of \eqref{eqn:crossing_z=1/T} is $\widetilde{O}(T^{-2(1-\sigma_1}))$.
	\end{proof}
	
	The convexity bound easily follows:
	
	\begin{proof}[Proof of Proposition~\ref{prop:convexity}]
		By the lower bound in Proposition~\ref{prop:u_block_at_T^-1}, followed by Proposition~\ref{prop:convexity_asymptotic},
		\begin{align*}
			\sum_{t_r \sim T} \widetilde{C}_r^2
			\lesssim T^{-4k+2} \sum_{r=0}^{\infty} \widetilde{C}_r^2 \widetilde{H}_{s_r}(1-T^{-2})
			\sim T^2.
		\end{align*}
		Since this holds for all $T$, we get the convexity bound \eqref{eqn:convexity} by dyadic decomposition.
	\end{proof}
	
	\subsection{Weyl bound}
	
	For the Weyl bound, we use the averaged crossing equation in Corollary~\ref{cor:crossing_averaged_specific}, rather than \eqref{eqn:crossing_2}.
	\textit{So from now on, let $W,\widecheck{W}$ be defined by \eqref{eqn:final_W_def} and \eqref{eqn:final_W_check_def}.
		Whenever the variable $z$ is used in the context of \eqref{eqn:final_W_def} and \eqref{eqn:final_W_check_def}, we understand that $z$ is of the form \eqref{eqn:y_range}.}
	
	When $t_r \leq 1$, trivial bounds for $W(s_r)$ and $\widecheck{W}(s_r)$ using Propositions~\ref{prop:u_complementary} and \ref{prop:t_complementary} will suffice.
	
	\begin{prop}[Trivial bound for $W$ when $t \leq 1$] \label{prop:I_u_complementary}
		Let $s \in (\frac{1}{2},1] \cup (\frac{1}{2}+i\R_{\geq 0})$, and write $s = \sigma+it$. Assume $t \leq 1$. Then
		\begin{align*}
			|W(s)|
			\lessapprox T^{-4k+\frac{5}{2}} H^{-\frac{1}{2}}.
		\end{align*}
	\end{prop}
	
	\begin{proof}
		Taking absolute values inside the integral defining $W$, and then using Proposition~\ref{prop:u_complementary},
		\begin{align*}
			|W(s)|
			&\lessapprox T^{-4k+\frac{5}{2}} H \int_{-T^{\varepsilon}/H}^{T^{\varepsilon}/H} e^{-(Hy)^2} |z|^{\frac{1}{2}} \, dy
			\lesssim T^{-4k+\frac{5}{2}} H^{-\frac{1}{2}}.
			\qedhere
		\end{align*}
	\end{proof}
	
	\begin{prop}[Trivial bound for $\widecheck{W}$ when $t \leq 1$] \label{prop:I_t_complementary}
		Let $s \in (\frac{1}{2},1] \cup (\frac{1}{2}+i\R_{\geq 0})$, and write $s = \sigma+it$. Assume $t \leq 1$. Then
		\begin{align} \label{eqn:I_t_complemenetary}
			|\widecheck{W}(s)|
			\lessapprox TH.
		\end{align}
	\end{prop}
	
	\begin{proof}
		Taking absolute values inside the integral defining $\widecheck{W}$, and then using Proposition~\ref{prop:t_complementary},
		\begin{align*}
			|\widecheck{W}(s)|
			&\lessapprox T^{-4k+\frac{5}{2}} H \int_{\R} |z|^{-4k+\frac{1}{2}} \, dy
			\lesssim TH.
		\end{align*}
		The second inequality is because the dominant contribution to the integral $\int_{\R} |z|^{-4k+\frac{1}{2}} \, dy$ comes from $|y| \lesssim T^{-1}$.
	\end{proof}
	
	When $t_r > 1$, the following asymptotics for $W(s_r)$ and $\widecheck{W}(s_r)$ are proved in Sections~\ref{sec:u_averaging} and \ref{sec:t_averaging}, respectively. They are illustrated in Figures~\ref{fig:I_u} and \ref{fig:I_t} below.
	
	\begin{prop}[Asymptotic for $W$ when $t>1$] \label{prop:I_u_principal}
		Let $s = \frac{1}{2}+it$ with $t>1$. Then
		\begin{align} \label{eqn:I_u_asymptotic}
			W(s)
			= (1+\widetilde{O}(T^{-6\varepsilon})) \frac{1}{2} \Big(\frac{t}{T}\Big)^{4k-\frac{5}{2}} e^{-2t/T}  \exp\Big(-\Big(\frac{t-T}{H}\Big)^2\Big) + \widetilde{O}(\1_{t \lessapprox T} H/T + t^{-\infty}).
		\end{align}
		In particular,
		\begin{align} \label{eqn:I_2_bd_in_reduction}
			\1_{|t-T| \leq H}
			\lesssim W(s) + \widetilde{O}(\1_{t \lessapprox T} H/T + t^{-\infty}).
		\end{align}
	\end{prop}
	
	\begin{prop}[Asymptotic for $\widecheck{W}$ when $t>1$] \label{prop:I_t_principal}
		Let $s = \frac{1}{2}+it$ with $t>1$. Then
		\begin{align}
			\widecheck{W}(s)
			= e^{-2} H t^{-\frac{1}{2}} \exp\Big(-\Big(\frac{t}{T/H}\Big)^2\Big) \cos(\alpha(t))
			+ \widetilde{O}(\1_{t \lessapprox T/H} H t^{-1} + t^{-\infty}),
			\label{eqn:I_t_asymptotic}
		\end{align}
		where
		\begin{align} \label{eqn:alpha_def}
			\alpha(t)
			= \frac{3\pi}{4} + 2t\Big(1 + \log(2) - \log\Big(\frac{t}{T}\Big) + \frac{1}{4}\Big(\frac{t}{T}\Big)^2\Big).
		\end{align}
		In particular,
		\begin{align} \label{eqn:I_4_bd_in_reduction}
			|\widecheck{W}(s)|
			\lesssim \1_{t \lessapprox T/H} H t^{-\frac{1}{2}} + t^{-\infty}.
		\end{align}
	\end{prop}
	
	We will only need the inequalities \eqref{eqn:I_2_bd_in_reduction} and \eqref{eqn:I_4_bd_in_reduction} rather than the asymptotics \eqref{eqn:I_u_asymptotic} and \eqref{eqn:I_t_asymptotic}. However, the asymptotics require little extra work to prove, and they are useful as a sanity check, because they can be tested numerically --- see Figures~\ref{fig:I_u} and \ref{fig:I_t} below.
	They are also useful for heuristics.
	For example, the properties (3)--(5) of $\widecheck{W}$ in Section~\ref{sec:strategy_heuristics} follow from \eqref{eqn:I_t_asymptotic}, and the property (2) also partially follows from \eqref{eqn:I_t_asymptotic}.
	
	\begin{figure}[H]
		\centering
		\includegraphics[width=0.86\textwidth]{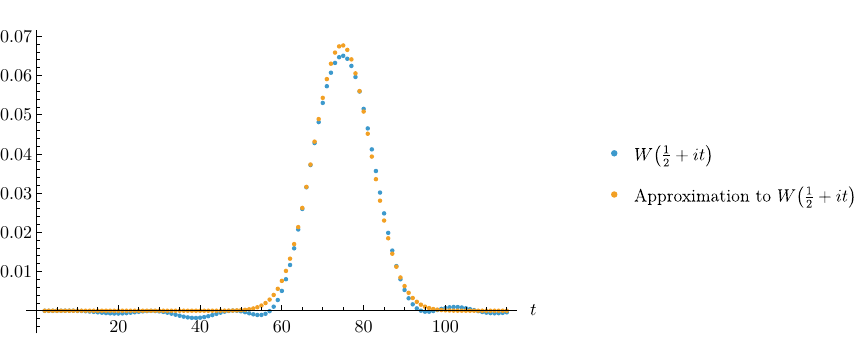}
		\caption{Take $k = 1$, $T = 75$, and $\varepsilon = \frac{1}{10}$. Then the points $(t, W(\frac{1}{2}+it))$, for each integer $t \in [2,115]$, are plotted in blue.
			The approximations to these points given by the asymptotic \eqref{eqn:I_u_asymptotic} are plotted in orange.}
		\label{fig:I_u}
	\end{figure}
	
	\begin{figure}[H]
		\centering
		\includegraphics[width=0.86\textwidth]{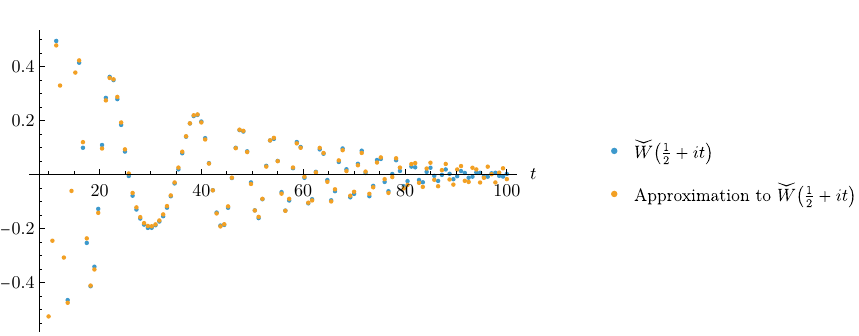}
		\caption{Take $k = 1$, $T = 1000$, and $\varepsilon = \frac{1}{40}$. Then the points $(t, \widecheck{W}(\frac{1}{2}+it))$, for $t \in [10,100]$ of the form $10 + \frac{3}{4}n$ with $n \in \Z$, are plotted in blue. The approximations to these points given by the asymptotic \eqref{eqn:I_t_asymptotic} are plotted in orange.
		}
		\label{fig:I_t}
	\end{figure}
	
	We now prove the Weyl bound, Theorem~\ref{thm:Weyl}.
	
	\begin{proof}[Proof of Theorem~\ref{thm:Weyl}]
		By \eqref{eqn:I_2_bd_in_reduction} in Proposition~\ref{prop:I_u_principal},
		\begin{align*}
			\sum_{|t_r-T| \leq H} \widetilde{C}_r^2
			= \sum_{t_r>1} \1_{|t_r-T| \leq H} \widetilde{C}_r^2
			\lesssim \sum_{t_r > 1} W(s_r) \widetilde{C}_r^2
			+ \widetilde{O}\Big(\frac{H}{T} \sum_{1 < t_r \lessapprox T} \widetilde{C}_r^2 + \sum_{t_r>1} t_r^{-\infty} \widetilde{C}_r^2\Big).
		\end{align*}
		By the convexity bound \eqref{eqn:convexity}, the error term is $\lessapprox TH$.
		Therefore
		\begin{align} \label{eqn:reduction:bdd_by_I_u}
			\sum_{|t_r-T| \leq H} \widetilde{C}_r^2
			\lessapprox \sum_{t_r > 1} W(s_r) \widetilde{C}_r^2 + TH.
		\end{align}
		By the averaged crossing equation in Corollary~\ref{cor:crossing_averaged_specific}, the sum on the right can be written as
		\begin{align*}
			\sum_{t_r > 1} W(s_r) \widetilde{C}_r^2
			&= \sum_{r=0}^{\infty} W(s_r) \widetilde{C}_r^2
			- \sum_{t_r \leq 1} W(s_r) \widetilde{C}_r^2
			\\&= \sum_{r=0}^{\infty} \widecheck{W}(s_r) \widetilde{C}_r^2
			- \sum_{t_r \leq 1} W(s_r) \widetilde{C}_r^2
			\\&= \sum_{t_r > 1} \widecheck{W}(s_r) \widetilde{C}_r^2
			+ \sum_{t_r \leq 1} (\widecheck{W}(s_r) - W(s_r)) \widetilde{C}_r^2.
		\end{align*}
		By Propositions~\ref{prop:I_u_complementary} and \ref{prop:I_t_complementary}, together with the fact that there are only $O(1)$ many $r$ with $t_r \leq 1$, the second sum on the right is $\widetilde{O}(TH)$.
		So
		\begin{align*}
			\sum_{t_r > 1} W(s_r) \widetilde{C}_r^2
			= \sum_{t_r > 1} \widecheck{W}(s_r) \widetilde{C}_r^2 + \widetilde{O}(TH).
		\end{align*}
		Inserting this into \eqref{eqn:reduction:bdd_by_I_u},
		\begin{align} \label{eqn:reduction:bdd_by_I_t}
			\sum_{|t_r-T| \leq H} \widetilde{C}_r^2
			\lessapprox \sum_{t_r > 1} \widecheck{W}(s_r) \widetilde{C}_r^2 + TH.
		\end{align}
		By \eqref{eqn:I_4_bd_in_reduction} in Proposition~\ref{prop:I_t_principal}, followed by dyadic decomposition and the convexity bound \eqref{eqn:convexity}, the sum on the right satisfies
		\begin{align}
			\Big|\sum_{t_r > 1} \widecheck{W}(s_r) \widetilde{C}_r^2\Big|
			&\leq \sum_{t_r > 1} |\widecheck{W}(s_r)| \widetilde{C}_r^2
			\notag
			\\&\lesssim H \sum_{1<t_r \lessapprox T/H} t_r^{-\frac{1}{2}} \widetilde{C}_r^2
			+ \sum_{t_r > 1} t_r^{-\infty} \widetilde{C}_r^2
			\notag
			\\&\lesssim H \sum_{\substack{N \in 2^{\mathbf{N}} \\ N \lessapprox T/H}} N^{-\frac{1}{2}} \sum_{t_r \sim N} \widetilde{C}_r^2
			+ 1
			\notag
			\\&\lesssim H \sum_{\substack{N \in 2^{\mathbf{N}} \\ N \lessapprox T/H}} N^{\frac{3}{2}}
			\notag
			\\&\lessapprox H(T/H)^{\frac{3}{2}}
			= T^{\frac{3}{2}} H^{-\frac{1}{2}}.
			\label{eqn:t_principal_contribution}
		\end{align}
		Plugging this into \eqref{eqn:reduction:bdd_by_I_t} gives
		\begin{align} \label{eqn:t_contribution+Lindelof}
			\sum_{|t_r-T| \leq H} \widetilde{C}_r^2
			\lessapprox T^{\frac{3}{2}} H^{-\frac{1}{2}} + TH.
		\end{align}
		Recall that $H = T^{\frac{1}{3}+2\varepsilon}$.
		Since $H \geq T^{\frac{1}{3}}$, we have $T^{\frac{3}{2}} H^{-\frac{1}{2}} \leq TH$. Since $\varepsilon$ was arbitrary, we conclude the Weyl bound \eqref{eqn:Weyl}.
	\end{proof}
	
	\begin{rem}[Weyl is a barrier] \label{rem:Weyl_barrier}
		If $H$ were smaller than $T^{\frac{1}{3}}$, then $\RHS\eqref{eqn:t_principal_contribution} = T^{\frac{3}{2}} H^{-\frac{1}{2}}$ would be bigger than $TH$, and \eqref{eqn:t_contribution+Lindelof} would be weaker than Lindel\"of on average.
		Therefore, to get a sub-Weyl bound by our method, we would need to obtain cancellation in $\LHS\eqref{eqn:t_principal_contribution}$. This is exactly as predicted in Section~\ref{sec:strategy_heuristics}.
		Indeed, by Proposition~\ref{prop:I_t_principal}, $\LHS\eqref{eqn:t_principal_contribution}$ is essentially a sum over $t_r \lesssim T/H$, and if we approximate this sum by the top dyadic range, then
		\begin{align*}
			\LHS\eqref{eqn:t_principal_contribution}
			\approx \Big|\sum_{t_r \sim T/H} \widecheck{W}(s_r) \widetilde{C}_r^2 \Big|
			= E_{\true},
		\end{align*}
		where the last equality uses the notation \eqref{eqn:X_E_def} from Section~\ref{sec:strategy_heuristics}.
		It seems difficult to prove a nontrivial bound for this sum (without using the crossing equation which would put us back where we started).
		One reason is that $\widecheck{W}(\tfrac{1}{2}+it)$ oscillates quickly for $t \sim T/H$. Specifically, by \eqref{eqn:I_t_asymptotic} in Proposition~\ref{prop:I_t_principal}, it oscillates at frequency $\alpha'(t) \sim \log T$, where $\alpha(t)$ is as in \eqref{eqn:alpha_def}.
		For comparison, the cutoff $\1_{|t-T| \lesssim T^{\frac{1}{3}}}$ oscillates at frequency $\sim T^{-\frac{1}{3}}$, far less than $\log T$. However, it already tests the limits of our abilities to sum this cutoff against $\widetilde{C}_r^2$. Another sign that the frequency $\log T$ is too fast to handle is that the distribution of the spectral parameters $t_r$ is not known at the dual scale $\frac{1}{\log T}$. For example, it is conjectured but not known that, for every $c > 0$ and $T \gg_c 1$, there exists $r$ with $t_r \in [T, T+\frac{c}{\log T}]$. When $c$ is a sufficiently large constant, this is true by \cite{Randol} (in constant negative curvature) or \cite{Berard} (in variable negative curvature).
	\end{rem}
	
	\section{Trivial bounds for conformal blocks when $t \leq 1$}
	\label{sec:complementary}
	
	Let $s \in (\frac{1}{2},1] \cup (\frac{1}{2}+i\R_{\geq 0})$. Write $s = \sigma+it$. The next two propositions are proved in Subsections~\ref{subsec:complementary:u} and \ref{subsec:complementary:t}, respectively.
	
	\begin{prop}[Trivial bound for u-channel conformal block when $t \leq 1$] \label{prop:u_complementary}
		Assume $t \leq 1$.
		Let $z \in \C$ with $\re z > 0$ and $|z| \ll 1$. Then
		\begin{align} \label{eqn:u_complementary_bd}
			|H_s(1-z^2)|
			\lesssim \log(|z|^{-1}).
		\end{align}
	\end{prop}
	
	\begin{prop}[Trivial bound for t-channel conformal block when $t \leq 1$] \label{prop:t_complementary}
		Assume $t \leq 1$.
		Let $z \in \C$ with $\re z > 0$ and $|z| \ll 1$. Then
		\begin{align} \label{eqn:t_complementary_bd}
			|H_s(1-z^{-2})|
			\lesssim |z|^{2(1-\sigma)} \log(|z|^{-1})
			\leq \log(|z|^{-1}).
		\end{align}
	\end{prop}
	
	Assume in the remainder of this section that $t \leq 1$. This assumption will often be used to estimate $|\zeta^s| \sim |\zeta|^{\sigma}$ for $\zeta \in \C \setminus (-\infty,0]$.
	
	\subsection{u-Channel} \label{subsec:complementary:u}
	
	Here we prove Proposition~\ref{prop:u_complementary}.
	This is trivial when $s=1$ because $H_0$ is identically $1$, so we may assume $\sigma < 1$.
	Then by Euler's integral representation \cite[(9.1.6)]{Lebedev} for the hypergeometric ${}_2F_1$,
	\begin{align*}
		H_s(1-z^2)
		= {}_2F_1(1-s,s;1;1-z^2)
		= \frac{1}{\Gamma(s) \Gamma(1-s)} \int_{0}^{1} x^{s-1} (1-x)^{-s} (1 - x(1-z^2))^{s-1} \, dx.
	\end{align*}
	Replacing $x$ with $1-x$ and using Euler's reflection formula for the Gamma function,
	\begin{align} \label{eqn:Euler_integral_rep_u}
		H_s(1-z^2)
		= \frac{\sin(\pi s)}{\pi} \int_{0}^{1} (1-x)^{s-1} x^{-s} (x + z^2(1-x))^{s-1} \, dx.
	\end{align}
	Let $I(x)$ denote the integrand, so
	\begin{align*}
		I(x)
		= (1-x)^{s-1} x^{-s} (x+z^2(1-x))^{s-1}.
	\end{align*}
	Bearing in mind that $|z| \ll 1$, we trivially have
	\begin{align} \label{eqn:I_half_to_1_u}
		\int_{\frac{1}{2}}^{1} |I(x)| \, dx
		\lesssim \int_0^1 (1-x)^{\sigma-1} \, dx
		\lesssim 1,
	\end{align}
	where the second inequality is because $1-\sigma \leq \frac{1}{2}$.
	Write $z = re^{i\theta}$ in polar coordinates. Then
	\begin{align*}
		\int_{0}^{\frac{1}{2}} |I(x)| \, dx
		\lesssim \int_{0}^{\frac{1}{2}} x^{-\sigma} |x+r^2e^{2i\theta}(1-x)|^{\sigma-1} \, dx.
	\end{align*}
	Making the substitution $x = r^2y$,
	\begin{align*}
		\int_{0}^{\frac{1}{2}} |I(x)| \, dx
		&\lesssim r^2 \int_{0}^{\frac{1}{2}r^{-2}} (r^2y)^{-\sigma} |r^2y+r^2e^{2i\theta}(1-r^2y)|^{\sigma-1} \, dy
		\\&= \int_{0}^{\frac{1}{2}r^{-2}} y^{-\sigma} |y+e^{2i\theta}(1-r^2y)|^{\sigma-1} \, dy.
	\end{align*}
	Rearranging the expression in absolute values,
	\begin{align*}
		\int_{0}^{\frac{1}{2}} |I(x)| \, dx
		\lesssim \int_{0}^{\frac{1}{2}r^{-2}} y^{-\sigma} |e^{2i\theta} + (1-r^2e^{2i\theta})y|^{\sigma-1} \, dy.
	\end{align*}
	Since $r \ll 1$, the absolute value on the right is $\gtrsim 1$ when $y \leq \frac{1}{2}$ and $\gtrsim y$ when $y \geq 2$. Therefore
	\begin{align*}
		\int_{0}^{\frac{1}{2}} |I(x)| \, dx
		&\lesssim \int_{0}^{\frac{1}{2}} y^{-\sigma} \, dy + \int_{\frac{1}{2}}^{2} |e^{2i\theta}
		+ (1-r^2e^{2i\theta})y|^{\sigma-1} \, dy
		+ \int_{2}^{\frac{1}{2}r^{-2}} y^{-1} \, dy
		\\&\lesssim \frac{1}{1-\sigma} + \int_{\frac{1}{2}}^{2} |e^{2i\theta}
		+ (1-r^2e^{2i\theta})y|^{\sigma-1} \, dy
		+ \log(r^{-1}).
	\end{align*}
	Using the inequality $|\zeta| \geq |\re \zeta|$ and then making the substitution $u = (1-r^2\cos(2\theta))y$,
	\begin{align*}
		\int_{0}^{\frac{1}{2}} |I(x)| \, dx
		\lesssim \frac{1}{1-\sigma} + \log(r^{-1}) + \int_{u \sim 1} |\cos(2\theta) + u|^{\sigma-1} \, du.
	\end{align*}
	The integral on the right hand side is $\lesssim 1$ because $1-\sigma \leq \frac{1}{2}$, so
	\begin{align} \label{eqn:I_0_to_half_u}
		\int_{0}^{\frac{1}{2}} |I(x)| \, dx
		\lesssim \frac{1}{1-\sigma} + \log(r^{-1}).
	\end{align}
	Inserting \eqref{eqn:I_half_to_1_u} and \eqref{eqn:I_0_to_half_u} into \eqref{eqn:Euler_integral_rep_u},
	\begin{align*}
		|H_s(1-z^2)|
		\lesssim \frac{|\sin(\pi s)|}{1-\sigma} + |\sin(\pi s)| \log(r^{-1}).
	\end{align*}
	In the range of $s$ that we are considering, we have $|\sin(\pi s)| \lesssim 1-\sigma \leq \frac{1}{2}$. Thus
	\begin{align*}
		|H_s(1-z^2)|
		\lesssim \log(r^{-1}).
	\end{align*}
	This is the desired bound \eqref{eqn:u_complementary_bd}.
	\qed
	
	\subsection{t-Channel} \label{subsec:complementary:t}
	
	The proof of Proposition~\ref{prop:t_complementary} is closely analogous to that of Proposition~\ref{prop:u_complementary}. Again, the result is trivial when $s=1$, so assume $\sigma < 1$. Then Euler's integral representation \cite[(9.1.6)]{Lebedev} gives
	\begin{align*}
		H_s(1-z^{-2})
		= {}_2F_1(1-s;s;1;1-z^{-2})
		= \frac{1}{\Gamma(s) \Gamma(1-s)} \int_{0}^{1} x^{s-1} (1-x)^{-s} (1 - x(1-z^{-2}))^{s-1} \, dx.
	\end{align*}
	Pulling a factor of $z^{2(1-s)}$ out of the integrand and using Euler's reflection formula,
	\begin{align*}
		H_s(1-z^{-2})
		= z^{2(1-s)} \frac{\sin(\pi s)}{\pi} \int_0^1 x^{s-1} (1-x)^{-s} (x + z^2(1-x))^{s-1} \, dx.
	\end{align*}
	The only differences between this and \eqref{eqn:Euler_integral_rep_u} are the factor of $z^{2(1-s)}$ in front and the fact that $x$ and $1-x$ have switched places in the first two factors in the integrand. Estimating the integral in the same way as \eqref{eqn:Euler_integral_rep_u} yields
	\begin{align*}
		\Big|\frac{\sin(\pi s)}{\pi} \int_0^1 x^{s-1} (1-x)^{-s} (x + z^2(1-x))^{s-1} \, dx\Big|
		\lesssim \log(|z|^{-1}),
	\end{align*}
	from which the desired bound \eqref{eqn:t_complementary_bd} follows.
	\qed
	
	\section{Asymptotics for conformal blocks when $t>1$}
	\label{sec:asymptotics}
	
	Suppose $r \in \Z_{\geq 0}$ is such that $t_r>1$ (or more generally $t_r>0$). Then $\sigma_r = \frac{1}{2}$, so $t_r$ determines $s_r = \frac{1}{2}+it_r$. Therefore, in this section, the variable $t$ should be thought of as a spectral parameter, and we free the variables $s$ and $\sigma$ so that they have no special meaning.
	
	The two propositions below are the most precise asymptotics for the u- and t-channel conformal blocks that we will use. They are proven in Subsections~\ref{subsec:u_asymptotic} and \ref{subsec:t_asymptotic}, respectively. Along the way, we will establish some more crude but useful estimates.
	
	\begin{prop}[Asymptotic for u-channel conformal block when $t>1$] \label{prop:u_block_asymptotic}
		There exist constants $c_j \in \Q$ with $c_0 = 1$, such that the following holds. Let $t > 1$. Let $z \in \C$ with $t^{-1} \lesssim |z| \lesssim 1$ and $\re z \gtrsim t^{-\frac{2}{3}+\varepsilon} |z|$. Then for each $M \geq 0$, there exists $J \in \Z_{\geq 0}$ depending only on $M$ and $\varepsilon$, such that
		\begin{align} \label{eqn:u_block_asymptotic}
			\widetilde{H}_{\frac{1}{2}+it}(1-z^2)
			= \frac{t^{4k-\frac{5}{2}}}{2\sqrt{\pi z}} \Big[e^{-2tz} \sum_{j=0}^{J} c_j (tz^3)^j + O_M(e^{-2t\re z} |tz|^{-1} + t^{-M})\Big].
		\end{align}
	\end{prop}
	
	Recall that implicit constants are allowed to depend on $\varepsilon$.
	Inspecting the proof of Proposition~\ref{prop:u_block_asymptotic} shows that one can take
	\begin{align*}
		c_j = \frac{(-1)^j}{3^j(j!)}.
	\end{align*}
	At first, it may seem odd that \eqref{eqn:u_block_asymptotic} is an asymptotic expansion in powers of $tz^3$, but that we never assume $tz^3$ is small.
	However, if $|tz^3| \gtrsim 1$, then the hypothesis $\re z \gtrsim t^{-\frac{2}{3}+\varepsilon}|z|$ forces
	\begin{align*}
		\re tz
		\gtrsim t^{\frac{1}{3}+\varepsilon}|z|
		= t^{\varepsilon} |tz^3|^{\frac{1}{3}}
		\gtrsim t^{\varepsilon},
	\end{align*}
	and so the exponential decay of the factor $e^{-2tz}$ dominates the polynomial growth of the sum on the right hand side of \eqref{eqn:u_block_asymptotic}.
	
	\begin{prop}[Asymptotic for t-channel conformal block when $t>1$] \label{prop:t_block_asymptotic}
		Let $t > 1$. Let $z \in \C$ with $\re z > 0$, $\im z \geq 0$, and $|z| \ll 1$. Denote
		\begin{align*}
			\phi(z)
			= -\log(1+\sqrt{1-z^2})
			\qquad \text{and} \qquad
			\psi(z)
			= \frac{\pi^{-\frac{1}{2}} e^{-i\pi/4}}{1+\sqrt{1-z^2}}.
		\end{align*}
		Then
		\begin{align} \label{eqn:t_block_asymptotic}
			\widetilde{H}_{\frac{1}{2}+it}(1-z^{-2})
			= t^{4k-\frac{5}{2}} z \psi(z) e^{-2it[\log(-iz) + \phi(z)]} (1 + O(t^{-1} + |z|^4)) + O(t^{4k-\frac{5}{2}} |z|e^{-\pi t}).
		\end{align}
	\end{prop}
	
	Assume in the remainder of this section that $t>1$. Then by the definition \eqref{eqn:tilde_H_s_def},
	\begin{align*}
		\widetilde{H}_{\frac{1}{2}+it}
		= t^{4k-2} \, e^{-\pi t} H_{\frac{1}{2}+it}.
	\end{align*}
	
	\subsection{u-Channel} \label{subsec:u_asymptotic}
	
	Let $z \in \C$ with $\re z > 0$.
	Our starting point is Barnes's contour integral representation \cite[(15.6.7)]{DLMF}, which gives
	\begin{align} \label{eqn:contour_integral_rep}
		H_{\frac{1}{2}+it}(1-z^2)
		= \frac{1}{2\pi i \Gamma(\frac{1}{2}+it)^2 \Gamma(\frac{1}{2}-it)^2} \int_{\sigma-i\infty}^{\sigma+i\infty} \Gamma(s)^2 \Gamma\Big(\frac{1}{2}+it-s\Big) \Gamma\Big(\frac{1}{2}-it-s\Big) z^{-2s} \, ds
	\end{align}
	for any $\sigma \in (0,\frac{1}{2})$.
	For a proof of this integral representation, see \cite[Section 14.53]{Whittaker--Watson}.
	By Euler's reflection formula,
	\begin{align*}
		\Gamma\Big(\frac{1}{2}+it\Big) \Gamma\Big(\frac{1}{2}-it\Big)
		= \frac{\pi}{\sin(\pi(\frac{1}{2}+it))}
		= \frac{2\pi}{e^{\pi t} + e^{-\pi t}}.
	\end{align*}
	Therefore, multiplying both sides of \eqref{eqn:contour_integral_rep} by $\Gamma(\frac{1}{2}+it)^2 \Gamma(\frac{1}{2}-it)^2$,
	\begin{align} \label{eqn:contour_integral_rep_2}
		\Big(\frac{2\pi}{e^{\pi t} + e^{-\pi t}}\Big)^2
		H_{\frac{1}{2}+it}(1-z^2)
		= \frac{1}{2\pi i} \int_{\sigma-i\infty}^{\sigma+i\infty} \Gamma(s)^2 \Gamma\Big(\frac{1}{2}+it-s\Big) \Gamma\Big(\frac{1}{2}-it-s\Big) z^{-2s} \, ds.
	\end{align}
	The next lemma gives an upper bound for the integrand.
	
	\begin{lem} \label{lem:contour_integrand_bd}
		Let $\sigma \in (0,\frac{1}{2})$. Let $s \in \C$ with $\re s = \sigma$.
		Let $\delta \in [0,1]$ be such that $\re z \geq \delta|z|$. Then
		\begin{align} \label{eqn:contour_integrand_bd}
			\Big|\Gamma(s)^2 \Gamma\Big(\frac{1}{2}+it-s\Big) \Gamma\Big(\frac{1}{2}-it-s\Big) z^{-2s}\Big|
			\lesssim_{\sigma} e^{-\pi t} |z|^{-2\sigma} |s|^{2\sigma-1} e^{-2\delta|s|}.
		\end{align}
	\end{lem}
	
	\begin{proof}
		Write $z = re^{i\theta}$ with $|\theta| < \frac{\pi}{2}$.
		For any $\sigma_0 > 0$, on the vertical strip $\sigma_0 \leq \re w \leq 1$, the Gamma function is bounded by
		\begin{align*}
			|\Gamma(w)|
			\lesssim_{\sigma_0} |w|^{\re w - \frac{1}{2}} e^{-\frac{\pi}{2}|w|}.
		\end{align*}
		Therefore the left hand side of \eqref{eqn:contour_integrand_bd} satisfies
		\begin{align*}
			\LHS\eqref{eqn:contour_integrand_bd}
			\lesssim_{\sigma} |s|^{2\sigma-1} e^{-\pi|s|} |s-it|^{-\sigma} |s+it|^{-\sigma} e^{-\frac{\pi}{2}(|s-it|+|s+it|)} r^{-2\sigma} e^{2\theta\im s}.
		\end{align*}
		Crudely estimate
		\begin{align*}
			|s\pm it|^{-\sigma}
			\lesssim_{\sigma} 1
		\end{align*}
		and
		\begin{align*}
			e^{-\frac{\pi}{2}(|s-it|+|s+it|)}
			\leq e^{-\pi t}
		\end{align*}
		and
		\begin{align*}
			e^{-\pi|s|} e^{2\theta \im s}
			\lesssim e^{-(\pi-2|\theta|)|s|}.
		\end{align*}
		Putting these estimates together,
		\begin{align} \label{eqn:contour_integrand_bd_almost}
			\LHS\eqref{eqn:contour_integrand_bd}
			\lesssim_{\sigma} e^{-\pi t} r^{-2\sigma} |s|^{2\sigma-1} e^{-(\pi-2|\theta|) |s|}.
		\end{align}
		Observe that
		\begin{align*}
			\pi-2|\theta|
			= 2(\tfrac{\pi}{2} - |\theta|)
			\geq 2\sin(\tfrac{\pi}{2}-|\theta|)
			= 2\cos\theta
			\geq 2\delta.
		\end{align*}
		Inserting this into \eqref{eqn:contour_integrand_bd_almost} yields \eqref{eqn:contour_integrand_bd}.
	\end{proof}
	
	Trivially estimating the right hand side of \eqref{eqn:contour_integral_rep_2} using Lemma~\ref{lem:contour_integrand_bd} gives
	
	\begin{prop}
		\label{prop:u_block_triv_bd}
		Assume $|z| \geq T^{-O(1)}$ and $\re z \geq T^{-O(1)} |z|$. Then
		\begin{align} \label{eqn:u_block_triv_bd}
			|\widetilde{H}_{\frac{1}{2}+it}(1-z^2)|
			\lessapprox t^{4k-2}.
		\end{align}
	\end{prop}
	
	\begin{proof}
		Let $\sigma \in (0,\frac{1}{2})$.
		Combining \eqref{eqn:contour_integral_rep_2} with Lemma~\ref{lem:contour_integrand_bd},
		\begin{align*}
			\Big(\frac{2\pi}{e^{\pi t} + e^{-\pi t}}\Big)^2
			|H_{\frac{1}{2}+it}(1-z^2)|
			\lesssim_{\sigma} e^{-\pi t} |z|^{-2\sigma} \int_{\sigma-i\infty}^{\sigma+i\infty} |s|^{2\sigma-1} e^{-T^{-O(1)}|s|} |ds|.
		\end{align*}
		Writing $|z|^{-2\sigma} \leq T^{O(\sigma)}$ and multiplying both sides by $e^{\pi t}$,
		\begin{align*}
			e^{-\pi t} |H_{\frac{1}{2}+it}(1-z^2)|
			\lesssim_{\sigma} T^{O(\sigma)} \int_{\sigma-i\infty}^{\sigma+i\infty} |s|^{2\sigma-1} e^{-T^{-O(1)}|s|} |ds|.
		\end{align*}
		The exponential factor in the integrand is essentially a cutoff to $|s| \lesssim T^{O(1)}$, so
		\begin{align*}
			e^{-\pi t} |H_{\frac{1}{2}+it}(1-z^2)|
			\lesssim_{\sigma} T^{O(\sigma)} \int_{\substack{\re s = \sigma \\ |s| \lesssim T^{O(1)}}} |s|^{2\sigma-1} |ds|
			\lesssim_{\sigma} T^{O(\sigma)}.
		\end{align*}
		Since this holds for arbitrary $\sigma > 0$, we deduce
		\begin{align*}
			e^{-\pi t} |H_{\frac{1}{2}+it}(1-z^2)|
			\lessapprox 1.
		\end{align*}
		This is equivalent to \eqref{eqn:u_block_triv_bd}.
	\end{proof}
	
	We now turn to the main result of this subsection, Proposition~\ref{prop:u_block_asymptotic}.
	
	\begin{proof}[Proof of Proposition~\ref{prop:u_block_asymptotic}]
		Let $z$ be as in Proposition~\ref{prop:u_block_asymptotic}, so $t^{-1} \lesssim |z| \lesssim 1$ and $\re z \gtrsim t^{-\frac{2}{3}+\varepsilon}|z|$.
		Fix $\sigma \in (0,\frac{1}{2})$, say $\sigma = \frac{1}{4}$ for concreteness.
		Then by \eqref{eqn:contour_integral_rep_2}, together with Lemma~\ref{lem:contour_integrand_bd} to bound the tails in \eqref{eqn:contour_integral_rep_2},
		\begin{align} \label{eqn:truncate_integral_1}
			\Big(\frac{2\pi}{e^{\pi t} + e^{-\pi t}}\Big)^2
			H_{\frac{1}{2}+it}(1-z^2)
			= \frac{1}{2\pi i} \int_{\substack{\re s = \sigma \\ |s| \leq t^{\frac{2}{3}-\frac{\varepsilon}{2}}}} \Gamma(s)^2 \Gamma\Big(\frac{1}{2}+it-s\Big) \Gamma\Big(\frac{1}{2}-it-s\Big) z^{-2s} \, ds + O(e^{-\pi t} t^{-\infty}).
		\end{align}
		Stirling's approximation says that for $w \in \C$ with, say, $\re w \geq -0.49$ and $|\im w| \gtrsim 1$,
		\begin{align*}
			\log\Gamma\Big(\frac{1}{2}+w\Big)
			= w\log w - w + \frac{1}{2} \log(2\pi) + \sum_{n=0}^{\infty} \frac{a_n}{w^{2n+1}} + O(|w|^{-\infty})
		\end{align*}
		for some constants $a_n \in \Q$.
		This infinite series is generally divergent, so it should be interpreted in the usual asymptotic sense, namely that for any $M \geq 0$, there exists $N \in \Z_{\geq 0}$ such that if the series is replaced by the $N$th partial sum, then the error is $O_M(|w|^{-M})$. Since the domain of integration in \eqref{eqn:truncate_integral_1} is truncated to $|s| \leq t^{\frac{2}{3}-\frac{\varepsilon}{2}}$, we have $|\im(\pm it-s)| \sim t$. Thus by Stirling,
		\begin{align*}
			\log\Gamma\Big(\frac{1}{2} \pm it - s\Big)
			= (\pm it-s)\log(\pm it-s) - (\pm it-s) + \frac{1}{2}\log(2\pi) + \sum_{n=0}^{\infty} \frac{a_n}{(\pm it-s)^{2n+1}} + O(t^{-\infty}).
		\end{align*}
		Summing this asymptotic for $\log\Gamma(\frac{1}{2} + it - s)$ and $\log\Gamma(\frac{1}{2} - it - s)$,
		and then writing $\pm it-s = \pm it(1+\pm i\frac{s}{t})$ and expanding in powers of $s/t = O(t^{-\frac{1}{3}-\frac{\varepsilon}{2}})$,
		\begin{align*}
			\log\Gamma\Big(\frac{1}{2} + it - s\Big)
			+ \log\Gamma\Big(\frac{1}{2} - it - s\Big)
			= -\pi t - 2s\log t + \log(2\pi) + \sum_{j=1}^{\infty} \sum_{\ell=0}^{j} b_{j,\ell} \frac{s^{2\ell+1}}{t^{2j}} + O(t^{-\infty})
		\end{align*}
		for some constants $b_{j,\ell} \in \Q$. Exponentiating,
		\begin{align*}
			\Gamma\Big(\frac{1}{2} + it - s\Big) \Gamma\Big(\frac{1}{2} - it - s\Big)
			= 2\pi e^{-\pi t} t^{-2s} \exp\Big(\sum_{j=1}^{\infty} \sum_{\ell=0}^{j} b_{j,\ell} \frac{s^{2\ell+1}}{t^{2j}} + O(t^{-\infty})\Big).
		\end{align*}
		Since $|s| \leq t^{\frac{2}{3}-\frac{\varepsilon}{2}}$, each term in the series is $O(t^{-\frac{3}{2}\varepsilon})$ (indeed, the asymptotically largest term is $s^3/t^2 = O(t^{-\frac{3}{2}\varepsilon})$, corresponding to $j=\ell = 1$).
		In addition, since the series is interpreted in the asymptotic sense, we are only interested in finitely many terms at a time. 
		We can therefore Taylor expand the exponential, to get
		\begin{align*}
			\Gamma\Big(\frac{1}{2} + it - s\Big) \Gamma\Big(\frac{1}{2} - it - s\Big)
			= 2\pi e^{-\pi t} t^{-2s} (1+O(t^{-\infty}))
			\sum_{j=0}^{\infty} \sum_{\ell=0}^{3j} c_{j,\ell} \frac{s^{\ell}}{t^{2j}}
		\end{align*}
		for some constants $c_{j,\ell} \in \Q$ with $c_{0,0} = 1$. Inserting this into \eqref{eqn:truncate_integral_1},
		\begin{align*}
			\Big(\frac{2\pi}{e^{\pi t} + e^{-\pi t}}\Big)^2 H_{\frac{1}{2}+it}(1-z^2)
			&= \frac{e^{-\pi t}}{i} \int_{\substack{\re s = \sigma \\ |s| \leq t^{\frac{2}{3}-\frac{\varepsilon}{2}}}} \Gamma(s)^2 (tz)^{-2s} (1+O(t^{-\infty}))
			\sum_{j=0}^{\infty} \sum_{\ell=0}^{3j} c_{j,\ell} \frac{s^{\ell}}{t^{2j}} \, ds
			\\&+ O(e^{-\pi t} t^{-\infty}).
		\end{align*}
		Rearranging factors of $2$, $\pi$, and $e^{\pi t}$, and using that $e^{\pi t}+e^{-\pi t} = e^{\pi t} (1+O(t^{-\infty}))$,
		\begin{align*}
			\pi e^{-\pi t} H_{\frac{1}{2}+it}(1-z^2)
			= \frac{1}{4\pi i} \int_{\substack{\re s = \sigma \\ |s| \leq t^{\frac{2}{3}-\frac{\varepsilon}{2}}}} \Gamma(s)^2 (tz)^{-2s} (1+O(t^{-\infty}))
			\sum_{j=0}^{\infty} \sum_{\ell=0}^{3j} c_{j,\ell} \frac{s^{\ell}}{t^{2j}} \, ds
			+ O(t^{-\infty}).
		\end{align*}
		Switching the sums and the integral (which is valid because the sums are interpreted in the asymptotic sense and hence are really finite sums), and absorbing  the error term in the integral into the error term outside the integral,
		\begin{align*}
			\pi e^{-\pi t} H_{\frac{1}{2}+it}(1-z^2)
			= \sum_{j=0}^{\infty} t^{-2j} \sum_{\ell=0}^{3j} \frac{c_{j,\ell}}{4\pi i} \int_{\substack{\re s = \sigma \\ |s| \leq t^{\frac{2}{3}-\frac{\varepsilon}{2}}}} \Gamma(s)^2 (tz)^{-2s} s^{\ell} \, ds
			+ O(t^{-\infty}).
		\end{align*}
		By a similar computation to Lemma~\ref{lem:contour_integrand_bd}, the integrand decays exponentially once $|s|$ exceeds $t^{\frac{2}{3}-\varepsilon}$, so in particular the integrand is $\lesssim |s|^{-\infty}$ for $|s| > t^{\frac{2}{3}-\frac{\varepsilon}{2}}$.
		Thus we can complete the integral:
		\begin{align*}
			\pi e^{-\pi t} H_{\frac{1}{2}+it}(1-z^2)
			= \sum_{j=0}^{\infty} t^{-2j} \sum_{\ell=0}^{3j} \frac{c_{j,\ell}}{4\pi i} \int_{\sigma-i\infty}^{\sigma+i\infty} \Gamma(s)^2 (tz)^{-2s} s^{\ell} \, ds + O(t^{-\infty}).
		\end{align*}
		Rewrite this as
		\begin{align*}
			\pi e^{-\pi t} H_{\frac{1}{2}+it}(1-z^2)
			= \sum_{j=0}^{\infty} t^{-2j} \sum_{\ell=0}^{3j} \frac{d_{j,\ell}}{4\pi i} (z\partial_z)^{\ell}
			\int_{\sigma-i\infty}^{\sigma+i\infty} \Gamma(s)^2 (tz)^{-2s} \, ds + O(t^{-\infty}),
		\end{align*}
		with $d_{j,\ell} = (-1)^{\ell} 2^{-\ell} c_{j,\ell}$. So $d_{j,\ell} \in \Q$ and $d_{0,0} = 1$. Using the integral representation
		\begin{align*}
			K_0(w) = \frac{1}{4\pi i}\int_{\sigma-i\infty}^{\sigma+i\infty} \Gamma(s)^2 (w/2)^{-2s} \, ds
		\end{align*}
		for the Bessel function $K_0$ \cite[(10.32.13)]{DLMF}, the above simplifies to
		\begin{align*}
			\pi e^{-\pi t} H_{\frac{1}{2}+it}(1-z^2)
			= \sum_{j=0}^{\infty} t^{-2j} \sum_{\ell=0}^{3j} d_{j,\ell} (z\partial_z)^{\ell}
			[K_0(2tz)] + O(t^{-\infty}).
		\end{align*}
		Expanding $(z\partial_z)^{\ell}$ using the product rule shows that $(z\partial_z)^{\ell}$ is a $\Z$-linear combination of $\{z^m\partial_z^m\}_{0 \leq m \leq \ell}$. Therefore
		\begin{align} \label{eqn:u_asymptotic_in_Bessels}
			\pi e^{-\pi t} H_{\frac{1}{2}+it}(1-z^2)
			= \sum_{j=0}^{\infty} t^{-2j} \sum_{m=0}^{3j} e_{j,m} (tz)^m K_0^{(m)}(2tz) + O(t^{-\infty})
		\end{align}
		for some constants $e_{j,m} \in \Q$ with $e_{0,0} = 1$.
		By \cite[(5.6.9)]{Lebedev}, the $m$th derivative $K_0^{(m)}$ is $(-1)^m$ times a finite convex combination of $\{K_{\nu}\}_{\nu \in \Z}$. By \cite[(5.11.9)]{Lebedev}, for $w \in \C \setminus (-\infty,0]$ with $|w| \gtrsim 1$ and $|\Arg w| \leq \pi-\delta$ for some $\delta > 0$,
		\begin{align*}
			K_{\nu}(w)
			= e^{-w} \sqrt{\frac{\pi}{2w}} \, (1+O_{\nu,\delta}(|w|^{-1})).
		\end{align*}
		It follows that for $w$ satisfying the same conditions,
		\begin{align*}
			K_0^{(m)}(w)
			= (-1)^m e^{-w} \sqrt{\frac{\pi}{2w}} \, (1+O_{m,\delta}(|w|^{-1})).
		\end{align*}
		By assumption, $|z| \gtrsim t^{-1}$, and so $|tz| \gtrsim 1$. In addition, by assumption, $\re z > 0$. Thus we can apply the above asymptotic for $K_0^{(m)}$ to obtain
		\begin{align*}
			\pi e^{-\pi t} H_{\frac{1}{2}+it}(1-z^2)
			&= e^{-2tz} \sqrt{\frac{\pi}{4tz}}\sum_{j=0}^{\infty} t^{-2j} \sum_{m=0}^{3j} f_{j,m} (tz)^m (1+O_m(|tz|^{-1})) + O(t^{-\infty})
			\\&= e^{-2tz} \frac{\sqrt{\pi}}{2\sqrt{tz}} \sum_{j=0}^{\infty} f_{j,3j} t^{-2j} (tz)^{3j} (1+O_j(|tz|^{-1})) + O(t^{-\infty})
		\end{align*}
		with $f_{j,m} = (-1)^m e_{j,m}$. Again, $f_{j,m} \in \Q$ and $f_{0,0} = 1$. Denote $c_j = f_{j,3j}$, so $c_j \in \Q$ with $c_0 = 1$. Then multiplying both sides above by $\frac{1}{\pi} t^{4k-2}$, and simplifying,
		\begin{align} \label{eqn:u_asymptotic_before_claim}
			\widetilde{H}_{\frac{1}{2}+it}(1-z^2)
			= \frac{t^{4k-\frac{5}{2}}}{2\sqrt{\pi z}} \Big[e^{-2tz} \sum_{j=0}^{\infty} c_j (tz^3)^j (1+O_j(|tz|^{-1}))\Big] + O(t^{-\infty}).
		\end{align}
		We claim that we can rewrite the right hand side to get
		\begin{align} \label{eqn:u_asymptotic_claim}
			\widetilde{H}_{\frac{1}{2}+it}(1-z^2)
			= \frac{t^{4k-\frac{5}{2}}}{2\sqrt{\pi z}} \Big[e^{-2tz} \sum_{j=0}^{\infty} c_j (tz^3)^j + O(e^{-2t\re z}|tz|^{-1} + t^{-\infty})\Big].
		\end{align}
		If $|tz^3| \lesssim 1$, then this is clear. If $|tz^3| \gtrsim 1$, then as mentioned above, our assumption $\re z \gtrsim t^{-\frac{2}{3}+\varepsilon} |z|$ implies the lower bound
		\begin{align*}
			\re(tz)
			\gtrsim t^{\frac{1}{3}+\varepsilon}|z|
			= t^{\varepsilon} |tz^3|^{\frac{1}{3}}
			\gtrsim t^{\varepsilon},
		\end{align*}
		and hence the exponential factor on the right hand side of \eqref{eqn:u_asymptotic_before_claim} is extremely small and dominates the rest of the asymptotic.
		In particular, $\RHS\eqref{eqn:u_asymptotic_before_claim} \lesssim t^{-\infty}$. Thus \eqref{eqn:u_asymptotic_claim} holds in this case as well. In summary, we have proved \eqref{eqn:u_asymptotic_claim} both when $|tz^3| \lesssim 1$ and when $|tz^3| \gtrsim 1$, so $\eqref{eqn:u_asymptotic_claim}$ holds in general. Since \eqref{eqn:u_asymptotic_claim} is interpreted in the asymptotic sense, the content of \eqref{eqn:u_asymptotic_claim} is precisely Proposition~\ref{prop:u_block_asymptotic}.
		This concludes the proof.
	\end{proof}
	
	\begin{cor} \label{cor:u_block_rapid_decay}
		Assume $t \geq T\log^2T$. Let $z$ be as in \eqref{eqn:y_range}. Then
		\begin{align} \label{eqn:u_block_rapid_decay}
			|\widetilde{H}_{\frac{1}{2}+it}(1-z^2)|
			\lesssim t^{-\infty}.
		\end{align}
	\end{cor}
	
	\begin{proof}
		Note that $z$ obeys the assumptions of Proposition~\ref{prop:u_block_asymptotic}. Therefore, we can take absolute values in \eqref{eqn:u_block_asymptotic} to get
		\begin{align*}
			|\widetilde{H}_{\frac{1}{2}+it}(1-z^2)|
			\lesssim_M t^{4k-\frac{5}{2}} |z|^{-\frac{1}{2}} [e^{-2t/T} t^{O_M(1)} + t^{-M}]
		\end{align*}
		for any $M$. Since $|z| \geq T^{-1}$ and $t \geq T \log^2 T$, we have $|z|^{-\frac{1}{2}} e^{-2t/T} \lesssim t^{-\infty}$. We conclude \eqref{eqn:u_block_rapid_decay}.
	\end{proof}
	
	A corollary of \eqref{eqn:u_asymptotic_in_Bessels} in the proof of Proposition~\ref{prop:u_block_asymptotic} is
	
	\begin{cor} \label{cor:u_block_size_z>0}
		Assume $t \gg 1$.
		Assume also that $z$ is real and positive, and that $z \sim t^{-1}$. Then
		\begin{align} \label{eqn:u_block_size_z>0}
			0 \leq \widetilde{H}_{\frac{1}{2}+it}(1-z^2)
			\sim t^{4k-2}.
		\end{align}
	\end{cor}
	
	\begin{proof}
		As explained in the proof of Proposition~\ref{prop:u_block_at_T^-1}, $\widetilde{H}_{\frac{1}{2}+it}(1-z^2)$ is nonnegative because $\widetilde{H}_{\frac{1}{2}+it}$ has nonnegative Taylor coefficients at the origin.
		
		Now, the assumptions on $z$ in Corollary~\ref{cor:u_block_size_z>0} imply the assumptions on $z$ in Proposition~\ref{prop:u_block_asymptotic}, so \eqref{eqn:u_asymptotic_in_Bessels} is valid.
		Since $tz \sim 1$, the sum over $m$ in \eqref{eqn:u_asymptotic_in_Bessels} is $O_j(1)$. Therefore, retaining only the $j=0$ term in \eqref{eqn:u_asymptotic_in_Bessels},
		\begin{align*}
			\pi e^{-\pi t} H_{\frac{1}{2}+it}(1-z^2)
			= K_0(2tz) + O(t^{-2}).
		\end{align*}
		The $K_0$ Bessel function is positive and of size $\sim 1$ on $\{x \sim 1\}$. Since $t \gg 1$, it follows that
		\begin{align*}
			e^{-\pi t} H_{\frac{1}{2}+it}(1-z^2)
			\sim 1.
		\end{align*}
		This is equivalent to \eqref{eqn:u_block_size_z>0}.
	\end{proof}
	
	\subsection{t-Channel} \label{subsec:t_asymptotic}
	
	We only need one simple lemma before beginning the proof of Proposition~\ref{prop:t_block_asymptotic}.
	
	\begin{lem} \label{lem:imag_part_bd}
		Let $h$ be a holomorphic function defined on a neighborhood of the origin. Assume $h$ is real-valued on the real line. Let $z \in \C$ with $|z| \ll_h 1$. Then
		\begin{align*}
			|\im h(z)|
			\lesssim_h |\im z|.
		\end{align*}
	\end{lem}
	
	\begin{proof}
		Let $z = x+iy$. Since $h$ is real-valued on the real line, $\im h(x) = 0$. Thus we can write
		\begin{align*}
			&|\im h(z)|
			= |\im(h(z)-h(x))|
			\leq |h(z)-h(x)|
			= \Big|\int_{0}^{y} h'(x+i\tau) \, d\tau\Big|
			\lesssim_h |y|.
			\qedhere
		\end{align*}
	\end{proof}
	
	\begin{proof}[Proof of Proposition~\ref{prop:t_block_asymptotic}]
		By \cite[(9.5.9)]{Lebedev}, one has the functional equation
		\begin{align}
			H_{\frac{1}{2}+it}(w)
			&= \frac{\Gamma(-2it)}{\Gamma(\frac{1}{2}-it)^2} (-w)^{-\frac{1}{2}-it} {}_2F_1\Big(\frac{1}{2}+it,\frac{1}{2}+it; 1+2it; w^{-1}\Big)
			\notag
			\\&+ \frac{\Gamma(2it)}{\Gamma(\frac{1}{2}+it)^2} (-w)^{-\frac{1}{2}+it} {}_2F_1\Big(\frac{1}{2}-it,\frac{1}{2}-it; 1-2it; w^{-1}\Big)
			\quad \text{for } \quad w \in \C \setminus [0,\infty).
			\label{eqn:functional_eqn_1/w}
		\end{align}
		Denote
		\begin{align} \label{eqn:rho_variable_def}
			\rho(u)
			= \frac{1-\sqrt{1-u}}{1+\sqrt{1-u}}
			= \frac{u}{(1+\sqrt{1-u})^2},
		\end{align}
		and note that
		\begin{align} \label{eqn:4wrho(1/w)_formula}
			4w \rho(w^{-1})
			= \Big(\frac{1 + \sqrt{1-w^{-1}}}{2}\Big)^{-2}.
		\end{align}
		Then by \cite[(9.6.12)]{Lebedev} and \eqref{eqn:4wrho(1/w)_formula}, the hypergeometrics on the right hand side of \eqref{eqn:functional_eqn_1/w} satisfy the further functional equations
		\begin{align*}
			(-w)^{-\frac{1}{2}-it} {}_2F_1\Big(\frac{1}{2}+it,\frac{1}{2}+it; 1+2it; w^{-1}\Big)
			= (-4\rho(w^{-1}))^{\frac{1}{2}+it} {}_2F_{1}\Big(\frac{1}{2},\frac{1}{2}+it; 1+it; \rho(w^{-1})^2\Big)
		\end{align*}
		and
		\begin{align*}
			(-w)^{-\frac{1}{2}+it} {}_2F_1\Big(\frac{1}{2}-it,\frac{1}{2}-it; 1-2it; w^{-1}\Big)
			= (-4\rho(w^{-1}))^{\frac{1}{2}-it} {}_2F_{1}\Big(\frac{1}{2},\frac{1}{2}-it; 1-it; \rho(w^{-1})^2\Big),
		\end{align*}
		again for $w \in \C \setminus [0,\infty)$.
		Inserting these into \eqref{eqn:functional_eqn_1/w},
		\begin{align*}
			H_{\frac{1}{2}+it}(w)
			&= \frac{\Gamma(-2it)}{\Gamma(\frac{1}{2}-it)^2} (-4\rho(w^{-1}))^{\frac{1}{2}+it} {}_2F_{1}\Big(\frac{1}{2},\frac{1}{2}+it; 1+it; \rho(w^{-1})^2\Big)
			\\&+ \frac{\Gamma(2it)}{\Gamma(\frac{1}{2}+it)^2} (-4\rho(w^{-1}))^{\frac{1}{2}-it} {}_2F_{1}\Big(\frac{1}{2},\frac{1}{2}-it; 1-it; \rho(w^{-1})^2\Big)
			\quad \text{for } \quad w \in \C \setminus [0,\infty).
		\end{align*}
		Observe from \eqref{eqn:2F1_series} that the coefficients in the Taylor expansion of ${}_2F_{1}(\frac{1}{2},\frac{1}{2}\pm it; 1\pm it; \rho)$ (as a function of $\rho$) at $\rho=0$ are bounded independent of $t$. Therefore
		\begin{align*}
			{}_2F_{1}\Big(\frac{1}{2},\frac{1}{2}\pm it; 1\pm it; \rho\Big)
			= 1+O(\rho)
			\qquad \text{for} \qquad
			|\rho| \ll 1.
		\end{align*}
		Noting that $|\rho(w^{-1})| \lesssim |w|^{-1}$ for $|w| \gg 1$, we get
		\begin{align*}
			H_{\frac{1}{2}+it}(w)
			= \frac{\Gamma(-2it)}{\Gamma(\frac{1}{2}-it)^2} (-4\rho(w^{-1}))^{\frac{1}{2}+it} (1+O(|w|^{-2}))
			+ \frac{\Gamma(2it)}{\Gamma(\frac{1}{2}+it)^2} (-4\rho(w^{-1}))^{\frac{1}{2}-it} (1+O(|w|^{-2}))
		\end{align*}
		for $w \in \C \setminus [0,\infty)$ with $|w| \gg 1$.
		Taking $w=1-z^{-2}$ with $z$ as in the statement of Proposition~\ref{prop:t_block_asymptotic},
		\begin{align*}
			H_{\frac{1}{2}+it}(1-z^{-2})
			&= \frac{\Gamma(-2it)}{\Gamma(\frac{1}{2}-it)^2} \Big(-4\rho\Big(\frac{1}{1-z^{-2}}\Big)\Big)^{\frac{1}{2}+it} (1+O(|z|^4))
			\\&+ \frac{\Gamma(2it)}{\Gamma(\frac{1}{2}+it)^2} \Big(-4\rho\Big(\frac{1}{1-z^{-2}}\Big)\Big)^{\frac{1}{2}-it} (1+O(|z|^4)).
		\end{align*}
		By Stirling's approximation,
		\begin{align*}
			\frac{\Gamma(-2it)}{\Gamma(\frac{1}{2}-it)^2} 4^{\frac{1}{2}+it}
			= \frac{e^{i\pi/4}}{\sqrt{\pi t}} (1+O(t^{-1}))
			\qquad \text{and} \qquad
			\frac{\Gamma(2it)}{\Gamma(\frac{1}{2}+it)^2} 4^{\frac{1}{2}-it}
			= \frac{e^{-i\pi/4}}{\sqrt{\pi t}} (1+O(t^{-1})).
		\end{align*}
		Thus,
		\begin{align}
			H_{\frac{1}{2}+it}(1-z^{-2})
			&= \frac{e^{i\pi/4}}{\sqrt{\pi t}} \Big(-\rho\Big(\frac{1}{1-z^{-2}}\Big)\Big)^{\frac{1}{2}+it} (1+O(t^{-1}+|z|^4))
			\notag
			\\&+ \frac{e^{-i\pi/4}}{\sqrt{\pi t}} \Big(-\rho\Big(\frac{1}{1-z^{-2}}\Big)\Big)^{\frac{1}{2}-it} (1+O(t^{-1}+|z|^4)).
			\label{eqn:last_t-approx_with_rho}
		\end{align}
		Expanding the definition of $\rho$ and simplifying,
		\begin{align*}
			-\rho\Big(\frac{1}{1-z^{-2}}\Big)
			= \Big(\frac{z}{1+\sqrt{1-z^2}}\Big)^2.
		\end{align*}
		Taking logs,
		\begin{align*}
			\log\Big[-\rho\Big(\frac{1}{1-z^{-2}}\Big)\Big]
			= 2\log\Big(\frac{z}{1+\sqrt{1-z^2}}\Big)
			= 2\log z - 2\log(1+\sqrt{1-z^2})
		\end{align*}
		(here we are using the assumption in Proposition~\ref{prop:t_block_asymptotic} that $\re z > 0$ to make the first equality hold with the principal branch of $\log$).
		As in the statement of Proposition~\ref{prop:t_block_asymptotic}, denote
		\begin{align*}
			\phi(z) = -\log(1+\sqrt{1-z^2}),
		\end{align*}
		so that
		\begin{align*}
			\log\Big[-\rho\Big(\frac{1}{1-z^{-2}}\Big)\Big]
			= 2(\log z + \phi(z)).
		\end{align*}
		Writing $\zeta^{\frac{1}{2} \pm it} = \zeta^{\frac{1}{2}} \exp(\pm it \log \zeta)$ in \eqref{eqn:last_t-approx_with_rho},
		\begin{align}
			H_{\frac{1}{2}+it}(1-z^{-2})
			= \frac{z}{1+\sqrt{1-z^2}} &\Big[\frac{e^{i\pi/4}}{\sqrt{\pi t}} \exp(2it[\log z + \phi(z)]) (1+O(t^{-1}+|z|^4))
			\notag
			\\&+ \frac{e^{-i\pi/4}}{\sqrt{\pi t}} \exp(-2it[\log z + \phi(z)]) (1+O(t^{-1}+|z|^4)) \Big].
			\label{eqn:first_t-approx_without_rho}
		\end{align}
		By assumption in Proposition~\ref{prop:t_block_asymptotic}, $\im z \geq 0$, so $\Arg z \geq 0$. Moreover, since by assumption $|z| \ll 1$, we have $\Arg z \gg \im z$. Thus by Lemma~\ref{lem:imag_part_bd},
		\begin{align*}
			\im[\log z + \phi(z)]
			= \Arg z + O(\im z)
			\gtrsim \Arg z.
		\end{align*}
		In particular, $\log z + \phi(z)$ has nonnegative imaginary part.
		It follows from this nonnegativity that the exponential in the first term in \eqref{eqn:first_t-approx_without_rho} has absolute value $\leq 1$. Therefore
		\begin{align*}
			H_{\frac{1}{2}+it}(1-z^{-2})
			= \frac{z}{1+\sqrt{1-z^2}} \Big[\frac{e^{-i\pi/4}}{\sqrt{\pi t}} \exp(-2it[\log z + \phi(z)]) (1+O(t^{-1}+|z|^4)) + O(t^{-\frac{1}{2}})\Big].
		\end{align*}
		Multiplying both sides by $t^{4k-2} e^{-\pi t}$ transforms this into an estimate for $\widetilde{H}_{\frac{1}{2}+it}$:
		\begin{align*}
			\widetilde{H}_{\frac{1}{2}+it}(1-z^{-2})
			= t^{4k-\frac{5}{2}} \frac{\pi^{-\frac{1}{2}} e^{-i\pi/4} z}{1+\sqrt{1-z^2}} e^{-\pi t} \exp(-2it[\log z + \phi(z)]) (1+O(t^{-1}+|z|^4)) + O(t^{4k-\frac{5}{2}} |z| e^{-\pi t}).
		\end{align*}
		Absorbing $e^{-\pi t}$ into the exponential,
		\begin{align*}
			\widetilde{H}_{\frac{1}{2}+it}(1-z^{-2})
			= t^{4k-\frac{5}{2}} \frac{\pi^{-\frac{1}{2}} e^{-i\pi/4} z}{1+\sqrt{1-z^2}} \exp(-2it[\log(-iz) + \phi(z)]) (1+O(t^{-1}+|z|^4)) + O(t^{4k-\frac{5}{2}} |z| e^{-\pi t}).
		\end{align*}
		This is the desired estimate \eqref{eqn:t_block_asymptotic}.
	\end{proof}
	
	\begin{cor} \label{cor:t_block_lower_bd}
		Let $\delta > 0$ be arbitrarily small. Then there exists $z \in \C$ with $\re z > 0$, such that for $t \gg 1$,
		\begin{align*}
			|\widetilde{H}_{\frac{1}{2}+it}(1-z^{-2})|
			\gtrsim_{\delta} e^{-\delta t}.
		\end{align*}
	\end{cor}
	
	\begin{proof}
		Fix $0 < y \ll 1$, and let $z = x+iy$ for some $x > 0$ with $x \ll_{\delta} 1$.
		We have
		\begin{align*}
			\re(2i[\log(-i(iy)) + \phi(iy)])
			= 0,
		\end{align*}
		so by continuity, we may assume $x$ is small enough that
		\begin{align*}
			\re(2i[\log(-iz) + \phi(z)])
			\leq \tfrac{1}{2} \delta.
		\end{align*}
		Then from Proposition~\ref{prop:t_block_asymptotic}, we see that indeed
		\begin{align*}
			|\widetilde{H}_{\frac{1}{2}+it}(1-z^{-2})|
			\gtrsim_{\delta} e^{-\delta t}
		\end{align*}
		for $t \gg 1$.
	\end{proof}
	
	In the range that is relevant for us, the exponential term in \eqref{eqn:t_block_asymptotic} is controlled by the following lemma.
	
	\begin{lem} \label{lem:t_exponential_bd}
		Let $z = T^{-1} + iy$ for some $y \in [0, T^{-\frac{1}{3}}]$. Then
		\begin{align*}
			\re(2it[\log(-iz) + \phi(z)])
			\gtrsim \frac{t}{Ty+1}.
		\end{align*}
		In particular, this real part is positive.
	\end{lem}
	
	\begin{proof}
		Write
		\begin{align*}
			\re(2it[\log(-iz) + \phi(z)])
			&= -2t\im\log(y-iT^{-1}) -2t\im\phi(T^{-1}+iy)
			\\&= 2t\Arg(y+iT^{-1}) + O(t|\im\phi(T^{-1}+iy)|).
		\end{align*}
		We can lower bound
		\begin{align*}
			\Arg(y+iT^{-1})
			= \Arg(Ty+i)
			\gtrsim \frac{1}{Ty+1},
		\end{align*}
		so we will be done if we can show that
		\begin{align} \label{eqn:im_phi_goal}
			|\im\phi(T^{-1}+iy)|
			\ll \frac{1}{Ty+1}.
		\end{align}
		By inspection of the definition of $\phi$, we can write $\phi(w) = \tilde{\phi}(w^2)$, where $\tilde{\phi}$ is holomorphic near $0$ and real-valued on the real line. Then
		\begin{align*}
			\phi(T^{-1}+iy)
			= \tilde{\phi}(T^{-2} - y^2 + 2iT^{-1}y).
		\end{align*}
		It follows from Lemma~\ref{lem:imag_part_bd} that
		\begin{align*}
			|\im\phi(T^{-1}+iy)|
			\lesssim T^{-1} y
			\ll \frac{1}{Ty+1}
		\end{align*}
		(here the second inequality is because $T \gg 1$ and $y \in [0,T^{-\frac{1}{3}}]$).
		Thus \eqref{eqn:im_phi_goal} holds.
	\end{proof}
	
	\begin{cor} \label{cor:t-block_triv_bd}
		There is a constant $c \gtrsim 1$, such that the following holds.
		Let $z = T^{-1}+iy$ for some $y \in [0,T^{-\frac{1}{3}}]$. Then
		\begin{align*}
			|\widetilde{H}_{\frac{1}{2}+it}(1-z^{-2})|
			\lesssim t^{4k-\frac{5}{2}} |z| \exp\Big(-\frac{ct}{Ty+1}\Big)
			\leq t^{4k-\frac{5}{2}} |z|.
		\end{align*}
	\end{cor}
	
	\begin{proof}
		Combine Proposition~\ref{prop:t_block_asymptotic} and Lemma~\ref{lem:t_exponential_bd}.
	\end{proof}
	
	In the critical range, the following proposition separates the smooth and oscillatory parts of the asymptotic \eqref{eqn:t_block_asymptotic} in Proposition~\ref{prop:t_block_asymptotic} (after multiplying both sides of \eqref{eqn:t_block_asymptotic} by $z^p$ for any fixed $p \in \R$).
	
	\begin{prop} \label{prop:t_block_asymptotic_separated}
		Assume $1 < t \leq T^{\frac{2}{3}}$. Let $p \in \R$. Then there exists a smooth function $\rho$ on $[2T^{-1}, T^{-\frac{1}{3}}]$
		obeying the derivative bounds
		\begin{align} \label{eqn:rho_deriv_bds}
			|\rho^{(j)}(y)|
			\lesssim_{p,j} y^{-j}
			\qquad \textnormal{for all} \qquad
			j \geq 0 \textnormal{ and } y \in [2T^{-1}, T^{-\frac{1}{3}}],
		\end{align}
		such that for $z = T^{-1}+iy$ with $y \in [2T^{-1}, T^{-\frac{1}{3}}]$,
		\begin{align} \label{eqn:t_asymp_via_rho}
			z^p \widetilde{H}_{\frac{1}{2}+it}(1-z^{-2})
			= t^{4k-\frac{5}{2}} y^{p+1}\rho(y) e^{-2it\log y} + O_p(t^{4k-\frac{7}{2}} y^{p+1}).
		\end{align}
		Explicitly, one can take
		\begin{align} \label{eqn:rho_def}
			\rho(y)
			= \Big(\frac{1}{Ty}+i\Big)^{p+1} \psi(z) \exp\Big(-2it\Big[\log\Big(1-\frac{i}{Ty}\Big) + \phi(z)\Big]\Big),
		\end{align}
		where $\phi,\psi$ are as in Proposition~\ref{prop:t_block_asymptotic}.
		To be clear, although $\rho$ depends on $p,t,T$, the implicit constants in \eqref{eqn:rho_deriv_bds} are independent of $t,T$.
	\end{prop}
	
	\begin{proof}
		By Proposition~\ref{prop:t_block_asymptotic},
		\begin{align*}
			t^{-4k+\frac{5}{2}} z^p\widetilde{H}_{\frac{1}{2}+it}(1-z^{-2})
			= z^{p+1} \psi(z) e^{-2it[\log(-iz) + \phi(z)]} (1 + O(t^{-1} + |z|^4)) + O(|z|^{p+1}e^{-\pi t}).
		\end{align*}
		By the positivity assertion in Lemma~\ref{lem:t_exponential_bd} and the fact that $|z| \sim y$,
		\begin{align*}
			t^{-4k+\frac{5}{2}} z^p\widetilde{H}_{\frac{1}{2}+it}(1-z^{-2})
			= z^{p+1} \psi(z) e^{-2it[\log(-iz) + \phi(z)]} + O_p(t^{-1}y^{p+1} + y^{p+5} + y^{p+1}e^{-\pi t}).
		\end{align*}
		Since $t \leq T^{\frac{2}{3}}$ and $y \leq T^{-\frac{1}{3}}$, the first error term dominates the other two, so
		\begin{align} \label{eqn:before_inserting_rho}
			t^{-4k+\frac{5}{2}} z^p\widetilde{H}_{\frac{1}{2}+it}(1-z^{-2})
			= z^{p+1} \psi(z) e^{-2it[\log(-iz) + \phi(z)]} + O_p(t^{-1}y^{p+1}).
		\end{align}
		Writing $z = y(\frac{1}{Ty}+i)$, we can express the main term on the right hand side as
		\begin{align*}
			z^{p+1} \psi(z) e^{-2it[\log(-iz) + \phi(z)]}
			= y^{p+1} \rho(y) e^{-2it\log y},
		\end{align*}
		where $\rho$ is given by \eqref{eqn:rho_def}.
		Then multiplying both sides of \eqref{eqn:before_inserting_rho} by $t^{4k-\frac{5}{2}}$ yields \eqref{eqn:t_asymp_via_rho}.
		
		It remains to check that $\rho$ obeys the derivative bounds in \eqref{eqn:rho_deriv_bds}.
		Write
		\begin{align} \label{eqn:rho_factored}
			\rho
			= \rho_1^{p+1} \rho_2 \exp(-\rho_3-\rho_4),
		\end{align}
		where
		\begin{align*}
			\rho_1(y) = \frac{1}{Ty} + i,
			\qquad
			\rho_2(y) = \psi(z),
			\qquad
			\rho_3(y) = 2it\log\Big(1-\frac{i}{Ty}\Big),
			\qquad
			\rho_4(y) = 2it\phi(z).
		\end{align*}
		Trivially
		\begin{align} \label{eqn:rho_1_deriv_bds}
			|\rho_1(y)| \sim 1
			\qquad \text{and} \qquad
			|\rho_1^{(j)}(y)|
			\lesssim_j \frac{1}{Ty^{j+1}}
			\lesssim y^{-j}
			\qquad \text{for all} \qquad
			j \geq 1,
		\end{align}
		and
		\begin{align} \label{eqn:rho_2_deriv_bds}
			|\rho_2^{(j)}(y)|
			\lesssim_j 1
			\lesssim y^{-j}
			\qquad \text{for all} \qquad
			j \geq 0.
		\end{align}
		Taylor expanding the logarithm in $\rho_3$ (which is valid because $Ty \geq 2$), we see that
		\begin{align} \label{eqn:rho_3_deriv_bds}
			|\rho_3^{(j)}(y)|
			\lesssim_j \Big(\frac{t}{Ty}\Big) y^{-j}
			\qquad \text{for all} \qquad
			j \geq 0.
		\end{align}
		Since $\phi(z)$ is a function of $z^2$, we have $\phi'(0) = 0$, so $|\phi'(z)| \lesssim |z| \lesssim y$. The higher derivatives are trivially bounded by $1$. Therefore
		\begin{align*}
			|\rho_4'(y)|
			\lesssim ty
			\qquad \text{and} \qquad
			|\rho_4^{(j)}(y)|
			\lesssim_j t
			\qquad \text{for all} \qquad
			j \geq 2.
		\end{align*}
		Since $t \leq T^{\frac{2}{3}}$ and $y \leq T^{-\frac{1}{3}}$,
		\begin{align*}
			|\rho_4'(y)|
			\lesssim T^{\frac{1}{3}}
			\leq y^{-1}
			\qquad \text{and} \qquad
			|\rho_4^{(j)}(y)|
			\lesssim_j T^{\frac{2}{3}}
			\leq y^{-2}
			\qquad \text{for all} \qquad
			j \geq 2.
		\end{align*}
		Thus
		\begin{align} \label{eqn:rho_4_deriv_bds}
			|\rho_4^{(j)}(y)|
			\lesssim_j y^{-j}
			\qquad \text{for all} \qquad
			j \geq 1.
		\end{align}
		Computing the derivatives of $\rho$ using \eqref{eqn:rho_factored} and the product and chain rules, and then applying \eqref{eqn:rho_1_deriv_bds}, \eqref{eqn:rho_2_deriv_bds}, \eqref{eqn:rho_3_deriv_bds}, and \eqref{eqn:rho_4_deriv_bds},
		\begin{align} \label{eqn:rho_product_chain_rule_bd}
			|\rho^{(j)}(y)|
			\lesssim_{p,j} \Big(\frac{t}{Ty}+1\Big)^j y^{-j} \exp(-\re[\rho_3(y) + \rho_4(y)])
			\qquad \text{for all} \qquad
			j \geq 0.
		\end{align}
		By Lemma~\ref{lem:t_exponential_bd} and the fact that $Ty \geq 2$ (in particular $Ty \gtrsim 1$),
		\begin{align*}
			\re[\rho_3(y)+\rho_4(y)]
			= \re(2it[\log(-iy^{-1}z) + \phi(z)])
			= \re(2it[\log(-iz) + \phi(z)])
			\geq \frac{ct}{Ty}
		\end{align*}
		for some $c \gtrsim 1$. Inserting this into \eqref{eqn:rho_product_chain_rule_bd},
		\begin{align*}
			|\rho^{(j)}(y)|
			\lesssim_{p,j} \Big(\frac{t}{Ty}+1\Big)^j y^{-j} \exp\Big(-\frac{ct}{Ty}\Big)
			\qquad \text{for all} \qquad
			j \geq 0.
		\end{align*}
		The product of the first and third terms is $\lesssim_j 1$, so we conclude
		\begin{align*}
			|\rho^{(j)}(y)|
			\lesssim_{p,j} y^{-j}
			\qquad \text{for all} \qquad
			j \geq 0.
		\end{align*}
		This is the desired estimate \eqref{eqn:rho_deriv_bds}.
	\end{proof}
	
	\section{Averaging the u-channel conformal block when $t>1$} \label{sec:u_averaging}
	
	In this section we prove Proposition~\ref{prop:I_u_principal}.
	Let $s = \frac{1}{2}+it$ with $t>1$ as in the statement of the proposition.
	
	\subsection{The rapidly decaying range: $t \ggg T$} 
	
	Suppose $t \geq T \log^2 T$. Then taking absolute values in the integral \eqref{eqn:final_W_def} defining $W(s)$ and applying Corollary~\ref{cor:u_block_rapid_decay} gives $|W(s)| \lesssim t^{-\infty}$.
	Thus Proposition~\ref{prop:I_u_principal} holds in the range $t \geq T\log^2T$.
	
	\subsection{The oscillatory range: $T \lesssim t \lessapprox T$}
	
	Now suppose $T \lesssim t \lessapprox T$.
	Then to prove Proposition~\ref{prop:I_u_principal}, we must show that
	\begin{align} \label{eqn:I_u_asymptotic_t_sim_T}
		W(s)
		= (1+\widetilde{O}(T^{-6\varepsilon})) \frac{1}{2} \Big(\frac{t}{T}\Big)^{4k-\frac{5}{2}} e^{-2t/T}  \exp\Big(-\Big(\frac{t-T}{H}\Big)^2\Big) + \widetilde{O}(H/T).
	\end{align}
	For $z$ as in \eqref{eqn:y_range}, Proposition~\ref{prop:u_block_asymptotic} (with $M=1$) tells us that there exists $J \geq 0$ depending only on $\varepsilon$, such that
	\begin{align*}
		\widetilde{H}_s(1-z^2)
		= \frac{t^{4k-\frac{5}{2}}}{2\sqrt{\pi z}} \Big[ e^{-2tz} \sum_{j=0}^{J} c_j (tz^3)^j + O(e^{-2t\re z} |tz|^{-1} + t^{-1}) \Big],
	\end{align*}
	where the $c_j$ are absolute constants with $c_0 = 1$.
	Noting that $e^{-2t\re z} \leq 1$ and $|z| \leq 1$, we see that the error term is $\lesssim |tz|^{-1}$.
	Therefore
	\begin{align} \label{eqn:I_u = sum I_u,j}
		W(s)
		= \sum_{j=0}^{J} W_j(t) + O(E(t)),
	\end{align}
	where
	\begin{align*}
		W_j(t)
		= \frac{c_j}{2\sqrt{\pi}} \Big(\frac{t}{T}\Big)^{4k-\frac{5}{2}} H t^j \int_{-T^{\varepsilon}/H}^{T^{\varepsilon}/H} e^{-(Hy)^2} e^{2iTy-2tz} z^{3j} \, dy
	\end{align*}
	and
	\begin{align*}
		E(t) = \Big(\frac{t}{T}\Big)^{4k-\frac{5}{2}} H \int_{-T^{\varepsilon}/H}^{T^{\varepsilon}/H} e^{-(Hy)^2} |tz|^{-1} \, dy.
	\end{align*}
	Since $t \lessapprox T$, we have $(t/T)^{4k-\frac{5}{2}} \lessapprox 1$.
	Since $t \gtrsim T$ and $z = T^{-1}+iy$, we also have $|tz|^{-1} \lesssim (T|y|+1)^{-1}$.
	It follows that up to a log factor, the dominant contribution to $E(t)$ comes from $|y| \lesssim T^{-1}$.
	Thus
	\begin{align} \label{eqn:K_bd}
		E(t)
		\lessapprox H/T.
	\end{align}
	
	It remains to estimate $W_j(t)$.
	We are only interested in $j \leq J$, and $J$ depends only on $\varepsilon$, so \textit{in the remainder of this subsection we suppress dependence of implicit constants on $j$.}
	Completing the integral defining $W_j(t)$ to an integral over $\R$,
	\begin{align} \label{eqn:I_u,j_approx_by_completion}
		W_j(t)
		= \widetilde{W}_j(t) + O(T^{-\infty}),
	\end{align}
	where
	\begin{align*}
		\widetilde{W}_j(t)
		= \frac{c_j}{2\sqrt{\pi}} \Big(\frac{t}{T}\Big)^{4k-\frac{5}{2}} H t^j \int_{\R} e^{-(Hy)^2} e^{2iTy-2tz} z^{3j} \, dy.
	\end{align*}
	Writing $z$ in terms of $y$,
	\begin{align*}
		\widetilde{W}_{j}(t)
		= \frac{c_j}{2\sqrt{\pi}} \Big(\frac{t}{T}\Big)^{4k-\frac{5}{2}} H t^j e^{-2t/T} \int_{\R} e^{-(Hy)^2} e^{2i(T-t)y} (T^{-1}+iy)^{3j} \, dy.
	\end{align*}
	Rescaling $y$ by $H$,
	\begin{align*}
		\widetilde{W}_{j}(t)
		= \frac{c_j}{2\sqrt{\pi}} \Big(\frac{t}{T}\Big)^{4k-\frac{5}{2}} t^j e^{-2t/T} \int_{\R} e^{-y^2} e^{2i(T-t)H^{-1}y} (T^{-1}+iH^{-1}y)^{3j} \, dy.
	\end{align*}
	Expanding the $3j$th power,
	\begin{align*}
		\widetilde{W}_{j}(t)
		= \frac{c_j}{2\sqrt{\pi}} \Big(\frac{t}{T}\Big)^{4k-\frac{5}{2}} t^j e^{-2t/T} \sum_{\ell+m=3j} \binom{3j}{\ell} H^{-\ell} T^{-m} \int_{\R} (iy)^{\ell} e^{-y^2} e^{2i(T-t)H^{-1}y} \, dy.
	\end{align*}
	Denote
	\begin{align} \label{eqn:u_g_def}
		g(\xi)
		= \int_{\R} e^{-y^2} e^{i\xi y} \, dy
		= \sqrt{\pi} e^{-\frac{1}{4}\xi^2}.
	\end{align}
	Then we can express the above as
	\begin{align} \label{eqn:expressed_via_g}
		\widetilde{W}_{j}(t)
		= \frac{c_j}{2\sqrt{\pi}} \Big(\frac{t}{T}\Big)^{4k-\frac{5}{2}} t^j e^{-2t/T} \sum_{\ell+m=3j} \binom{3j}{\ell} H^{-\ell} T^{-m} g^{(\ell)}\Big(\frac{2(T-t)}{H}\Big).
	\end{align}
	
	If $|t-T| \geq H\log T$, then $|\widetilde{W}_{j}(t)| \lesssim T^{-\infty}$ because of the decay of $g^{(\ell)}$, so $|W_j(s)| \lesssim T^{-\infty}$, and the desired bound \eqref{eqn:I_u_asymptotic_t_sim_T} follows from \eqref{eqn:I_u = sum I_u,j} and \eqref{eqn:K_bd}.
	
	Suppose instead that $|t-T| \lessapprox H$.
	Explicitly differentiating $g$ gives
	\begin{align} \label{eqn:g_deriv_ineq}
		|g^{(\ell)}(\xi)|
		\lessapprox_{\ell} g(\xi)
		\qquad \text{for} \qquad |\xi| \lessapprox 1.
	\end{align}
	Therefore
	\begin{align}
		|\widetilde{W}_{j}(t)|
		&\lessapprox T^j e^{-2t/T} \sum_{\ell+m=3j} H^{-\ell} T^{-m} g\Big(\frac{2(T-t)}{H}\Big)
		\notag
		\\&\lesssim T^j H^{-3j} e^{-2t/T} g\Big(\frac{2(T-t)}{H}\Big)
		\notag
		\\&= T^{-6\varepsilon j} e^{-2t/T} g\Big(\frac{2(T-t)}{H}\Big).
		\label{eqn:I_2,j_final_bd}
	\end{align}
	When $j=0$, we have $c_0 = 1$, so \eqref{eqn:expressed_via_g} specializes to
	\begin{align} \label{eqn:I_2,0_formula}
		\widetilde{W}_0(t)
		= \frac{1}{2\sqrt{\pi}} \Big(\frac{t}{T}\Big)^{4k-\frac{5}{2}} e^{-2t/T} g\Big(\frac{2(T-t)}{H}\Big).
	\end{align}
	Putting \eqref{eqn:I_2,j_final_bd} and \eqref{eqn:I_2,0_formula} together,
	\begin{align*}
		\sum_{j=0}^{J} \widetilde{W}_{j}(t)
		= (1+\widetilde{O}(T^{-6\varepsilon})) \frac{1}{2\sqrt{\pi}} \Big(\frac{t}{T}\Big)^{4k-\frac{5}{2}} e^{-2t/T}  g\Big(\frac{2(T-t)}{H}\Big).
	\end{align*}
	Inserting \eqref{eqn:I_u,j_approx_by_completion} and \eqref{eqn:u_g_def},
	\begin{align*}
		\sum_{j=0}^{J} W_{j}(t)
		= (1+\widetilde{O}(T^{-6\varepsilon})) \frac{1}{2} \Big(\frac{t}{T}\Big)^{4k-\frac{5}{2}} e^{-2t/T}  \exp\Big(-\Big(\frac{t-T}{H}\Big)^2\Big) + O(T^{-\infty}).
	\end{align*}
	Combined with \eqref{eqn:I_u = sum I_u,j} and \eqref{eqn:K_bd}, this implies \eqref{eqn:I_u_asymptotic_t_sim_T}.
	
	\subsection{The smooth range: $1<t \ll T$ and $|y| \ggg T^{-1}$}
	
	At this point, we have proved Proposition~\ref{prop:I_u_principal} whenever $t \gtrsim T$. In this subsection, suppose $1 < t \leq \frac{1}{2}T$.
	In this range, the hypothesis in Proposition~\ref{prop:u_block_asymptotic} that $\re z \gtrsim t^{-\frac{2}{3}+\varepsilon} |z|$ may not hold for all $z$ as in \eqref{eqn:y_range}.
	Therefore, we cannot use Proposition~\ref{prop:u_block_asymptotic}.
	Instead, we use that because $t \leq \frac{1}{2}T$, the function $y \mapsto \widetilde{H}_s(1-z^2)$ is smooth at scale $T^{-1}$, and so the integral defining $W(s)$ exhibits cancellation coming from the phase $e^{2iTy}$.
	We make this precise by an integration by parts argument.
	
	Fix $\eta > 0$ small. Then split
	\begin{align}
		T^{4k-\frac{5}{2}} H^{-1} |W(s)|
		&\leq \Big|\int_{|y| \lesssim T^{-1+\eta}} e^{-(Hy)^2} e^{2iTy} z^{\frac{1}{2}} \widetilde{H}_s(1-z^2) \, dy \Big|
		\notag
		\\&+ \Big|\int_{T^{-1+\eta} \lesssim |y| \leq T^{\varepsilon}/H} e^{-(Hy)^2} e^{2iTy} z^{\frac{1}{2}} \widetilde{H}_s(1-z^2) \, dy \Big|
		\label{eqn:I_2_split}
	\end{align}
	(with a smooth cutoff at $|y| \sim T^{-1+\eta}$).
	The first integral on the right is estimated in the next subsection.
	In this subsection we treat the second integral: we show by integration by parts that
	\begin{align} \label{eqn:u_smooth_goal}
		\Big|\int_{T^{-1+\eta} \lesssim |y| \leq T^{\varepsilon}/H} e^{-(Hy)^2} e^{2iTy} z^{\frac{1}{2}} \widetilde{H}_s(1-z^2) \, dy \Big|
		\lesssim_{\eta} T^{-\infty}.
	\end{align}
	By the exponential decay of the Gaussian in the integrand together with the trivial bound in Proposition~\ref{prop:u_block_triv_bd}, the sharp cutoff to $|y| \leq T^{\varepsilon}/H = T^{-\frac{1}{3}-\varepsilon}$ can be replaced by a smooth cutoff to $|y| \lesssim T^{-\frac{1}{3}}$.
	Then by dyadic decomposition, \eqref{eqn:u_smooth_goal} reduces to the estimate
	\begin{align} \label{eqn:u_smooth_after_dyadic}
		\Big|\int_{|y| \sim Y} e^{-(Hy)^2} e^{2iTy} z^{\frac{1}{2}} \widetilde{H}_s(1-z^2) \, dy \Big|
		\lesssim_{\eta} T^{-\infty}
	\end{align}
	for $T^{-1+\eta} \leq Y \leq T^{-\frac{1}{3}}$.
	
	\textit{For the remainder of this section, implicit constants may depend on $\eta$.}
	Let
	\begin{align*}
		F(z)
		= z^{\frac{1}{2}} \widetilde{H}_s(1-z^2),
		\qquad
		F_Y(\tilde{z})
		= F(Y\tilde{z}),
		\qquad
		F_{Y,T}(\tilde{z})
		= e^{2TY\tilde{z}} F_Y(\tilde{z}).
	\end{align*}
	Then the left hand side of \eqref{eqn:u_smooth_after_dyadic} is
	\begin{align*}
		\LHS\eqref{eqn:u_smooth_after_dyadic}
		= \Big|\int_{|y| \sim Y} e^{-(Hy)^2} e^{2iTy} F(z) \, dy \Big|.
	\end{align*}
	Substituting $y = Y\tilde{y}$,
	\begin{align} \label{eqn:u_smooth_after_substitution}
		\LHS\eqref{eqn:u_smooth_after_dyadic}
		= Y\Big|\int_{|\tilde{y}| \sim 1} e^{-(HY\tilde{y})^2} e^{2iTY\tilde{y}} F_Y(\tilde{z}) \, d\tilde{y} \Big|,
	\end{align}
	where we denote
	\begin{align*}
		\tilde{z}
		= z/Y
		= T^{-1}/Y + iy/Y
		= (TY)^{-1} + i\tilde{y}.
	\end{align*}
	The function $\tilde{y} \mapsto e^{-(HY\tilde{y})^2}$ obeys uniform derivative bounds for $|\tilde{y}| \sim 1$, i.e., the $j$th derivative at $\tilde{y}$ is $O_j(1)$ for each $j \geq 0$.
	Therefore, this function can be absorbed into the smooth cutoff in \eqref{eqn:u_smooth_after_substitution}, giving
	\begin{align*}
		\LHS\eqref{eqn:u_smooth_after_dyadic}
		\lesssim Y\Big|\int_{|\tilde{y}| \sim 1} e^{2iTY\tilde{y}} F_Y(\tilde{z}) \, d\tilde{y} \Big|.
	\end{align*}
	Since $Y \geq T^{-1+\eta}$, we have $(TY)^{-1} \leq T^{-\eta}$, so $|\tilde{y}| \sim 1$ if and only if $|\tilde{z}| \sim 1$.
	Thus, writing the above as an integral over $\tilde{z}$ rather than $\tilde{y}$,
	\begin{align*}
		\LHS\eqref{eqn:u_smooth_after_dyadic}
		&\lesssim Y\Big|\int_{(TY)^{-1}-i\infty}^{(TY)^{-1}+i\infty} \1_{|\tilde{z}| \sim 1} e^{2TY(\tilde{z}-(TY)^{-1})} F_Y(\tilde{z}) \, d\tilde{z} \Big|
		\\&\lesssim Y\Big|\int_{(TY)^{-1}-i\infty}^{(TY)^{-1}+i\infty} \1_{|\tilde{z}| \sim 1} F_{Y,T}(\tilde{z}) \, d\tilde{z} \Big|.
	\end{align*}
	For notational convenience, let $\L \subseteq \C$ denote the line $\re\tilde{z} = (TY)^{-1}$.
	Then the inequality above means that there exists $\chi \in C_c^{\infty}(\L)$ obeying uniform derivative bounds (see below) and with support contained in $\{|\tilde{z}| \sim 1\}$, such that
	\begin{align*}
		\LHS\eqref{eqn:u_smooth_after_dyadic}
		\lesssim Y\Big|\int_{\L} \chi(\tilde{z}) F_{Y,T}(\tilde{z}) \, d\tilde{z}\Big|.
	\end{align*}
	Here ``uniform derivative bounds" means
	\begin{align*}
		|\chi^{(j)}(\tilde{z})|
		\lesssim_j 1
		\qquad \text{for all} \qquad
		j \geq 0 \text{ and } \tilde{z} \in \L.
	\end{align*}
	Thus to prove \eqref{eqn:u_smooth_after_dyadic}, it suffices to show that
	\begin{align} \label{eqn:u_smooth_goal_2}
		\Big|\int_{\L} \chi(\tilde{z}) F_{Y,T}(\tilde{z}) \, d\tilde{z}\Big|
		\lesssim T^{-\infty}.
	\end{align}
	
	The first step toward proving \eqref{eqn:u_smooth_goal_2} is to show that $F_{Y,T}$ obeys an ODE which can be analyzed explicitly.
	The hypergeometric ODE is
	\begin{align*}
		w(1-w) H_s''(w)
		+ (1-2w) H_s'(w)
		- \Big(\frac{1}{4}+t^2\Big) H_s(w)
		= 0.
	\end{align*}
	This translates to the ODE
	\begin{align*}
		F''(z)
		- \frac{2z}{1-z^2} F'(z)
		- \frac{\frac{1}{4} + 4t^2 - \frac{1}{4}z^{-2}}{1-z^2} F(z)
		= 0
	\end{align*}
	for $F$, which in turn yields the ODE
	\begin{align*}
		F_Y''(\tilde{z})
		- \frac{2Y^2\tilde{z}}{1-Y^2\tilde{z}^2} F_Y'(\tilde{z})
		- \frac{\frac{1}{4}Y^2 + 4(tY)^2 - \frac{1}{4}\tilde{z}^{-2}}{1-Y^2\tilde{z}^2} F_Y(\tilde{z})
		= 0
	\end{align*}
	for $F_Y$.
	In other words, $F_Y$ is killed by the differential operator
	\begin{align*}
		\P_Y
		= \partial_{\tilde{z}}^2
		- \frac{2Y^2\tilde{z}}{1-Y^2\tilde{z}^2} \partial_{\tilde{z}}
		- \frac{\frac{1}{4}Y^2 + 4(tY)^2 - \frac{1}{4}\tilde{z}^{-2}}{1-Y^2\tilde{z}^2}.
	\end{align*}
	Let
	\begin{align*}
		\P_{Y,T}
		= e^{2TY\tilde{z}} \circ \P_Y \circ e^{-2TY\tilde{z}},
		\qquad \text{so} \qquad
		\P_{Y,T}F_{Y,T}
		= 0.
	\end{align*}
	Expanding $\P_{Y,T}$ by the product rule,
	\begin{align*}
		\P_{Y,T}
		= \partial_{\tilde{z}}^2
		- \Big(\frac{2Y^2\tilde{z}}{1-Y^2\tilde{z}^2} + 4TY\Big) \partial_{\tilde{z}}
		+ 4T^2Y^2
		+ \frac{4TY^3\tilde{z}}{1-Y^2\tilde{z}^2}
		- \frac{\frac{1}{4}Y^2 + 4(tY)^2 - \frac{1}{4}\tilde{z}^{-2}}{1-Y^2\tilde{z}^2}.
	\end{align*}
	The formal adjoint of $\P_{Y,T}$ (with respect to the symmetric pairing on $L^2(\L)$ rather than the Hermitian pairing) is
	\begin{align*}
		\P_{Y,T}^*
		&= \partial_{\tilde{z}}^2
		+ \partial_{\tilde{z}} \circ \Big(\frac{2Y^2\tilde{z}}{1-Y^2\tilde{z}^2} + 4TY\Big)
		+ 4T^2Y^2
		+ \frac{4TY^3\tilde{z}}{1-Y^2\tilde{z}^2}
		- \frac{\frac{1}{4}Y^2 + 4(tY)^2 - \frac{1}{4}\tilde{z}^{-2}}{1-Y^2\tilde{z}^2}
		\\&= \partial_{\tilde{z}}^2
		+ \Big(\frac{2Y^2\tilde{z}}{1-Y^2\tilde{z}^2} + 4TY\Big) \partial_{\tilde{z}}
		+ \frac{2Y^2 + 2Y^4\tilde{z}^2}{(1-Y^2\tilde{z}^2)^2}
		+ 4T^2Y^2
		+ \frac{4TY^3\tilde{z}}{1-Y^2\tilde{z}^2}
		- \frac{\frac{1}{4}Y^2 + 4(tY)^2 - \frac{1}{4}\tilde{z}^{-2}}{1-Y^2\tilde{z}^2}.
	\end{align*}
	Write this as
	\begin{align} \label{eqn:adjoint_phi_psi}
		\P_{Y,T}^*
		= \partial_{\tilde{z}}^2
		+ TY \varphi_{Y,T}(\tilde{z}) \partial_{\tilde{z}}
		+ (TY)^2 \psi_{Y,T}(\tilde{z}),
	\end{align}
	where
	\begin{align*}
		\varphi_{Y,T}(\tilde{z})
		= 4 + \frac{2T^{-1}Y\tilde{z}}{1-Y^2\tilde{z}^2}
	\end{align*}
	and
	\begin{align} \label{eqn:psi_Y,T_def}
		\psi_{Y,T}(\tilde{z})
		= 4
		+ \frac{2T^{-2}(1 + Y^2\tilde{z}^2)}{(1-Y^2\tilde{z}^2)^2}
		+ \frac{4T^{-1}Y\tilde{z}}{1-Y^2\tilde{z}^2}
		- \frac{\frac{1}{4}T^{-2} + 4(t/T)^2 - \frac{1}{4}(TY)^{-2}\tilde{z}^{-2}}{1-Y^2\tilde{z}^2}.
	\end{align}
	Recalling the inequalities $t \leq \frac{1}{2}T$ and $T^{-1+\eta} \leq Y \leq T^{-\frac{1}{3}}$, we see that when $|\tilde{z}| \sim 1$,
	\begin{align*}
		|\varphi_{Y,T}(\tilde{z})| \lesssim 1
		\qquad \text{and} \qquad
		|\psi_{Y,T}(\tilde{z})| \lesssim 1.
	\end{align*}
	It follows from Cauchy's estimates for derivatives of holomorphic functions that when $|\tilde{z}| \sim 1$,
	\begin{align} \label{eqn:phi_psi_deriv_bds}
		|\varphi_{Y,T}^{(j)}(\tilde{z})| \lesssim_j 1
		\qquad \text{and} \qquad
		|\psi_{Y,T}^{(j)}(\tilde{z})| \lesssim_j 1
		\qquad \text{for all} \qquad
		j \geq 0.
	\end{align}
	Observe furthermore that when $|\tilde{z}| \sim 1$,
	\begin{align} \label{eqn:psi_bdd_below}
		|\psi_{Y,T}(\tilde{z})|
		\sim 1;
	\end{align}
	indeed, of the four terms in \eqref{eqn:psi_Y,T_def}, the second and third have absolute value $\ll 1$, and the fourth has absolute value $\leq 4(\frac{1}{2})^2 + o(1) = 1+o(1)$ because $t \leq \frac{1}{2}T$, so the last three terms put together are still too small to cancel out the first term, which is $4$.
	From \eqref{eqn:phi_psi_deriv_bds} and \eqref{eqn:psi_bdd_below}, it follows that when $|\tilde{z}| \sim 1$,
	\begin{align*}
		\Big|\Big(\frac{1}{\psi_{Y,T}}\Big)^{(j)}(\tilde{z})\Big| \lesssim_j 1
		\qquad \text{for all} \qquad j \geq 0.
	\end{align*}
	In summary, $\varphi_{Y,T}$, $\psi_{Y,T}$, and $1/\psi_{Y,T}$ all obey uniform derivative bounds on $\{|\tilde{z}| \sim 1\}$.
	
	Now, the reason we are interested in $\P_{Y,T}^*$ is that $F_{Y,T}$ is orthogonal to the image of $\P_{Y,T}^*$ acting on $C_c^{\infty}(\L)$: we have
	\begin{align*}
		\int_{\L} \P_{Y,T}^*(g) F_{Y,T}
		= \int_{\L} g \P_{Y,T} F_{Y,T}
		= 0
	\end{align*}
	for all $g \in C_c^{\infty}(\L)$.
	Thus in \eqref{eqn:u_smooth_goal_2}, we can modify $\chi$ by any element of $\P_{Y,T}^* C_c^{\infty}(\L)$ without changing the value of the integral.
	
	\begin{lem} \label{lem:congruent_replacement}
		Let $\chi_0 \in C_c^{\infty}(\L)$, with support contained in $\{|\tilde{z}| \sim 1\}$, obey the uniform derivative bounds $\|\chi_0^{(j)}\|_{L^{\infty}(\L)} \lesssim_j 1$ for all $j \geq 0$. Then for any $n \geq 0$, there exists $\chi_n \in C_c^{\infty}(\L)$ with support contained in the support of $\chi_0$, obeying the uniform derivative bounds $\|\chi_n^{(j)}\|_{L^{\infty}(\L)} \lesssim_{j,n} 1$ for all $j \geq 0$, such that 
		\begin{align*}
			\chi_0 \equiv (TY)^{-n} \chi_n
			\mod \P_{Y,T}^* C_c^{\infty}(\L).
		\end{align*}
		This congruence condition means that there exists $g \in C_c^{\infty}(\L)$ such that
		\begin{align} \label{eqn:congruence_def}
			\P_{Y,T}^*g = \chi_0 - (TY)^{-n} \chi_n.
		\end{align}
	\end{lem}
	
	\begin{proof}
		It suffices to prove this for $n=1$, because the general case can be obtained from the $n=1$ case by iteration.
		Let
		\begin{align*}
			g = \frac{\chi_0}{(TY)^2 \psi_{Y,T}}.
		\end{align*}
		Then define $\chi_1$ to make \eqref{eqn:congruence_def} hold with $n=1$, so take
		\begin{align*}
			\chi_1 = TY(\chi_0 - \P_{Y,T}^*g).
		\end{align*}
		Clearly both $g$ and $\chi_1$ have support contained in that of $\chi_0$.
		By \eqref{eqn:adjoint_phi_psi},
		\begin{align*}
			\P_{Y,T}^*g
			= \frac{1}{(TY)^2} \Big(\frac{\chi_0}{\psi_{Y,T}}\Big)''
			+ \frac{1}{TY} \varphi_{Y,T} \Big(\frac{\chi_0}{\psi_{Y,T}}\Big)'
			+ \chi_0.
		\end{align*}
		Plugging this into the definition of $\chi_1$ above,
		\begin{align*}
			\chi_1 = -\frac{1}{TY} \Big(\frac{\chi_0}{\psi_{Y,T}}\Big)'' - \varphi_{Y,T} \Big(\frac{\chi_0}{\psi_{Y,T}}\Big)'.
		\end{align*}
		Since $\chi_0$, $\varphi_{Y,T}$, and $1/\psi_{Y,T}$ all obey uniform derivative bounds, so does $\chi_1$ (recall $TY \geq T^{\eta} \geq 1$).
	\end{proof}
	
	Now, let $n \geq 0$ be arbitrarily large, and apply Lemma~\ref{lem:congruent_replacement} with $\chi_0 = \chi$, where $\chi$ is as in \eqref{eqn:u_smooth_goal_2}.
	Then we get a function $\chi_n \in C_c^{\infty}(\L)$ obeying derivative bounds depending only on $n$, such that the support of $\chi_n$ is contained in that of $\chi$, and such that $\chi \equiv (TY)^{-n} \chi_n \mod \P_{Y,T}^* C_c^{\infty}(\L)$.
	Using this congruence and then the fact that $TY \geq T^{\eta}$,
	\begin{align*}
		\LHS\eqref{eqn:u_smooth_goal_2}
		= (TY)^{-n} \Big|\int_{\L} \chi_n(\tilde{z}) F_{Y,T}(\tilde{z}) \, d\tilde{z}\Big|
		\lesssim_n T^{-n\eta} \|F_{Y,T}\|_{L^{\infty}(\L \cap \{|\tilde{z}| \sim 1\})}.
	\end{align*}
	By Proposition~\ref{prop:u_block_triv_bd}, the $L^{\infty}$ norm on the right hand side is
	$\lesssim T^{O(1)}$.
	Since $n$ was arbitrary, we conclude \eqref{eqn:u_smooth_goal_2}.
	Above, we reduced \eqref{eqn:u_smooth_goal} to \eqref{eqn:u_smooth_goal_2}.
	Thus \eqref{eqn:u_smooth_goal} holds, and we have accomplished our goal for this subsection.
	
	\subsection{The trivially bounded range: $1 < t \lesssim T$ and $|y| \lessapprox T^{-1}$}
	
	Let $1 < t \leq \frac{1}{2}T$ as above. In the previous subsection we showed that the second term in \eqref{eqn:I_2_split} is $\lesssim T^{-\infty}$.
	It remains to estimate the first term, which is
	\begin{align} \label{eqn:first_term_in_I_2_split}
		\Big|\int_{|y| \lesssim T^{-1+\eta}} e^{-(Hy)^2} e^{2iTy} z^{\frac{1}{2}} \widetilde{H}_s(1-z^2) \, dy \Big|.
	\end{align}
	Taking the absolute values inside,
	\begin{align*}
		\eqref{eqn:first_term_in_I_2_split}
		\lesssim T^{-\frac{1}{2} + \frac{\eta}{2}} \int_{|y| \lesssim T^{-1+\eta}} |\widetilde{H}_s(1-z^2)| \, dy.
	\end{align*}
	Applying Proposition~\ref{prop:u_block_triv_bd},
	\begin{align*}
		\eqref{eqn:first_term_in_I_2_split}
		\lessapprox T^{-\frac{1}{2}+\frac{\eta}{2}} T^{-1+\eta} t^{4k-2}
		= T^{-\frac{3}{2} + \frac{3}{2}\eta} t^{4k-2}.
	\end{align*}
	Since $t \lesssim T$, it follows that
	\begin{align*}
		\eqref{eqn:first_term_in_I_2_split}
		\lessapprox T^{4k - \frac{7}{2} + \frac{3}{2}\eta}.
	\end{align*}
	Inserting this and \eqref{eqn:u_smooth_goal} into \eqref{eqn:I_2_split}, we obtain
	\begin{align*}
		T^{4k-\frac{5}{2}} H^{-1} |W(s)|
		\lessapprox T^{4k-\frac{7}{2}+\frac{3}{2}\eta}
	\end{align*}
	(recall that implicit constants may depend on $\eta$).
	Since $\eta$ was arbitrary, we conclude that
	\begin{align*}
		|W(s)| \lessapprox H/T.
	\end{align*}
	Since $t \leq \frac{1}{2}T$, Proposition~\ref{prop:I_u_principal} follows.
	
	We have now proved Proposition~\ref{prop:I_u_principal} in all cases.
	
	\section{Averaging the t-channel conformal block when $t>1$} \label{sec:t_averaging}
	
	In this section we prove Proposition~\ref{prop:I_t_principal}. As above, let $s = \frac{1}{2}+it$ with $t>1$. Recall that the definition \eqref{eqn:final_W_check_def} of $\widecheck{W}(s)$ is
	\begin{align*}
		\widecheck{W}(s)
		= T^{-4k+\frac{5}{2}} H \int_{-T^{\varepsilon}/H}^{T^{\varepsilon}/H} I(y) \, dy,
	\end{align*}
	where the integrand $I(y)$ is
	\begin{align} \label{eqn:t_integrand_def}
		I(y)
		= e^{-(Hy)^2} e^{2iTy} z^{-4k+\frac{1}{2}} \widetilde{H}_s(1-z^{-2})
	\end{align}
	with $z = T^{-1} + iy$ as usual.
	We have $I(-y) = \overline{I(y)}$, so
	\begin{align*}
		\widecheck{W}(s)
		= 2T^{-4k+\frac{5}{2}} H \re \int_{0}^{T^{\varepsilon}/H} I(y) \, dy.
	\end{align*}
	Thus to prove Proposition~\ref{prop:I_t_principal}, it suffices to show that
	\begin{align} \label{eqn:half_I_4_bd}
		\int_{0}^{T^{\varepsilon}/H} I(y) \, dy
		= \frac{1}{2} e^{-2} T^{4k-\frac{5}{2}} t^{-\frac{1}{2}} \exp\Big(-\Big(\frac{t}{T/H}\Big)^2\Big) e^{i\alpha(t)}
		+ \widetilde{O}(\1_{t \lessapprox T/H} T^{4k-\frac{5}{2}} t^{-1} + t^{-\infty}),
	\end{align}
	where $\alpha(t)$ is as in the proposition.
	This is our target for the remainder of the section.
	
	\subsection{The rapidly decaying range: $t \ggg 1$ and $y \lll \frac{t}{T}$}
	
	In this subsection and the next, suppose $t \geq \log^3 T$. Denote
	\begin{align*}
		Y_0
		= \frac{t}{T\log^2 T}.
	\end{align*}
	Then split
	\begin{align} \label{eqn:half_I_4_split}
		\LHS\eqref{eqn:half_I_4_bd}
		= \int_{y \lesssim Y_0} \1_{y \leq T^{\varepsilon}/H} I(y) \, dy
		+ \int_{Y_0 \lesssim y \leq T^{\varepsilon}/H} I(y) \, dy
	\end{align}
	(with a smooth cutoff at $y \sim Y_0$).
	We will estimate the first integral in this subsection and the second integral in the next.
	
	By Corollary~\ref{cor:t-block_triv_bd},
	\begin{align*}
		|\widetilde{H}_s(1-z^{-2})|
		\lesssim t^{4k-\frac{5}{2}} |z| \exp\Big(-\frac{ct}{Ty+1}\Big),
	\end{align*}
	where $c$ is a constant $\gtrsim 1$.
	When $y \lesssim Y_0$ and $y \leq T^{\varepsilon}/H$, the exponent satisfies
	\begin{align*}
		\frac{t}{Ty+1}
		\gtrsim \max\Big\{\frac{t}{TY_0}, \frac{t}{T^{1+\varepsilon}/H}\Big\}
		= \max\Big\{\log^2 T, \frac{t}{T^{\frac{2}{3}-\varepsilon}}\Big\}.
	\end{align*}
	It follows that for $y$ as in the first integral in $\RHS\eqref{eqn:half_I_4_split}$,
	\begin{align*}
		|\widetilde{H}_s(1-z^{-2})|
		\lesssim T^{-\infty} t^{-\infty}.
	\end{align*}
	Thus the whole first integral is $O(T^{-\infty} t^{-\infty})$, and from \eqref{eqn:half_I_4_split} we obtain
	\begin{align} \label{eqn:half_I_4_oscillatory_+_error}
		\LHS\eqref{eqn:half_I_4_bd}
		= \int_{Y_0 \lesssim y \leq T^{\varepsilon}/H} I(y) \, dy
		+ O(T^{-\infty} t^{-\infty}).
	\end{align}
	
	\subsection{The oscillatory range: $t \ggg 1$ and $y \gtrapprox \frac{t}{T}$} \label{subsec:t_oscillatory_range}
	
	As above, suppose throughout this subsection that $t \geq \log^3 T$.
	
	If $t \geq T^{\frac{2}{3}}$, then $Y_0 \gg T^{\varepsilon}/H$, so the integral on the right hand side of \eqref{eqn:half_I_4_oscillatory_+_error} vanishes. It follows that \eqref{eqn:half_I_4_bd} holds.
	
	If $(T/H) \log^3T \leq t < T^{\frac{2}{3}}$, then $Y_0 \geq H^{-1} \log T$, so $\RHS\eqref{eqn:half_I_4_oscillatory_+_error} \lesssim T^{-\infty}$ because of the Gaussian factor in \eqref{eqn:t_integrand_def} and because of Corollary~\ref{cor:t-block_triv_bd}. Therefore \eqref{eqn:half_I_4_bd} again holds.
	
	Thus we may assume that $t \lessapprox T/H$.
	Then it suffices to show that the integral on the right hand side of \eqref{eqn:half_I_4_oscillatory_+_error} satisfies
	\begin{align} \label{eqn:t_oscillatory_goal}
		\int_{Y_0 \lesssim y \leq T^{\varepsilon}/H} I(y) \, dy
		= \frac{1}{2} e^{-2} T^{4k-\frac{5}{2}} t^{-\frac{1}{2}} \exp\Big(-\Big(\frac{t}{T/H}\Big)^2\Big) e^{i\alpha(t)}
		+ \widetilde{O}(T^{4k-\frac{5}{2}} t^{-1}).
	\end{align}
	Applying Proposition~\ref{prop:t_block_asymptotic_separated} with $p = -4k+\frac{1}{2}$,
	\begin{align*}
		I(y)
		= t^{4k-\frac{5}{2}} e^{-(Hy)^2} y^{-4k+\frac{3}{2}} \rho(y) e^{2i(Ty - t\log y)} + O(t^{4k-\frac{7}{2}} y^{-4k+\frac{3}{2}}),
	\end{align*}
	where $\rho$ is the smooth function \eqref{eqn:rho_def} on $[2T^{-1}, T^{-\frac{1}{3}}]$ obeying the derivative bounds
	\begin{align} \label{eqn:rho_applied_deriv_bds}
		|\rho^{(j)}(y)|
		\lesssim_j y^{-j}
		\qquad \text{for all} \qquad
		j \geq 0.
	\end{align}
	Inserting this into $\LHS\eqref{eqn:t_oscillatory_goal}$,
	\begin{align*}
		\LHS\eqref{eqn:t_oscillatory_goal}
		= t^{4k-\frac{5}{2}} \int_{Y_0 \lesssim y \leq T^{\varepsilon}/H} e^{-(Hy)^2} y^{-4k+\frac{3}{2}} \rho(y) e^{2i(Ty - t\log y)} \, dy
		+ O\Big(t^{4k-\frac{7}{2}} \int_{Y_0 \lesssim y \leq T^{\varepsilon}/H} y^{-4k+\frac{3}{2}} \, dy\Big).
	\end{align*}
	The error term is
	\begin{align*}
		\lesssim t^{4k-\frac{7}{2}} Y_0^{-4k+\frac{5}{2}}
		\lessapprox t^{4k-\frac{7}{2}} \Big(\frac{t}{T}\Big)^{-4k+\frac{5}{2}}
		= T^{4k-\frac{5}{2}} t^{-1},
	\end{align*}
	matching the error term in our goal \eqref{eqn:t_oscillatory_goal}.
	
	It remains to check that the main term satisfies
	\begin{align}
		t^{4k-\frac{5}{2}} \int_{Y_0 \lesssim y \leq T^{\varepsilon}/H} e^{-(Hy)^2} y^{-4k+\frac{3}{2}} \rho(y) e^{2i(Ty - t\log y)} \, dy
		&= \frac{1}{2} e^{-2} T^{4k-\frac{5}{2}} t^{-\frac{1}{2}} \exp\Big(-\Big(\frac{t}{T/H}\Big)^2\Big) e^{i\alpha(t)}
		\notag
		\\&+ \widetilde{O}(T^{4k-\frac{5}{2}} t^{-1}).
		\label{eqn:t_stat_phase_goal}
	\end{align}
	Rescaling $y$ by $T/t$, and noting that $TY_0/t = \log^{-2}T$,
	\begin{align*}
		\LHS\eqref{eqn:t_stat_phase_goal}
		= t^{4k-\frac{5}{2}} \frac{t}{T}
		\int_{\log^{-2}T \lesssim y \leq t^{-1} T^{1+\varepsilon}/H} e^{-(Hty/T)^2} (ty/T)^{-4k+\frac{3}{2}} \rho(ty/T) e^{2i(ty - t\log(ty/T))} \, dy.
	\end{align*}
	Simplifying,
	\begin{align} \label{eqn:t_before_stat_phase}
		\LHS\eqref{eqn:t_stat_phase_goal}
		= T^{4k-\frac{5}{2}}
		e^{-2it \log(t/T)} \int_{\log^{-2}T \lesssim y \leq t^{-1} T^{1+\varepsilon}/H} g(y) e^{2it(y - \log y)} \, dy,
	\end{align}
	where
	\begin{align} \label{eqn:t_g_def}
		g(y) = e^{-(Hty/T)^2} y^{-4k+\frac{3}{2}} \rho(ty/T).
	\end{align}
	For $y \gtrapprox 1$, the $j$th derivative of the Gaussian factor in $g(y)$ is $\lessapprox_j 1$. It follows from this and from the derivative bounds \eqref{eqn:rho_applied_deriv_bds} on $\rho$ that
	\begin{align*}
		|g^{(j)}(y)|
		\lessapprox_j y^{-4k+\frac{3}{2}}
	\end{align*}
	for all $j \geq 0$ and $y$ as in \eqref{eqn:t_before_stat_phase}.
	Thus by \eqref{eqn:t_before_stat_phase} and stationary phase,
	\begin{align*}
		\LHS\eqref{eqn:t_stat_phase_goal}
		= T^{4k-\frac{5}{2}}
		e^{-2it \log(t/T)} \Big[e^{i\pi/4} \sqrt{\pi} g(1) e^{2it} t^{-\frac{1}{2}} + \widetilde{O}(t^{-\frac{3}{2}})\Big],
	\end{align*}
	where the main term in brackets comes from the stationary point $y=1$.
	Expanding out,
	\begin{align} \label{eqn:t_stat_phase_result}
		\LHS\eqref{eqn:t_stat_phase_goal}
		= \sqrt{\pi} g(1) T^{4k-\frac{5}{2}} t^{-\frac{1}{2}} e^{i\pi/4 + 2it(1 - \log(t/T))}
		+ \widetilde{O}(T^{4k-\frac{5}{2}} t^{-\frac{3}{2}}).
	\end{align}
	By the definitions \eqref{eqn:t_g_def} and \eqref{eqn:rho_def} of $g$ and $\rho$ (recalling that $p = -4k+\frac{1}{2}$ in the definition of $\rho$),
	\begin{align*}
		g(1)
		&= e^{-(tH/T)^2} \rho(t/T)
		\\&= \exp\Big(-\Big(\frac{t}{T/H}\Big)^2\Big) \Big(\frac{1}{t}+i\Big)^{-4k+\frac{3}{2}} \psi(z) \exp\Big(-2it\Big[\log\Big(1-\frac{i}{t}\Big) + \phi(z)\Big]\Big),
	\end{align*}
	where $\phi,\psi$ are as in Proposition~\ref{prop:t_block_asymptotic}, and where $z = T^{-1} + i\frac{t}{T}$.
	Making the approximations
	\begin{align*}
		\Big(\frac{1}{t}+i\Big)^{-4k+\frac{3}{2}}
		= e^{3i\pi/4} (1 + O(t^{-1}))
	\end{align*}
	and
	\begin{align*}
		\psi(z) = \frac{e^{-i\pi/4}}{2\sqrt{\pi}}(1 + O(|z|^2))
	\end{align*}
	and
	\begin{align*}
		\log\Big(1-\frac{i}{t}\Big)
		= -\frac{i}{t} + O(t^{-2})
	\end{align*}
	and
	\begin{align*}
		\phi(z)
		= -\log(2) + \frac{1}{4}z^2 + O(|z|^4),
	\end{align*}
	we get
	\begin{align*}
		g(1)
		= \frac{1}{2\sqrt{\pi}} \exp\Big(-\Big(\frac{t}{T/H}\Big)^2\Big) (1 + O(t^{-1} + |z|^2)) \exp\Big(\frac{i\pi}{2} - 2it\Big[-\frac{i}{t} - \log(2) + \frac{1}{4}z^2 + O(t^{-2}+|z|^4)\Big]\Big).
	\end{align*}
	This simplifies to
	\begin{align} \label{eqn:g(1)_approx_1}
		g(1)
		= \frac{e^{-2}}{2\sqrt{\pi}} \exp\Big(-\Big(\frac{t}{T/H}\Big)^2\Big) \exp\Big(\frac{i\pi}{2} + 2it\Big[\log(2) - \frac{1}{4}z^2\Big]\Big)
		+ O(t^{-1} + |z|^2 + t|z|^4).
	\end{align}
	Since $1 < t \leq T^{\frac{2}{3}}$ (the upper bound coming from $t \lessapprox T/H$), we have
	\begin{align*}
		z^2 = \Big(T^{-1} + i\frac{t}{T}\Big)^2
		= \Big(\frac{t}{T}\Big)^2 (t^{-1}+i)^2
		= -\Big(\frac{t}{T}\Big)^2(1+O(t^{-1}))
		= -\Big(\frac{t}{T}\Big)^2 + O(T^{-\frac{4}{3}})
		= -\Big(\frac{t}{T}\Big)^2 + O(t^{-2}).
	\end{align*}
	In particular,
	\begin{align*}
		|z|^2
		\lesssim \Big(\frac{t}{T}\Big)^2 + t^{-2}
		\leq T^{-\frac{2}{3}} + t^{-2}
		\lesssim t^{-1}.
	\end{align*}
	Plugging these two estimates for $z^2$ into \eqref{eqn:g(1)_approx_1},
	\begin{align*}
		g(1)
		= \frac{e^{-2}}{2\sqrt{\pi}} \exp\Big(-\Big(\frac{t}{T/H}\Big)^2\Big) \exp\Big(\frac{i\pi}{2} + 2it\Big[\log(2) + \frac{1}{4}\Big(\frac{t}{T}\Big)^2\Big]\Big)
		+ O(t^{-1}).
	\end{align*}
	Inserting this into \eqref{eqn:t_stat_phase_result} yields
	\begin{align*}
		\LHS\eqref{eqn:t_stat_phase_goal}
		=
		\frac{1}{2} e^{-2} T^{4k-\frac{5}{2}} t^{-\frac{1}{2}} \exp\Big(-\Big(\frac{t}{T/H}\Big)^2\Big) e^{i\alpha(t)}
		+ \widetilde{O}(T^{4k-\frac{5}{2}} t^{-\frac{3}{2}}),
	\end{align*}
	where $\alpha(t)$ is given by \eqref{eqn:alpha_def}. This is the desired estimate \eqref{eqn:t_stat_phase_goal}, in fact with a better error term.
	
	Above, we reduced \eqref{eqn:half_I_4_bd} to \eqref{eqn:t_stat_phase_goal} under the assumption $t \geq \log^3 T$. We conclude that \eqref{eqn:half_I_4_bd} holds for $t \geq \log^3T$.
	
	\subsection{The trivially bounded range: $1 < t \lessapprox 1$}
	
	It remains to prove \eqref{eqn:half_I_4_bd} for $1 < t < \log^3 T$. More generally, for $1<t \lessapprox 1$, the error term in $\RHS\eqref{eqn:half_I_4_bd}$ dominates the main term, so \eqref{eqn:half_I_4_bd} reduces to
	\begin{align*}
		\Big|\int_{0}^{T^{\varepsilon}/H} I(y) \, dy\Big|
		\lessapprox T^{4k-\frac{5}{2}}.
	\end{align*}
	Writing out the definition \eqref{eqn:t_integrand_def} of $I(y)$, and using Corollary~\ref{cor:t-block_triv_bd} and the fact that $t \lessapprox 1$,
	\begin{align*}
		|I(y)|
		= e^{-(Hy)^2} |z|^{-4k+\frac{1}{2}} |\widetilde{H}_s(1-z^{-2})|
		\lessapprox |z|^{-4k+\frac{3}{2}},
	\end{align*}
	so
	\begin{align*}
		\Big|\int_{0}^{T^{\varepsilon}/H} I(y) \, dy\Big|
		\lessapprox \int_{\R} |z|^{-4k+\frac{3}{2}} \, dy
		\lesssim T^{4k-\frac{5}{2}},
	\end{align*}
	as desired.
	Thus \eqref{eqn:half_I_4_bd} holds in all cases, and the proof of Proposition~\ref{prop:I_t_principal} is complete.
	
	\bibliography{references}
	
	\parskip = 0em
	
\end{document}